\documentclass[12pt,reqno]{amsart}
\usepackage[margin=1in]{geometry}
\usepackage{amsmath,amssymb,amsthm,graphicx,amsxtra, setspace}
\usepackage[utf8]{inputenc}
\usepackage{mathrsfs}
\usepackage{hyperref}
\usepackage{upgreek}
\usepackage{mathtools}
\allowdisplaybreaks

\newtheorem{theorem}{Theorem}[section]
\newtheorem{lemma}[theorem]{Lemma}
\newtheorem{proposition}[theorem]{Proposition}

\newtheorem{definition}[theorem]{Definition}

\newtheorem{remark}[theorem]{Remark}

\let\originalleft\left
\let\originalright\right
\renewcommand{\left}{\mathopen{}\mathclose\bgroup\originalleft}
\renewcommand{\right}{\aftergroup\egroup\originalright}

\newcommand{\Tr}{\mathop{\mathrm{Tr}}}
\renewcommand{\d}{\/\mathrm{d}\/}

\def\w{\textbf{W}^{\varepsilon}_{{\theta}^{\varepsilon}}}

\def\e{\varepsilon}

\def\S{\mathrm{S}}
\def\T{\mathrm{T}}
\def\L{\mathbb{L}}
\def\A{\mathrm{A}}
\def\I{\mathrm{I}}

\def\C{\mathrm{C}}
\def\f{\mathbf{f}}

\def\B{\mathrm{B}}
\def\D{\mathrm{D}}
\def\y{\mathbf{y}}
\def\q{\mathbf{q}}

\def\X{\mathbb{X}}
\def\x{\mathbf{x}}

\def\p{\mathbf{p}}

\def\z{\mathbf{z}}
\def\v{\mathbf{v}}
\def\V{\mathbb{V}}
\def\w{\mathbf{w}}
\def\W{\mathrm{W}}
\def\G{\mathrm{G}}
\def\Q{\mathrm{Q}}

\def\N{\mathbb{N}}

\def\V{\mathbb{V}}
\def\wi{\widetilde}
\def\Q{\mathrm{Q}}

\def\P{\mathrm{P}}
\def\u{\mathbf{u}}
\def\H{\mathbb{H}}
\def\n{\mathbf{n}}

\newcommand{\R}{\mathbb{R}}

\renewcommand{\d}{\/\mathrm{d}\/}

\newcommand{\Addresses}{{% additional braces for segregating \footnotesize
		\footnote{
			\noindent \textsuperscript{1}Department of Mathematics, Indian Institute of Technology Roorkee-IIT Roorkee,
			Haridwar Highway, Roorkee, Uttarakhand 247667, INDIA.\par\nopagebreak
			\noindent  \textit{e-mail:} \texttt{maniltmohan@iitr.ac.in, maniltmohan@gmail.com.}
			
			\noindent \textsuperscript{*}Corresponding author.

			\textit{Key words:} convective Brinkman-Forchheimer equations, Global attractors,  Exponential Attractors, Quasi-stability, Fractal and Hausdroff dimensions, Upper semicontinuity.
			
			Mathematics Subject Classification (2010): 37L30, 35Q35, 35Q30, 35B40.

}}}

\begin{document}
	
	\title[Global and  Exponential Attractors for CBF Equations]{Asymptotic analysis of the 2D convective Brinkman-Forchheimer equations in unbounded domains: Global attractors and upper semicontinuity
		\Addresses}
\author[M. T. Mohan ]{Manil T. Mohan\textsuperscript{1*}}
	
	\maketitle

	\begin{abstract}
	In this work, we carry out the asymptotic analysis of  the two dimensional  convective Brinkman-Forchheimer (CBF) equations, which characterize the motion of incompressible fluid flows in a saturated porous medium. We establish the existence of a global attractor in both bounded (using compact embedding) and  Poincar\'e domains (using asymptotic compactness property). In Poincar\'e domains, for $r=1,2$ and $3$, the estimates for Hausdorff as well as fractal	dimensions of the global attractors  are also obtained.  We then show an upper semicontinuity of global attractors for the 2D CBF equations. We consider an expanding sequence of simply connected,	bounded and smooth subdomains $\Omega_m$ of the Poincar\'e domain $\Omega$ such that $\Omega_m\to\Omega$ as $m\to\infty$. If $\mathscr{A}_m$ and $\mathscr{A}$ are the global attractors of the 2D CBF equations corresponding to $\Omega$ and $\Omega_m$, respectively, then we show that for  large enough $m$, the global attractor $\mathscr{A}_m$ enters  into any  neighborhood $\mathcal{U}(\mathscr{A})$ of $\mathscr{A}$. The presence of Darcy term in the CBF equations  helps us to obtain the above mentioned results in general unbounded domains also. Finally, we discuss about the quasi-stability property of the semigroup associated with the 2D CBF equations in bounded domains and establish the existence of finite fractal dimensional global as well as exponential attractors for $r\in[1,\infty)$. 
	\end{abstract}

	\section{Introduction}\label{sec1}\setcounter{equation}{0}
 Let $\Omega\subset\R^2$ be an open connected set (bounded or Poincar\'e domain).  We take  smooth boundary in $\R^2$ in the case of bounded domains. We say that $\Omega$ (can be unbounded) is a  Poincar\'e domain (see page 306, \cite{Te2}) if the Poincar\'e inequality is satisfied, that is, there exists  constant $\lambda_1 >0$ such that 
\begin{align}\label{2.1}
\int_{\Omega} |\phi(x)|^2 \d x \leq \frac{1}{\lambda_1}\int_{\Omega} |\nabla \phi(x)|^2 \d x,  \ \text{ for all } \  \phi \in \H^{1}_0 (\Omega).
\end{align}
Note that if $\Omega$ is bounded in some direction, then the Poincar\'e inequality holds. For $x=(x_1,x_2)\in\mathbb{R}^2$, one can think $\Omega$ is included in a region of the form  $0<x_1<L$. 
This work is concerned about the asymptotic analysis of solutions of the two dimensional  convective Brinkman-Forchheimer (CBF)  equations in bounded domains, Poincar\'e domains and general unbounded domains. 
Let us denote by $\u(x,t):\Omega\times[0,T]\to\R^2$, the \emph{velocity} of the fluid, $p(x,t):\Omega\times[0,T]\to \R$, the \emph{pressure} of the fluid and $\f(x,t):\Omega\times[0,T]\to\R^2$, an external time-dependent force. The CBF  equations are given by 
\begin{equation}\label{1}
\left\{
\begin{aligned}
\frac{\partial \u}{\partial t}-\mu \Delta\u+(\u\cdot\nabla)\u+\alpha\u+\beta|\u|^{r-1}\u+\nabla p&=\mathbf{f}, \ \text{ in } \ \Omega\times(0,T), \\ \nabla\cdot\u&=0, \ \text{ in } \ \Omega\times(0,T), \\
\u&=\mathbf{0}\ \text{ on } \ \partial\Omega\times(0,T), \\
\u(0)&=\u_0 \ \text{ in } \ \Omega,\\
\int_{\Omega}p(x,t)\d x&=0, \ \text{ in } \ (0,T).
\end{aligned}
\right.
\end{equation}
 The final condition in \eqref{1} is imposed for the uniqueness of the pressure $p$. The constant $\mu$ represents the positive Brinkman coefficient (effective viscosity), the positive constants $\alpha$ and $\beta$ represent the Darcy (permeability of porous medium) and Forchheimer (proportional to the porosity of the material) coefficients, respectively. The absorption exponent $r\in[1,\infty)$. The CBF equations \eqref{1} describe the motion of incompressible fluid flows in a saturated porous medium (cf. \cite{DNAB}).  The model given above is recognized to be more accurate when the flow velocity is too large for the Darcy's law to be valid alone, and in addition, the porosity is not too small.   It can be easily seen that for $\alpha=\beta=0$, we obtain the classical 2D Navier-Stokes equations. The nonlineartiy of the form $|\u|^{r-1}\u$ can be found in tidal dynamics as well as non-Newtonian fluid flows (see \cite{CTA,MTM3,MTM5}, etc and the references therein).  The global solvability of the system \eqref{1} in two and three dimensional bounded domains is available in \cite{SNA,VKK,MTM1}, etc.

The study of long time behavior of nonlinear dynamical system is an interesting branch of applied mathematics and it is essential in understanding many natural phenomena. For an extensive study on infinite dimensional dynamical systems in mathematical physics, the interested referred to see \cite{ICh,Te2}, etc. The existence of global attractors as well as their finite dimensionality results for the two dimensional Navier-Stokes equations (NSE) in bounded domains have been obtained in \cite{OAL1,Te2}, etc.   The author in \cite{RR} showed that  the global attractor exists for the 2D NSE in Poincar\'e domains with  forces $\f\in\V'$. In this case, the finite dimensionality of the global attractor is also obtained in  \cite{RR}. The upper semicontinuity of the global attractors  for the 2D NSE is established in \cite{ZD}. For an extensive literature on the existence of global attractors and related problems for the 2D NSE, the interested readers are referred to see \cite{FA,FA1,AVB,JMB,ICh,PCCF1,PCCF,CFM,NJ,OAL1,RR,JCR1,SS7,Te2} etc, and the references therein.

The authors in \cite{AIKP} considered finite energy solutions for the 2D NSE with linear damping  in the plane and showed that the corresponding dynamical system possesses a global attractor. Upper bounds for the number of asymptotic degrees (determining modes and nodes) of freedom for the 2D NSE with linear damping in periodic domains is obtained in \cite{AAI}.  The authors in \cite{VKK} established  the existence of regular dissipative solutions and global attractors for the 3D Brinkmann-Forchheimer equations with a nonlinearity of arbitrary polynomial growth rate. Using  of the theory of evolutionary systems due to Cheskidov \cite{ACh}, the authors in \cite{KWH} showed the existence of a strong global attractor for the critical 3D convective Brinkman-Forchheimer equations ($r=3$). The existence of a global attractor for the Brinkman-Forchheimer equations in bounded domains have been discussed in the works \cite{DUg,BWSL}, etc.  The existence of global attractors for the dynamical systems generated by weak as well as strong solutions of the 3D NSE with damping in bounded domains has been obtained in the works \cite{FLBY,DPJV,XSYH1,XSYH}, etc. As mentioned  in the paper \cite{VKK}, the major difficulty in working with bounded domains $\Omega\subset\mathbb{R}^3$ is that $\mathcal{P}(|\u|^{r-1}\u)$ (here $\mathcal{P}:\L^p(\Omega)\to\H,$ $p\in[2,\infty)$, is the Helmholtz-Hodge projection) need not be zero on the boundary, and $\mathcal{P}$ and $-\Delta$ are not necessarily commuting (for a counter example, see Example 2.19, \cite{JCR4}). Moreover, $\Delta\u\big|_{\partial\Omega}\neq 0$ in general and the term with pressure will not disappear (see \cite{VKK}), while taking the inner product with $\Delta\u$ to the first equation in \eqref{1}. Therefore, the equality (\cite{KWH})
\begin{align}\label{13}
&\int_{\Omega}(-\Delta\u(x))\cdot|\u(x)|^{r-1}\u(x)\d x\nonumber\\&=\int_{\Omega}|\nabla\u(x)|^2|\u(x)|^{r-1}\d x+4\left[\frac{r-1}{(r+1)^2}\right]\int_{\Omega}|\nabla|\u(x)|^{\frac{r+1}{2}}|^2\d x\nonumber\\&=\int_{\Omega}|\nabla\u(x)|^2|\u(x)|^{r-1}\d x+\frac{r-1}{4}\int_{\Omega}|\u(x)|^{r-3}|\nabla|\u(x)|^2|^2\d x,
\end{align}
may not be useful in the context of bounded domains. But in periodic domains or whole space the above equality holds true.  Due to this technical difficulty, it appears to the author that some of the results obtained in the above works may not be true in bounded domains.  

 In this work, we study the asymptotic analysis of solutions of \eqref{1} under the following  three cases:
\begin{enumerate}
	\item [(i)]  the domain $\Omega\subset\R^2$ is bounded with smooth boundary $\partial\Omega$,
	\item [(ii)]  an unbounded open set in $\R^2$ without any  regularity assumption on its boundary $\partial\Omega$ and with the only assumption that the Poincar\'e inequality \eqref{2.1} holds on it (Poincar\'e domain), 
	\item [(iii)] a general unbounded domain. 
\end{enumerate} We show the existence of global attractors in both bounded and  Poincar\'e domains for the system \eqref{1}. In Poincar\'e domains case, we are able to prove that the solution map $\S(t) : \H \to \H$ is Fr\'echet differentiable with respect to the initial data for the absorption exponent $r=1,2,3$ and hence we obtain estimates for the Hausdorff and Fractal dimensions of such attractors.  Then, we establish an upper semicontinuity of global attractors for the 2D CBF equations. We take an expanding sequence of simply connected, bounded and smooth subdomains $\Omega_m$ of the Poincar\'e domain $\Omega$. If $\mathscr{A}_m$ and $\mathscr{A}$ are the global attractors of \eqref{1} corresponding to $\Omega$ and $\Omega_m$, respectively, then we show that for  large enough $m$, the global attractor $\mathscr{A}_m$ enters  into any  neighborhood $\mathcal{U}(\mathscr{A})$ of $\mathscr{A}$. We remark that the above mentioned results are true in general unbounded domains also (the presence of Darcy term in \eqref{1} helps us to get such results). Finally, we discuss about the quasi-stability of the semigroup $\S(t) : \H \to \H$ and establish the existence of finite fractal dimensional global as well as exponential attractors for the  2D CBF equations  in bounded domains for all $r\geq 1$. 

The rest of the paper is organized as follows. In the next section, we provide an  abstract formulation of the system \eqref{1} and give the necessary function spaces needed to prove the global solvability results of the system \eqref{1}. We also discuss about the existence and uniqueness of the weak as well as strong solutions to the system \eqref{1} in the same section. In section \ref{sec3},  for $\f\in\V'$, we show that the system \eqref{1} in bounded domains possesses a global attractor, using the compact embedding of $\V\subset\H$ (Remark \ref{rem3.7} and Theorem \ref{thm3.4}). The existence of a global attractor for the system \eqref{1} in Poincar\'e domains is obtained in section \ref{sec4}, using asymptotic compactness property (Theorem \ref{thm4.3}). For the absorption exponent $r=1,2,3$, the estimates for  Hausdorff as well as Fractal	dimensions of the global attractor for the system \eqref{1} in Poincar\'e domains is obtained in section \ref{sec5} (Theorem \ref{thm52}).  In section \ref{sec7}, we prove an upper semicontinuity of global attractors for the 2D CBF equations (Lemma \ref{lem4.2} and Theorem \ref{thm4.1}). In the final section, we describe the quasi-stability property of the semigroup associated with the 2D CBF equations in bounded domains for the absorption exponent $r\in[1,\infty)$ and establish the existence of finite fractal dimensional global as well as exponential attractors (Theorem \ref{thm614}) using Ladyzhenskaya's squeezing property (Proposition \ref{prop6.12}) and the results available in \cite{ICh}.

\section{Mathematical Formulation}\label{sec2}\setcounter{equation}{0}
This section is devoted for providing  the necessary function spaces needed to obtain the global solvability results of the system \eqref{1} as well as to discuss about the well-posedness of the system \eqref{1}.

\subsection{Functional setting}
Let us define  the space $$\mathscr{V}:=\left\{\u\in\C_0^{\infty}(\Omega;\R^2):\nabla\cdot\u=0\right\},$$ where  $\C_0^{\infty}(\Omega;\R^2)$ is the space of all infinitely differentiable functions with compact support in $\Omega$.  Let $\H$ and $\V$ denote the completion of $\mathscr{V}$ in $\mathrm{L}^2(\Omega;\R^2)$ and $\mathrm{H}^1(\Omega;\R^2)$ norms respectively. For bounded domains, under some smoothness assumptions on the boundary, we characterize the spaces $\H$ and $\V$ as 
$$\H:=\left\{\u\in\L^2(\Omega):\nabla\cdot\u=0,\u\cdot\n\big|_{\partial\Omega}=0\right\},$$  with the norm defined by $\|\u\|_{\H}^2:=\int_{\Omega}|\u(x)|^2\d x,
$ where $\n$ is the outward normal to the boundary $\partial\Omega$  and $\u\cdot\n\big|_{\partial\Omega}$ should be understood in the sense of trace in $\H^{-1/2}(\partial\Omega)$ (cf. Theorem 1.2, Chapter 1, \cite{Te}), 
$$\wi\L^p:=\left\{\u\in\L^p(\Omega):\nabla\cdot\u=0,\u\cdot\n\big|_{\partial\Omega}=0\right\},$$  with the norm denoted by $\|\u\|_{\wi\L^p}^2:=\int_{\Omega}|\u(x)|^p\d x,
$ for $p\in(2,\infty)$ and 
$$\V:=\left\{\u\in\H^1(\Omega):\nabla\cdot\u=0\right\}, $$ with the norm given by $ \int_{\Omega}|\u(x)|^2\d x+\int_{\Omega}|\nabla\u(x)|^2\d x.
$ Using the Poincar\'e inequality \eqref{2.1}, one can easily see that this norm is equivalent to the norm $ \|\u\|_{\V}^2:=\int_{\Omega}|\nabla\u(x)|^2\d x.$
 The inner product in the Hilbert space $\H$ is denoted by $(\cdot,\cdot)$, in $\V$ is represented by $[\cdot,\cdot]$ and the induced duality, for instance between the spaces $\V$  and its dual $\V'$, and $\wi\L^p$ and its dual $\wi\L^{\frac{p}{p-1}}$ is expressed by $\left<\cdot,\cdot\right>$. Note that $\V$ is densely and continuously embedded into $\H$ and $\H$ can be identified with its dual $\H'$ and we have the  \emph{Gelfand triple}: $\V\subset	\H\equiv\H'\subset\V'.$ Remember that for a bounded domain $\Omega$, the embedding of $\V\subset\H$ is compact. In the rest of the paper, we also use the notation $\mathbb{H}^2(\Omega):=\mathrm{H}^2(\Omega;\mathbb{R}^2)$ for the second order Sobolev spaces. 
 
 \iffalse 
 Let us now provide the Gagliardo-Nirenberg and Ladyzhenskaya inequalities, which are frequently used  in the paper.  The Gagliardo-Nirenberg inequality is given below:
 \begin{lemma}[Gagliardo-Nirenberg inequality, Theorem 2.1, \cite{LN}] \label{gn}
 	Let $\Omega\subset\R^n$ and $\u\in\mathrm{W}_0^{1,p}(\Omega;\R^n),\ p\geq 1$. Then for any fixed number $q,r\geq 1$, there exists a constant $C>0$ depending only on $n,p,q$ such that 
 	\begin{align}\label{gag}
 	\|\u\|_{\mathbb{L}^r}\leq C\|\nabla\u\|_{\mathbb{L}^p}^{\theta}\|\u\|_{\mathbb{L}^q}^{1-\theta},\ \theta\in[0,1],
 	\end{align}
 	where the numbers $p, q, r$ and $\theta$ satisfy the relation
 	$$\theta=\left(\frac{1}{q}-\frac{1}{r}\right)\left(\frac{1}{n}-\frac{1}{p}+\frac{1}{q}\right)^{-1}.$$
 \end{lemma}
 Now we state an important inequality due to Ladyzhenskaya (see Lemma 1 and 2, Chapter 1, \cite{OAL}), which is true even in unbounded domains:
 \begin{lemma}[Ladyzhenskaya inequality]\label{lady}
 	For $\u\in\ \C_0^{\infty}(\Omega;\R^n), n = 2, 3$, there exists a constant $C$ such that
 	\begin{align}
 	\|\u\|_{\L^4}\leq C^{1/4}\|\u\|^{1-\frac{n}{4}}_{\L^2}\|\nabla\u\|^{\frac{n}{4}}_{\L^2},\text{ for } n=2,3,
 	\end{align}
 	where $C=2,4$ for $n=2,3$ respectively. 
 \end{lemma}
\fi
\subsection{Linear operator}
It is well-known  that every vector field $\u\in\mathbb{L}^p(\Omega)$, for $1<p<\infty$ can be uniquely represented as $\u=\v+\nabla q,$ where $\v\in\mathbb{L}^p(\Omega)$ with $\mathrm{div \ }\v=0$ in the sense of distributions in $\Omega$ with $\v\cdot\n=0$ on $\partial\Omega$ and $q\in\mathrm{W}^{1,p}(\Omega)$ (Helmholtz-Weyl or Helmholtz-Hodge decomposition, cf. \cite{DFHM,HKTY}, etc). For smooth vector fields in $\Omega$, such a decomposition is an orthogonal sum in $\mathbb{L}^2(\Omega)$. Note that $\u=\v+\nabla q$ holds for all $\u\in\mathbb{L}^p(\Omega)$, so that we can define the projection operator $\mathcal{P}_p$  by $\mathcal{P}_p\u = \v$. For simplicity, we use the notation $\mathcal{P}$ instead of $\mathcal{P}_p$ in the rest of the paper.  Let us define
\begin{equation*}
\A\u:=-\mathcal{P}\Delta\u,\ \u\in\D(\A):=\V\cap\H^{2}(\Omega).
\end{equation*}
Remember that the operator $\A$ is a non-negative, self-adjoint operator in $\H$  and \begin{align}\label{2.7a}\left<\A\u,\u\right>=\|\u\|_{\V}^2,\ \textrm{ for all }\ \u\in\V, \ \text{ so that }\ \|\A\u\|_{\V'}\leq \|\u\|_{\V}.\end{align}
\iffalse 
Moreover, for $\u\in\D(\A)$, we have 
\begin{align*}
\|\u\|_{\V}^2=(\nabla\u,\nabla\u)=(\A\u,\u)\leq\|\A\u\|_{\H}\|\u\|_{\H}\leq\frac{1}{\lambda_1^{1/2}}\|\A\u\|_{\H}\|\u\|_{\V},
\end{align*}
so that we get 
\begin{align}\label{22}
\|\u\|_{\V}\leq\frac{1}{\lambda_1^{1/2}}\|\A\u\|_{\H}, \ \text{ for all }\u\in\D(\A). 
\end{align}
\fi 
\begin{remark}
It is well known  that for a bounded domain $\Omega$, the operator $\A$ is invertible and its inverse $\A^{-1}$ is bounded, self-adjoint and compact in $\H$. Hence the spectrum of $\A$ consists of an infinite sequence $0< \lambda_1\leq \lambda_2\leq\ldots\leq \lambda_k\leq \ldots,$ with $\lambda_k\to\infty$ as $k\to\infty$ of eigenvalues, and there exists an orthogonal basis $\{e_k\}_{k=1}^{\infty} $ of $\H$ consisting of eigenvectors of $\A$ such that $\A e_k =\lambda_ke_k$,  for all $ k\in\mathbb{N}$.  Since $\u=\sum\limits_{j=1}^{\infty}(\u,e_j)e_j$, we have  $\A\u=\sum\limits_{j=1}^{\infty}\lambda_j(\u,e_j)e_j$. Thus, it is immediate that 
\begin{align*}
\|\nabla\u\|_{\mathbb{H}}^2=\langle \A\u,\u\rangle =\sum_{j=1}^{\infty}\lambda_j|(\u,e_j)|^2\geq \lambda_1\sum_{j=1}^{\infty}|(\u,e_j)|^2=\lambda_1\|\u\|_{\mathbb{H}}^2,
\end{align*}
for all $\u\in\V$, which is the \emph{Poincar\'e inequality}.  Note that the constant appearing in \eqref{2.1} is different from the first eigenvalue $\lambda_1$ and for convenience, we are using the same notation. In the rest of the paper, these constants will be chosen according to the context. We  use the following interpolation inequality in the sequel: If for some $\alpha,\beta,\gamma\in\mathbb{R}$ with  $\gamma=\theta\alpha+(1-\theta)\beta$ and  $\theta\in[0,1]$, then  we have (see \cite{ICh})
\begin{align}\label{inp}
\|\A^{\gamma}\u\|_{\H}\leq \|\A^{\alpha}\u\|_{\H}^{\theta}\|\A^{\beta}\u\|_{\H}^{(1-\theta)}, \ \text{ for all }\ \u\in \D(\A^{\alpha})\cap\D(\A^{\beta})\cap\D(\A^{\gamma}). 
\end{align}
\end{remark}

\begin{remark}
	The following important estimates are needed in the rest of the paper:
	\begin{enumerate}
		\item [(i)] Assume $1\leq s\leq r\leq t\leq \infty$, $\theta\in(0,1)$ such that $\frac{1}{r}=\frac{\theta}{s}+\frac{1-\theta}{t}$ and $\u\in\L^s(\mathcal{O})\cap\L^t(\mathcal{O})$, then we have 
		\begin{align}\label{211}
		\|\u\|_{\L^r}\leq\|\u\|_{\L^s}^{\theta}\|\u\|_{\L^t}^{1-\theta}. 
		\end{align}
	\item [(ii)] Using Ladyzhenskaya inequality (Lemma 1, Chapter 1, \cite{OAL}), we have 
		\begin{align}
	\|\u\|_{\L^4}\leq 2^{1/4}\|\u\|^{\frac{1}{2}}_{\L^2}\|\nabla\u\|^{\frac{1}{2}}_{\L^2},\ \u\in\H_0^1(\Omega).
	\end{align}
	\item [(iii)] Applying  Gagliardo-Nirenberg inequality  (Theorem 1, \cite{LN}), we obtain 
		\begin{align}\label{Gen_lady}
	\|\u\|_{\L^{p} } \leq C \|\u\|^{\frac{2}{p}}_{\L^{2} } \|\nabla \u\|^{1 - \frac{2}{p}}_{\L^{2} }, \ \u\in \H^{1,2}_{0} (\Omega),
	\end{align}
	for all $p\in[2,\infty)$.
	From the inequality \eqref{Gen_lady} and Poincar\'e's inequality \eqref{2.1}, it is clear that  $\V\subset\L^p$, for all $p\in[2,\infty)$.
	\item [(iv)] An application of Agmon's inequality yields 
	\begin{align}\label{24}
	\|\u\|_{\L^{\infty}}\leq C\|\u\|_{\L^2}^{1/2}\|\u\|_{\H^2}^{1/2},  \ \u\in\H^2(\Omega). 
	\end{align}
	\end{enumerate}
\end{remark}

\subsection{Trilinear form and bilinear operator}
We define the \emph{trilinear form} $b(\cdot,\cdot,\cdot):\V\times\V\times\V\to\R$ by $$b(\u,\v,\w)=\int_{\Omega}(\u(x)\cdot\nabla)\v(x)\cdot\w(x)\d x=\sum_{i,j=1}^2\int_{\Omega}\u_i(x)\frac{\partial \v_j(x)}{\partial x_i}\w_j(x)\d x.$$ If $\u, \v$ are such that the linear map $b(\u, \v, \cdot) $ is continuous on $\V$, the corresponding element of $\V'$ is denoted by $\B(\u, \v)$. We represent   $\B(\u) = \B(\u, \u)=\mathcal{P}(\u\cdot\nabla)\u$. Using an integration by parts, it is immediate that 
\begin{equation}\label{b0}
\left\{
\begin{aligned}
b(\u,\v,\v) &= 0,\ \text{ for all }\ \u,\v \in\V,\\
b(\u,\v,\w) &=  -b(\u,\w,\v),\ \text{ for all }\ \u,\v,\w\in \V.
\end{aligned}
\right.\end{equation}
It can be easily seen that $\B$ maps $\wi\L^4$ (and so $\V$) into $\V'$ and
\begin{align*}
\left|\left<\B(\u,\u),\v\right>\right|=\left|b(\u,\v,\u)\right|\leq \|\u\|_{\L^4}^2\|\nabla\v\|_{\H}\leq \sqrt{2}\|\u\|_{\H}\|\nabla\u\|_{\H}\|\v\|_{\V},
\end{align*}
for all $\v\in\V$, so that 
\begin{align}\label{2.9a}
\|\B(\u)\|_{\V'}\leq \sqrt{2}\|\u\|_{\H}\|\nabla\u\|_{\H}\leq \frac{\sqrt{2}}{\lambda_1^{1/4}}\|\u\|_{\V}^2,\ \text{ for all }\ \u\in\V,
\end{align}
using the Poincar\'e inequality. 

\subsection{Nonlinear operator}
Let us now consider the operator $\mathcal{C}(\u):=\mathcal{P}(|\u|^{r-1}\u)$. It is immediate that $\langle\mathcal{C}(\u),\u\rangle =\|\u\|_{\widetilde{\L}^{r+1}}^{r+1}$. Note that for $\u\in\wi\L^{r+1}(\mathcal{O})$, we have 
\begin{align*}
\|\mathcal{C}(\u)\|_{\wi\L^{\frac{r+1}{r}}}=\|\mathcal{P}(|\u|^{r-1}\u)\|_{\wi\L^{\frac{r+1}{r}}}\leq\|\mathcal{P}\|_{\mathcal{L}(\L^{\frac{r+1}{r}}(\mathcal{O});\wi\L^{\frac{r+1}{r}})}\|\u\|_{\L^{r+1}(\mathcal{O})}^{r}\leq C\|\u\|_{\wi\L^{r+1}}^{r},
\end{align*}
and hence the map $\mathcal{C}(\cdot):\wi{\L}^{r+1}(\mathcal{O})\to\widetilde{\L}^{\frac{r+1}{r}}$. Furthermore, for all $\u\in\wi\L^{r+1}$, the map is Gateaux differentiable with Gateaux derivative 
\begin{align}\label{29}
\mathcal{C}'(\u)\v&=\left\{\begin{array}{cc}\mathcal{P}(\v),&\text{ for }r=1,\\ \left\{\begin{array}{cc}\mathcal{P}(|\u|\v)+\mathcal{P}\left(\frac{\u}{|\u|}(\u\cdot\v)\right),&\text{ if }\u\neq \mathbf{0},\\\mathbf{0},&\text{ if }\u=\mathbf{0},\end{array}\right.&\text{ for } r=2,\\ \mathcal{P}(|\u|^{r-1}\v)+(r-1)\mathcal{P}(\u|\u|^{r-3}(\u\cdot\v)), &\text{ for }r\geq 3,\end{array}\right.
\end{align}
for $\v\in\widetilde{\L}^{r+1}$.  For $\u,\v\in\wi\L^{r+1}$, it can be easily seen that 
\begin{align}\label{2.9}
\langle\mathcal{C}'(\u)\v,\v\rangle=\int_{\mathcal{O}}|\u(x)|^{r-1}|\v(x)|^2\d x+(r-1)\int_{\mathcal{O}}|\u(x)|^{r-3}|\u(x)\cdot\v(x)|^2\d x\geq 0,
\end{align}
for $r\geq 1$ (note that for $r=2$ also the same result holds, since in that case, the second integral becomes $\int_{\{x\in\mathcal{O}:\u(x)\neq 0\}}\frac{1}{|\u(x)|}|\u(x)\cdot\v(x)|^2\d x\geq 0$). For $r\geq 3$, we further have 
\begin{align}\label{30}
\mathcal{C}''(\u)(\v\otimes\w)&=(r-1)\mathcal{P}\left\{|\u|^{r-3}\left[(\u\cdot\w)\v+(\u\cdot\v)\w+(\w\cdot\v)\u\right]\right\}\nonumber\\&\quad+(r-1)(r-3)\mathcal{P}\left[|\u|^{r-5}(\u\cdot\v)(\u\cdot\w)\u\right],
\end{align}
for all $\u,\v,\w\in\widetilde{\L}^{r+1}$ (for $r=4$, the final term appearing in the right hand side of the equality \eqref{30} is defined for $\u\neq\mathbf{0}$ and one has to take $\mathbf{0}$ for $\u=\mathbf{0}$). For  $\u,\v\in\widetilde{\L}^{r+1}$, using Taylor's formula (Theorem 7.9.1, \cite{PGC}), we have 
\begin{align}\label{213}
\langle \mathcal{P}(|\u|^{r-1}\u)-\mathcal{P}(|\v|^{r-1}\v),\w\rangle&\leq \|(|\u|^{r-1}\u)-(|\v|^{r-1}\v)\|_{\widetilde{\L}^{\frac{r+1}{r}}}\|\w\|_{\widetilde{\L}^{r+1}}\nonumber\\&\leq\sup_{0<\theta<1}r\|(\u-\v)|\theta\u+(1-\theta)\v|^{r-1}\|_{\widetilde{\L}^{{\frac{r+1}{r}}}}\|\w\|_{\widetilde{\L}^{r+1}}\nonumber\\&\leq r\sup_{0<\theta<1}\|\theta\u+(1-\theta)\v\|_{\widetilde{\L}^{r+1}}^{r-1}\|\u-\v\|_{\widetilde{\L}^{r+1}}\|\w\|_{\widetilde{\L}^{r+1}}\nonumber\\&\leq r\left(\|\u\|_{\widetilde{\L}^{r+1}}+\|\v\|_{\widetilde{\L}^{r+1}}\right)^{r-1}\|\u-\v\|_{\widetilde{\L}^{r+1}}\|\w\|_{\widetilde{\L}^{r+1}},
\end{align}
for all $\u,\v,\w\in\widetilde{\L}^{r+1}$. 
Thus the operator $\mathcal{C}(\cdot):\widetilde{\L}^{r+1}\to\widetilde{\L}^{\frac{r+1}{r}}$ is locally Lipschitz. Moreover, 	for any $r\in[1,\infty)$, we have (cf. \cite{MTM1})
\begin{align}\label{2.23}
&\langle\mathcal{P}(\u|\u|^{r-1})-\mathcal{P}(\v|\v|^{r-1}),\u-\v\rangle\geq \frac{1}{2}\||\u|^{\frac{r-1}{2}}(\u-\v)\|_{\H}^2+\frac{1}{2}\||\v|^{\frac{r-1}{2}}(\u-\v)\|_{\H}^2\geq 0,
\end{align}
for $r\geq 1$. 	It is important to note that (cf. \cite{MTM1})
\begin{align}\label{a215}
\|\u-\v\|_{\wi\L^{r+1}}^{r+1}\leq 2^{r-2}\||\u|^{\frac{r-1}{2}}(\u-\v)\|_{\H}^2+2^{r-2}\||\v|^{\frac{r-1}{2}}(\u-\v)\|_{\H}^2,
\end{align}
for $r\geq 1$ (replace $2^{r-2}$ with $1,$ for $1\leq r\leq 2$).
\iffalse 
For the critical case $r=3$, we have 
\begin{align}\label{2p23}
&\langle \mathcal{P}(\u|\u|^{2})-\mathcal{P}(\v|\v|^{2}),\u-\v\rangle\nonumber\\&=\langle \u|\u|^2-\v|\v|^2,\u-\v\rangle =\langle\u(|\u|^2-|\v|^2)+(\u-\v)|\v|^2,\u-\v\rangle \nonumber\\&=\|\v(\u-\v)\|_{\H}^2+\langle\u\cdot(\u-\v),(\u+\v)\cdot(\u-\v)\rangle \nonumber\\&=\|\v(\u-\v)\|_{\H}^2+\|\u\cdot(\u-\v)\|_{\H}^2+\langle\u\cdot(\u-\v),\v\cdot(\u-\v)\rangle \nonumber\\&\geq \|\v(\u-\v)\|_{\H}^2+\|\u\cdot(\u-\v)\|_{\H}^2-\|\u\cdot(\u-\v)\|_{\H}\|\v(\u-\v)\|_{\H}\nonumber\\&\geq\frac{1}{2}\|\v(\u-\v)\|_{\H}^2+\frac{1}{2}\|\u\cdot(\u-\v)\|_{\H}^2\geq \frac{1}{2}\|\v(\u-\v)\|_{\H}^2\geq 0,
\end{align}
where we used Young's inequality.

Applying integration by parts, we get 
\begin{align}\label{219}
(\mathcal{C}(\u),\A\u)&=\sum_{i,j=1}^n\int_{\mathcal{O}}|\u(x)|^{r-1}\u_j(x)\frac{\partial^2\u_j(x)}{\partial x_i^2}\d x\nonumber\\&= -\sum_{i,j=1}^n\int_{\mathcal{O}}|\u(x)|^{r-1}\left(\frac{\partial\u_j(x)}{\partial x_i}\right)^2\d x-\sum_{i,j=1}^n\int_{\mathcal{O}}\frac{\partial}{\partial x_i}|\u(x)|^{r-1}\u_j(x)\frac{\partial\u_j(x)}{\partial x_i}\d x\nonumber\\&=-(r-1)\sum_{i,j=1}^n\int_{\mathcal{O}}|\u(x)|^{r-1}\left(\frac{\partial\u_j(x)}{\partial x_i}\right)^2\d x=-(r-1)\||\u|^{\frac{r-1}{2}}|\nabla\u|\|_{\H}^2.
\end{align}

\fi 

\begin{lemma}[Remark 2.4, \cite{MTM4}]\label{lem2.3}
For 	$r\in[1,3]$ and $\u,\v\in\V$, we obtain 
	\begin{align}\label{fe2}
	\langle(\G(\u)-\G(\v),\u-\v\rangle+ \frac{27}{32\mu ^3}N^4\|\u-\v\|_{\H}^2\geq 0,
	\end{align}
	for all $\v\in\widehat{\mathbb{B}}_N$, where $\widehat{\mathbb{B}}_N$ is an $\widetilde{\L}^4$-ball of radius $N$, that is,
	$
	\widehat{\mathbb{B}}_N:=\big\{\z\in\widetilde{\L}^4:\|\z\|_{\widetilde{\L}^4}\leq N\big\}.
	$ Thus, the operator $\G(\cdot)$ is locally monotone.
\end{lemma}

\begin{lemma}[Theorem 2.2, \cite{MTM1}]\label{lem2.4}
	Let $\u,\v\in\V$, for $r>3$. Then,	for the operator $\G(\u)=\mu \A\u+\B(\u)+\beta\mathcal{C}(\u)$, we  have 
	\begin{align}\label{fe}
	\langle(\G(\u)-\G(\v),\u-\v\rangle+\eta\|\u-\v\|_{\H}^2\geq 0,
	\end{align}
	where \begin{align}\label{215}\eta=\frac{r-3}{2\mu(r-1)}\left(\frac{2}{\beta\mu (r-1)}\right)^{\frac{2}{r-3}}.\end{align} That is, the operator $\G+\eta\mathrm{I}$ is a monotone operator from $\V$ to $\V'$. 
\end{lemma}

\subsection{Abstract formulation, weak and strong solutions}  Let us take  the projection $\mathcal{P}$ onto the system \eqref{1} to write down the abstract formulation of the system \eqref{1} as: 
\begin{equation}\label{kvf}
\left\{
\begin{aligned}
\frac{\d\u(t)}{\d t}+\mu\A\u(t)+\B(\u(t))+\alpha\u(t)+\beta\mathcal{C}(\u(t))&=\f(t), \ t\in (0,T),\\
\u(0)&=\u_0\in\H,
\end{aligned}
\right.
\end{equation}
where $\f\in\mathrm{L}^2(0,T;\V')$. For convenience, we used $\f$ instead of $\mathcal{P}\f$. 
Let us now provide the definition of \emph{weak solutions} of the system \eqref{kvf}. 
\begin{definition}\label{weakd}
	A function  $\u\in\mathrm{L}^{\infty}(0,T;\H)\cap\mathrm{L}^2(0,T;\V)\cap\mathrm{L}^{r+1}(0,T;\wi\L^{r+1})$ with $\partial_t\u\in\mathrm{L}^{\frac{r+1}{r}}(0,T;\mathbb{V}'),$  is called a \emph{weak solution} to the system (\ref{kvf}), if for $\f\in\mathrm{L}^2(0,T;\mathbb{V}')$, $\u_0\in\H$ and $\v\in\V$, $\u(\cdot)$ satisfies:
	\begin{equation}\label{3.13}
	\left\{
	\begin{aligned}
	\langle\partial_t\u(t)+\mu\mathbf{A}\u(t)+\mathbf{B}(\u(t)+\alpha\u(t)+\beta\mathcal{C}(\u(t)),\v\rangle&=\langle\f(t),\v\rangle,\\
	\lim_{t\downarrow 0}\int_{\Omega}\u(t)\v\d x&=\int_{\Omega}\u_0\v\d x,
	\end{aligned}
	\right.
	\end{equation}
	and the energy equality 
	\begin{align}\label{energy}
\frac{1}{2}	\frac{\d}{\d t}\|\u(t)\|^2_{\H}+\mu\|\u(t)\|^2_{\V}+\alpha\|\u(t)\|_{\H}^2+\beta\|\u(t)\|_{\wi\L^{r+1}}^{r+1}=\langle \f(t),\u(t)\rangle,
	\end{align}
	for a.e. $t\in(0,T)$.
\end{definition}

\iffalse 
\begin{definition}\label{strongd}
	A function  $\u\in\mathrm{L}^{\infty}(0,T;\V)\cap\mathrm{L}^2(0,T;\D(\A))$ with $\partial_t\u\in\mathrm{L}^{2}(0,T;\H),$  is called a \emph{strong solution} to the system (\ref{kvf}), if for $\f\in\mathrm{L}^2(0,T;\H)$, $\u_0\in\V$,  $\u(\cdot)$ satisfies \eqref{kvf} as an equality in $\H,$ for a.e. $t\in(0,T)$. 
\end{definition}
\fi

	\begin{remark}\label{rem5.10}

	1. For Poincar\'e domains, the embedding of $\H\subset\V'$ is continuous and it can be easily seen that $\u\in \mathrm{L}^{\infty}(0,T;\H)$ implies $\u\in \mathrm{L}^{\infty}(0,T;\V')$. Thus, it is immediate that $\u,\frac{\d\u}{\d t}\in \mathrm{L}^{\frac{r+1}{r}}(0,T;\V')$ and then using Theorem 2, section 5.9.2 \cite{LCE}, it can be easily seen that $\u\in\C([0,T];\V')$. The reflexivity of the space $\H$ and  Proposition 1.7.1 \cite{PCAM} gives $\u\in\C_w([0,T];\H)$ and the map $t\mapsto\|\u(t)\|_{\H}$ is bounded, where $\C_w([0,T];\H)$ represents the space of functions $\u:[0,T]\to\H$ which are weakly continuous. Thus the second condition  in the Definition \ref{weakd} makes sense. 
	
	2.  If the existence of a weak solution to the problem \eqref{weakd} is known, then following similar arguments as in the proof of Theorem 3.1, \cite{GPG},  one can establish that the solution verifies:
	\begin{enumerate}
		\item [(i)] the energy inequality: 
		\begin{align*}
		&\|\u(t)\|_{\H}^2+2\mu\int_0^t\|\u(s)\|_{\V}^2\d s+2\alpha\int_0^t\|\u(s)\|_{\V}^2\d s+2\beta\int_0^t\|\u(s)\|_{\wi\L^{r+1}}^{r+1}\d s\\&\leq \|\u_0\|_{\H}^2+2\int_0^t\langle\f,\u(s)\rangle\d s,
		\end{align*}
		for all $\ t\in[0,T].$
		\item [(ii)] $\lim\limits_{t\to 0}\|\u(t)-\u_0\|_{\H}=0.$
	\end{enumerate}

	3. Moreover, if $\u(\cdot)$ is a unique weak solution of \eqref{weakd} and if $\u\in\mathrm{L}^4(0,T;\wi\L^4)$, then following similar arguments as in Theorem 4.1, \cite{GPG}, we infer that $\u(\cdot)$ satisfies the energy equality: 
	\begin{align}\label{521}
	&\|\u(t)\|_{\H}^2+2\mu\int_0^t\|\u(s)\|_{\V}^2\d s+2\alpha\int_0^t\|\u(s)\|_{\H}^2\d s+2\beta\int_0^t\|\u(s)\|_{\wi\L^{r+1}}^{r+1}\d s\nonumber\\& = \|\u_0\|_{\H}^2+2\int_0^t\langle\f,\u(s)\rangle\d s,
	\end{align}
	for all $\ t\in[0,T],$
	For $r\geq 1$, $\u\in\mathrm{L}^4(0,T;\wi\L^4)$ can be established by using Ladyzhenskaya's inequality as
	\begin{align*}
	\int_0^T\|\u(t)\|_{\wi\L^4}^4\d t\leq 2\sup_{t\in[0,T]}\|\u(t)\|_{\H}^2\int_0^T\|\u(t)\|_{\V}^2\d t<\infty,
	\end{align*}
	since $\u\in \mathrm{L}^{\infty}(0,T;\H)\cap\mathrm{L}^2(0, T;\V)$.
For $r>3$, we use the fact that $\u\in\mathrm{L}^{r+1}(0,T;\wi\L^{r+1})$ and the interpolation inequality to obtain that $\u\in\mathrm{L}^4(0,T;\wi\L^4)$. Using \eqref{211}, we find 
	\begin{align*}
	\|\u\|_{\wi\L^4}\leq\|\u\|_{\H}^{\frac{r-3}{2(r-1)}}\|\u\|_{\wi\L^{r+1}}^{\frac{r+1}{2(r-1)}}
	\end{align*}
	and 
	\begin{align*}
	\int_0^T\|\u(t)\|_{\wi\L^4}^4\d t\leq T^{\frac{r-3}{r-1}}\sup_{t\in[0,T]}\|\u(t)\|_{\H}^{\frac{2(r-3)}{r-1}}\left(\int_0^T\|\u(t)\|_{\wi\L^{r+1}}^{r+1}\d t\right)^{\frac{2}{r-1}}<\infty. 
	\end{align*}

4.	Remember that every weak solution of \eqref{weakd} is $\H$ weakly continuous in time, all weak solutions satisfying the energy equality \eqref{521} and so, all weak solutions of \eqref{weakd} belong to $\C([0, T];\H)$ (see \cite{GPG,KWH} also).
\end{remark}

Let us now state the existence and uniqueness theorem for the system \eqref{kvf}. A proof of the following theorem can be obtained from Theorem 3.4, \cite{MTM1}. A  monotonicity property (local for $r\in[1,3]$) of the linear and nolinear operators (Lemmas \ref{lem2.3} and \ref{lem2.4}) and a generalization of the Minty-Browder technique are exploited in the proof of Theorem 3.4, \cite{MTM1}. 

\begin{theorem}\label{weake} For $\u_0\in\H$ and $
	\f\in\mathrm{L}^2(0,T;\V')$, there exists a \emph{unique weak solution} $\u(\cdot)$  to the system \eqref{kvf} in the sense of Definition \ref{weakd} satisfying the energy equality \eqref{521}. 
\end{theorem}

\section{Global Attractor (Bounded Domains)}\label{sec3}\setcounter{equation}{0} 
In this section, we discuss the existence of a global attractor for the CBF equations in two dimensional bounded domains. We assume that $\f\in\V'$ is independent of $t$ in \eqref{kvf}. 

\begin{lemma}
	Let $\u(\cdot)$ be the unique weak solution of the system \eqref{kvf}.  Then, we have 
	\begin{align}\label{3p9}
	\frac{1}{t}\int_0^t\left[\mu\|\u(s)\|_{\V}^2+\alpha\|\u(s)\|_{\H}^2+\beta\|\u(s)\|_{\wi\L^{r+1}}^{r+1}\right]\d s\leq  \frac{\|\u_0\|_{\H}^2}{t}+\frac{1}{\mu}\|\f\|_{\V'}^2.
	\end{align}
\end{lemma}
\begin{proof}
Let us take the inner product with $\u(\cdot)$ to the first equation in \eqref{kvf}  to get 
\begin{align}\label{3p5}
&\frac{1}{2}\frac{\d}{\d t}\|\u(t)\|_{\H}^2+\mu\|\nabla\u(t)\|_{\H}^2+\alpha\|\u(t)\|_{\H}^2+\beta\|\u(t)\|_{\wi\L^{r+1}}^{r+1} =\langle \f,\u(t)\rangle,
\end{align}
where we used the fact that $\langle \B(\u),\u\rangle =0$. 
Using the Cauchy-Schwarz inequality and Young's inequality, we estimate  $|\langle \f,\u\rangle|$ as 
\begin{align}
|\langle\f,\u\rangle|&\leq \|\f\|_{\V'}\|\u\|_{\V}\leq \frac{\mu}{2}\|\u\|_{\V}^2+\frac{1}{2\mu}\|\f\|_{\V'}^2. \label{3p6}
\end{align}
Applying \eqref{3p6}  in \eqref{3p5} and then  integrating from $0$ to $t$, we obtain 
\begin{align}\label{3p8}
\|\u(t)\|_{\H}^2+\mu\int_0^t\|\u(s)\|_{\V}^2\d s+2\alpha\int_0^t\|\u(s)\|_{\H}^2\d s+2\beta\int_0^t\|\u(s)\|_{\wi\L^{r+1}}^{r+1}\d s\leq \|\u_0\|_{\H}^2+\frac{t}{\mu}\|\f\|_{\V'}^2. 
\end{align}
From \eqref{3p8}, we easily obtain \eqref{3p9}. 
\end{proof}

Note that whenever a Cauchy problem on a Banach space $\X$ is well posed in the sense of Hadamard (for proper notion of weak solution) for all initial data $\x(0)=\x_0\in\X$, then the corresponding (forward) solutions $\x(t)$ can be written in the form $\x(t) = \S(t)\x_0$, where the semigroup $\S(t)$ is uniquely determined by the dynamical system. Let $\u(t)$, $t\geq 0$ be the unique weak solution of the system \eqref{kvf} with $\u_0\in\H$. Thanks to the existence and uniqueness of weak solution for the system \eqref{kvf} (Theorem \ref{weake}), we can define a continuous semigroup $\{\S(t)\}_{t\geq 0}$ in $\H$ by 
\begin{align}
\S(t)\u_0= \u(t), \ t\geq 0,
\end{align}
where $\u(\cdot)$ is the unique weak solution of \eqref{kvf} with $\u(0)=\u_0\in\H$. 

  Next, we prove the existence of a global attractor for the semigroup $\S(t),  t \geq 0$ defined on $\H$ for the 2D CBF  \eqref{kvf} in bounded domains.  Our first aim is to establish the existence of an absorbing ball in $\H$ for $\S(t),  t \geq 0$. We use the compact embedding of $\V$ in $\H$  to get the  existence of a global attractor for $\S(t),  t \geq 0$ in $\H$. In the next result, we show that the operator $\S(t) : \H\to\H$, for $t\geq 0$, is Lipschitz continuous on bounded subsets of $\H$. 

  \begin{lemma}\label{lem3.1} The map
	$\S(t) : \H\to\H$, for $t\geq 0$, is Lipschitz continuous on bounded subsets of $\H$. 
\end{lemma}
\begin{proof}
	Let us take $\S(t)\u_0=\u(t)$ and $\S(t)\v_0=\v(t)$, for all $t\geq 0$. Then $\w(t):=\u(t)-\v(t)$ satisfies: 
	\begin{equation}\label{38}
	\left\{
	\begin{aligned}
	\frac{\d\w(t)}{\d t}+\mu\A\w(t)+\B(\w(t),\u(t))+\B(\v(t),\w(t))&+\alpha\w(t)\\+\beta[\mathcal{C}(\u(t))-\mathcal{C}(\v(t))]&=\mathbf{0},\\ 
	\w(0)&=\w_0\in\H,
	\end{aligned}
	\right.
	\end{equation}
in $\V'$ for a.e.  $ t\in (0,T)$. Let us  take the inner product with $\w(\cdot)$ to the first equation in \eqref{38} to find 
	\begin{align}\label{309}
	&\frac{1}{2}\frac{\d}{\d t}\|\w(t)\|_{\H}^2+\mu\|\nabla\w(t)\|_{\H}^2+\alpha\|\w(t)\|_{\H}^2+\beta\langle\mathcal{C}(\u(t))-\mathcal{C}(\v(t)),\w(t)\rangle=-b(\w(t),\u(t),\w(t)),
	\end{align}
	where we used the fact that $b(\v,\w,\w)=0$. Integrating from $0$ to $t$, we obtain 
	\begin{align}\label{39a}
	&\|\w(t)\|_{\H}^2+2\mu\int_0^t\|\w(s)\|_{\V}^2\d s+2\alpha\int_0^t\|\w(s)\|_{\H}^2\d s+2\beta\int_0^t\langle\mathcal{C}(\u(s))-\mathcal{C}(\v(s)),\w(s)\rangle\d s\nonumber\\&=\|\w_0\|_{\H}^2-2\int_0^tb(\w(s),\u(s),\w(s))\d s.
	\end{align}
	Using H\"older's and Ladyzhenskaya's inequalities, we estimate $|b(\w,\u,\w)|$ as 
	\begin{align}\label{3.10}
	|b(\w,\u,\w)|\leq \|\w\|_{\L^4}^2\|\nabla\u\|_{\H}\leq \sqrt{2}\|\w\|_{\H}\|\nabla\w\|_{\H}\|\nabla\u\|_{\H}\leq \frac{\mu}{2}\|\w\|_{\V}^2+\frac{1}{\mu}\|\u\|_{\V}^2\|\w\|_{\H}^2. 
	\end{align}
	Let us use  \eqref{3.10} and \eqref{2.23} in \eqref{39a} to deduce that 
	\begin{align}\label{311}
	&\|\w(t)\|_{\H}^2+\mu\int_0^t\|\w(s)\|_{\V}^2\d s\leq\|\w_0\|_{\H}^2+\frac{2}{\mu}\int_0^t\|\u(s)\|_{\V}^2\|\w(s)\|_{\H}^2\d s.
	\end{align}
	Applying  Gronwall's inequality in \eqref{311}, we obtain 
	\begin{align}
	\|\w(t)\|_{\H}^2\leq \|\w_0\|_{\H}^2\exp\left(\frac{2}{\mu}\int_0^t\|\u(s)\|_{\V}^2\d s\right),
	\end{align}
	and 
	\begin{align}\label{3p12}
	\|\w(t)\|_{\H}^2+\mu\int_0^t\|\w(s)\|_{\V}^2\d s\leq \|\w_0\|_{\H}^2 \exp\left(\frac{4}{\mu}\int_0^t\|\u(s)\|_{\V}^2\d s\right),
	\end{align}
for all $t\in[0,T]$. 	Thus, it is immediate that
	\begin{align}\label{313a}
	\|\S(t)\u_0-\S(t)\v_0\|_{\H}&\leq \|\u_0-\v_0\|_{\H}\exp\left(\frac{1}{\mu}\int_0^t\|\u(s)\|_{\V}^2\d s\right)\nonumber\\&\leq \|\u_0-\v_0\|_{\H}\exp\left\{\frac{1}{\mu^2}\left(\|\u_0\|_{\H}^2+\frac{t}{\mu}\|\f\|_{\V'}^2\right)\right\},
	\end{align}
	where we used \eqref{3p8}. Thus, the map
	$\S(t) : \H\to\H$, for $t\geq 0$, is Lipschitz continuous on bounded subsets of $\H$. 	
\end{proof}

\subsection{Absorbing ball in $\H$}\label{sec3.1} In this subsection, we show that the semigroup $\S(t)$	has a bounded absorbing set  in $\H$.

\iffalse 
\begin{remark}
Remember that for $0\leq \gamma<1$, $2(1-\gamma)|(\nabla\z,\nabla\w)|\leq 2|(\nabla\z,\nabla\w)|$. For any $\gamma\geq 0$, one can estimate $|(\nabla\z,\nabla\w)|$ as 
	\begin{align}
	|(\nabla\z,\nabla\w)|\leq \|\nabla\z\|_{\H}\|\nabla\w\|_{\H}\leq \frac{\mu}{2}\|\w\|_{\V}^2+\frac{1}{2\mu}\|\z\|_{\V}^2.
	\end{align}
	Thus, the above result holds true for all $0\leq \gamma<1$ and $\mu>\frac{1}{2\delta}$. 
\end{remark}
\fi 
\begin{proposition} \label{prop32}
The set 
	\begin{align}\label{314b}
\mathcal{B}_1:=\left\{\v\in\H:\|\v\|_{\H}\leq M_1\equiv \frac{1}{\mu}\sqrt{\frac{2}{\lambda_1}}\|\f\|_{\V'}\right\},
	\end{align}
	is a bounded absorbing set  in $\H$ for the semigroup $\S(t)$. That is, given a bounded set $\B\subset\H$, there exists an entering time $t_\B>0$ such that $\S(t)\B\subset\mathcal{B}_1$, for all $t\geq t_\B$. 
\end{proposition}
\begin{proof}
From \eqref{3p5}, we have 
\begin{align}\label{5p3}
&\frac{\d}{\d t}\|\u(t)\|_{\H}^2+\mu\lambda_1\|\u(t)\|_{\H}^2+2\alpha\|\u(t)\|_{\H}^2+2\beta\|\u(t)\|_{\wi\L^{r+1}}^{r+1} \leq \frac{1}{\mu}\|\f\|_{\V'}^2,
\end{align}
and an application of Gronwall's inequality yields 
\begin{align}\label{54}
\|\u(t)\|_{\H}^2\leq\|\u_0\|_{\H}^2e^{-\mu\lambda_1 t}+\frac{1}{\mu^2\lambda_1}\|\f\|_{\V'}^2, 
\end{align}
for all $t\geq 0$. Moreover, we have 
\begin{align}\label{est}
\limsup_{t\to\infty}\|\u(t)\|_{\H}^2\leq \frac{1}{\mu^2\lambda_1}\|\f\|_{\V'}^2. 
\end{align}
That is, we obtain
\begin{align}\label{est1}
\limsup_{t\to\infty}\|\u(t)\|_{\H}\leq \frac{1}{\sqrt{\mu^2\lambda_1}}\|\f\|_{\V'}. 
\end{align}
Furthermore,  it follows from \eqref{54} and \eqref{est1} that the set \eqref{314b} is absorbing in $\H$ for the semigroup $\S(t)$. Hence, the following uniform estimate is valid:
\begin{align}\label{3.8}
\|\u(t)\|_{\H}\leq M_1, \ \text{ where }\ M_1=\sqrt{\frac{2}{\mu^2\lambda_1}}\|\f\|_{\V'},
\end{align}
for $t$ large enough ($t\gg 1$ or $t\geq t_{\B}$) depending on the initial data. From \eqref{3p5}, we infer that 
\begin{align}\label{3p29}
&\mu\int_{t}^{t+\theta}\|\u(s)\|_{\V}^2\d s+2\beta\int_t^{t+\theta}\|\u(s)\|_{\wi\L^{r+1}}^{r+1}\d s\leq  \|\u(t)\|_{\H}^2+\frac{\theta}{\mu}\|\f\|_{\V'}^2,
\end{align}
for all $\theta>0$. Using \eqref{est} in \eqref{3p29}, we find 
\begin{align}\label{3p30}
&\limsup_{t\to\infty}\left[\mu\int_{t}^{t+\theta}\|\u(s)\|_{\V}^2\d s+\beta\int_t^{t+\theta}\|\u(s)\|_{\wi\L^{r+1}}^{r+1}\d s\right] \leq  \frac{1}{\mu}\left(\frac{1}{\mu\lambda_1}+\theta\right)\|\f\|_{\V'}^2.
\end{align}
For  $\u_0\in\mathcal{B}_1$, we further have 
	\begin{align}\label{3p32}
\sup_{t\geq 0}\sup_{\u_0\in\mathcal{B}_1}\|\u(t)\|_{\H}^2 \leq \frac{3}{\mu^2\lambda_1}\|\f\|_{\V'}^2=:\widetilde{M}_1^2. 
\end{align}
and 
\begin{align}\label{3.30}
\int_{t}^{t+\theta}\left[\mu\|\u(s)\|_{\V}^2+\beta\|\u(s)\|_{\wi\L^{r+1}}^{r+1}\right]\d s\leq  \frac{1}{\mu}\left(\frac{3}{\mu\lambda_1}+\theta\right)\|\f\|_{\V'}^2,
\end{align}
for all $t\geq 0$ and $\theta>0$. 
\end{proof} 

\subsection{Absorbing ball in $\mathbb{V}$}\label{sec3.2} Let us now show that the semigroup $\S(t)$ has an absorbing ball in $\V$. The following Proposition provided the existence of absorbing ball in $\V$ for $r\in[1,3]$.

\begin{proposition}\label{prop3.2}
	For $\f\in\H$ and $r\in[1,3]$, we have 
	\begin{align}\label{3.36}
	\|\u(t)\|_{\V}\leq  M_2,
	\end{align}
 for all $ t\geq t_{\B}+\theta$, where $\theta>0$ and $t_{\B}$ depends on the initial data given in Proposition \ref{prop32}. 
\end{proposition}

\begin{proof}
	Let us take the inner product with $\A\u(\cdot)$ the first equation in \eqref{kvf}   to obtain 
	\begin{align}\label{3p27}
	&\frac{1}{2}\frac{\d }{\d t}\|\nabla \u(t)\|_{\H}^2+\mu\|\A\u(t)\|_{\H}^2+\alpha\|\nabla\u(t)\|_{\H}^2\nonumber\\&= -(\B(\u(t)),\A\u(t))-(\mathcal{C}(\u(t)),\A\u(t))+(\f,\A\u(t)).
	\end{align}
	We estimate $|(\f,\A\u(t))|$, using Cauchy-Schwarz and Young's inequalities as 
	\begin{align}\label{3p45}
	|(\f,\A\u)|\leq \|\f\|_{\H}\|\A\u\|_{\H}\leq \frac{\mu}{4}\|\A\u\|_{\H}^2+\frac{1}{\mu}\|\f\|^2_{\H}.
	\end{align}
	Using H\"older's, Agmon's and Young's inequalities, we estimate $|(\B(\u),\A\u)|$ as 
	\begin{align}\label{3a29}
	|(\B(\u),\A\u)|&\leq \|\u\|_{\L^{\infty}}\|\u\|_{\V}\|\A\u\|_{\H}\leq C\|\u\|_{\H}^{1/2}\|\u\|_{\V}\|\A\u\|_{\H}^{3/2} \nonumber\\&\leq \frac{\mu}{4}\|\A\u\|_{\H}^2+\frac{C}{\mu^3}\|\u\|_{\H}^2\|\u\|_{\V}^4.
	\end{align}
 For $r\in[1,3]$, we estimate $\beta|(\mathcal{C}(\u),\A\u)|$ using H\"older's, Gagliardo-Nirenberg's  and Young's inequalities as 
	\begin{align}\label{327}
\beta	|(\mathcal{C}(\u),\A\u)|&\leq\beta\|\mathcal{C}(\u)\|_{\H}\|\A\u\|_{\H}\leq\beta\|\u\|_{\wi\L^{2r}}^{r}\|\A\u\|_{\H}\leq C\beta\|\u\|_{\H}\|\u\|_{\V}^{r-1}\|\A\u\|_{\H}\nonumber\\&\leq\frac{\mu}{4}\|\A\u\|_{\H}^2+\frac{C\beta^2}{\mu}\|\u\|_{\H}^2\|\u\|_{\V}^{2r-2}. 
	\end{align}
	Note that $|\langle\f,\u\rangle|\leq \|\f\|_{\H}\|\u\|_{\H}\leq M_1\|\f\|_{\H},$ for all $t\geq t_{\B}$.  From \eqref{3p5}, we obtain 
	\begin{align}\label{3.45}
	\mu\int_0^t\|\u(s)\|_{\V}^2\d s&\leq\|\u_0\|_{\H}^2+2 \|\f\|_{\H}\|\u(t)\|_{\H}\leq \|\u_0\|_{\H}^2+2M_1\|\f\|_{\H}\nonumber\\&\leq M_1(M_1+2\|\f\|_{\H}),
	\end{align}
for $\u_0\in\mathcal{B}_1$ (otherwise on can keep $\|\u_0\|_{\H}^2$ as such in \eqref{3.45} and take $t\geq t_{\B}$). 

Moreover, multiplying  \eqref{3p27} by $t$, using \eqref{3p45}-\eqref{327}, we also infer that 
	\begin{align}\label{346}
	&\frac{\d }{\d t}\left[t\|\u(t)\|_{\V}^2\right]+\frac{\mu t}{2}\|\A\u(t)\|_{\H}^2+2\alpha t\|\u(t)\|_{\V}^2\nonumber\\&\leq \frac{2t}{\mu}\|\f\|^2_{\H}+\|\u(t)\|_{\V}^2+ \frac{CM_1^2}{\mu^3}t\|\u(t)\|_{\V}^4+\frac{C\beta^2}{\mu}t\|\u(t)\|_{\H}^2\|\u(t)\|_{\V}^{2r-2},
	\end{align}
	for $r\in[1,3]$. 
	Integrating the above inequality from $0$ to $t$, we find 
	\begin{align}\label{3a43}
&	t\|\u(t)\|_{\V}^2+\mu\int_0^ts\|\A\u(s)\|_{\H}^2\d s+2\alpha\int_0^ts\|\u(s)\|_{\V}^2\d s\nonumber\\&\leq \frac{t^2}{\mu}\|\f\|_{\H}^2+\int_0^t\|\u(s)\|_{\V}^2\d s +\frac{CM_1^2}{\mu^3}\int_0^ts\|\u(s)\|_{\V}^4\d s+\frac{C\beta^2 M_1^2}{\mu}\int_0^ts\|\u(s)\|_{\V}^{2r-2}\d s. 
	\end{align}
We first consider the case $r=3$. For $r=3$, applying Gronwall's inequality in \eqref{3a43}, we further have 
	\begin{align}
t\|\u(t)\|_{\V}^2&\leq\left(\frac{t^2}{\mu}\|\f\|_{\H}^2+\int_0^t\|\u(s)\|_{\V}^2\d s\right)\exp\left\{\left(\frac{CM_1^2}{\mu^3}+\frac{C\beta^2 M_1^{2}}{\mu}\right)\int_0^t\|\u(s)\|_{\V}^2\d s\right\}\nonumber\\&\leq  \left(\frac{t^2}{\mu}\|\f\|_{\H}^2+\frac{M_1}{\mu}(M_1+2\|\f\|_{\H})\right)\exp\left\{\frac{CM_1^2}{\mu^2}(M_1+2\|\f\|_{\H})\left(\frac{1}{\mu^2}+\beta^2\right)\right\},
	\end{align}
	where we used \eqref{3.45}. It can be further deduced that 
	\begin{align}\label{3p48}
&	\sup_{0<\e<t\leq 1} \|\u(t)\|_{\V}^2\leq  \left(\frac{1}{\mu}\|\f\|_{\H}^2+\frac{M_1}{\mu}(M_1+2\|\f\|_{\H})\right)e^{\left\{\frac{CM_1^2}{\mu^2}(M_1+2\|\f\|_{\H})\left(\frac{1}{\mu^2}+\beta^2\right)\right\}}.
	\end{align}
	Since our problem \eqref{kvf} is invariant under $t$-translations, from \eqref{3p48}, we get a uniform estimate for $\|\u(t)\|_{\V}^2$, $t>\gamma\e>0$, that is, 
	\begin{align}\label{334}
&	\sup_{0<\gamma\e<t<\infty}\|\u(t)\|_{\V}^2 \leq \left(\frac{1}{\mu}\|\f\|_{\H}^2+\frac{M_1}{\mu}(M_1+2\|\f\|_{\H})\right)e^{\left\{\frac{CM_1^2}{\mu^2}(M_1+2\|\f\|_{\H})\left(\frac{1}{\mu^2}+\beta^2\right)\right\}},
	\end{align}
for $0<\e\leq 1$, and hence the estimate \eqref{3.36} follows.  The cases of $r=1,2$ can be estimated in a similar way.
	\end{proof}

\iffalse 
\begin{remark}\label{rem3.3}
1. From Proposition \ref{prop3.2}, it is clear that for any $\u_0\in\H$, 
	\begin{align}
	\|\S(t)\u_0\|_{\V}\leq M_2, \ \text{ for all }\ t\geq t_{\B}+r,
	\end{align}
	for all $r>0$. We denote $\mathcal{B}_2$ as the absorbing ball of radius $M_2$ in $\V$. 
	
	2. Using \eqref{3p45}-\eqref{327} in \eqref{3p27} and then integrating from $t$ to $t+\theta$, we find 
	\begin{align}
&	\|\u(t+\theta)\|_{\V}^2+\frac{\mu}{2}\int_t^{t+\theta}\|\A\u(s)\|_{\H}^2\d s\nonumber\\&\leq\|\u(t)\|_{\V}^2+\frac{2\theta}{\mu}\|\f\|_{\H}^2+\frac{C}{\mu^3}\int_t^{t+\theta}\|\u(s)\|_{\H}^2\|\u(s)\|_{\V}^4\d s+\frac{C\beta^2}{\mu}\int_t^{t+\theta}\|\u(s)\|_{\H}^{2r-2}\|\u(s)\|_{\V}^2\d s\nonumber\\&\leq \left(M_2^2+\frac{2\theta}{\mu}\|\f\|_{\H}^2+\frac{CM_1^2M_2^4\theta}{\mu^3}+\frac{C\beta^2M_1^{2r-2}M_2\theta}{\mu}\right),
	\end{align}
	for all $\theta\geq 0$ and $t\geq t_{\B}+r$. 
\end{remark}

\fi

\begin{remark}\label{rem3.7}
	For $\f\in\V'$ and $r\in[1,\infty)$, one can show the existence of an absorbing ball in $\V$ in the following way. These type of estimates for the 2D Navier-Stokes equations are due to Ladyzhenskaya (see \cite{OAL2}).	Taking the inner product with $\u_t(\cdot)$ to the first equation in \eqref{kvf}, we find 
	\begin{align}\label{244}
&	\|\u_t(t)\|_{\H}^2+\frac{\mu}{2}\frac{\d}{\d t}\|\nabla\u(t)\|_{\H}^2+\frac{\alpha}{2}\frac{\d}{\d t}\|\u(t)\|_{\H}^2\nonumber\\&=-b(\u(t),\u(t),\u_t(t))-\beta\langle\mathcal{C}(\u(t)),\u_t(t)\rangle+\langle\f,\u_t(t)\rangle.
	\end{align}
	We estimate the three terms on the right hand side of the equality \eqref{244} as 
	\begin{align*}
	-b(\u,\u,\u_t)&=b(\u,\u_t,\u)\leq\|\u\|_{\L^4}^2\|\nabla\u_t\|_{\H},\\
	|\langle\mathcal{C}(\u),\u_t\rangle|&\leq\||\u|^{\frac{r-1}{2}}\u_t\|_{\H}\|\u\|_{\H},\\
	|\langle\f,\u_t\rangle|&\leq\|\f\|_{\V'}\|\nabla\u_t\|_{\H}.
	\end{align*}
	Using the above estimates in \eqref{244}, we obtain 
	\begin{align}\label{386}
	&\|\u_t(t)\|_{\H}^2+\frac{\mu}{2}\frac{\d}{\d t}\|\nabla\u(t)\|_{\H}^2+\frac{\alpha}{2}\frac{\d}{\d t}\|\u(t)\|_{\H}^2\nonumber\\&\leq\|\u(t)\|_{\L^4}^2\|\nabla\u_t(t)\|_{\H}+\beta\||\u(t)|^{\frac{r-1}{2}}\u_t(t)\|_{\H}\|\u(t)\|_{\H}+\|\f\|_{\V'}\|\nabla\u_t(t)\|_{\H},
	\end{align}
	for $r\in[1,\infty)$.

Let us now differentiate \eqref{kvf} with respect to $t$ again to find 
	\begin{equation}\label{387}
	\u_{tt}(t)+\mu\A\u_t(t)+\alpha\u_t(t)=-\B(\u(t),\u_t(t))-\B(\u_t(t),\u(t))-\beta\mathcal{C}'(\u(t))\u_t(t),
	\end{equation}
	for all $t\in(0,T)$. Taking the inner product with $\u_t(\cdot)$ to the equation \eqref{387}, we get 
	\begin{align}\label{388}
	\frac{1}{2}\frac{\d}{\d t}\|\u_t(t)\|_{\H}^2+\mu\|\nabla\u_t(t)\|_{\H}^2+\alpha\|\u_t(t)\|_{\H}^2+\beta\||\u(t)|^{\frac{r-1}{2}}\u_t(t)\|_{\H}^2=-b(\u_t(t),\u(t),\u_t(t)). 
	\end{align}
	Using H\"older's, Ladyzhenskaya's and Young's inequalities, we estimate $b(\u_t,\u,\u_t)$ as
	\begin{align}\label{392}
	|b(\u_t,\u,\u_t)|\leq\|\nabla\u\|_{\H}\|\u_t\|_{\L^4}^2\leq\sqrt{2}\|\nabla\u\|_{\H}\|\u_t\|_{\H}\|\nabla\u_t\|_{\H}\leq\frac{\mu}{2}\|\nabla\u_t\|_{\H}^2+\frac{1}{\mu}\|\nabla\u\|_{\H}^2\|\u_t\|_{\H}^2. 
	\end{align}
	Substituting \eqref{392} in \eqref{388}, we deduce that 
	\begin{align}\label{393}
	&\frac{1}{2}\frac{\d}{\d t}\|\u_t(t)\|_{\H}^2+\frac{\mu}{2}\|\nabla\u_t(t)\|_{\H}^2+\alpha\|\u_t(t)\|_{\H}^2+\beta\||\u(t)|^{\frac{r-1}{2}}\u_t(t)\|_{\H}^2\leq\frac{1}{\mu}\|\nabla\u(t)\|_{\H}^2\|\u_t(t)\|_{\H}^2.
	\end{align}
	Using \eqref{386} and \eqref{393}, we further have 
	\begin{align}\label{394}
	&\frac{\d}{\d t}\left[t^2\|\u_t(t)\|_{\H}^2+2\mu t\|\nabla\u(t)\|_{\H}^2\right]+\mu t^2\|\nabla\u_t(t)\|_{\H}^2+2\alpha t^2\|\u_t(t)\|_{\H}^2+2\beta t^2\||\u(t)|^{\frac{r-1}{2}}\u_t(t)\|_{\H}^2\nonumber\\&\quad+4t\|\u_t(t)\|_{\H}^2\nonumber\\&\leq 2t\|\u_t(t)\|_{\H}^2+2\mu\|\nabla\u(t)\|_{\H}^2+\frac{t^2}{\mu}\|\nabla\u(t)\|_{\H}^2\|\u_t(t)\|_{\H}^2+4t\|\u(t)\|_{\L^4}^2\|\nabla\u_t(t)\|_{\H}\nonumber\\&\quad+4\beta t\||\u(t)|^{\frac{r-1}{2}}\u_t(t)\|_{\H}\|\u(t)\|_{\H}+4t\|\f\|_{\V'}\|\nabla\u_t(t)\|_{\H}.
	\end{align}
	We estimate the terms  $4t\|\u\|_{\L^4}^2\|\nabla\u_t\|_{\H}$,  $ 4 t\|\f\|_{\V'}\|\nabla\u_t\|_{\H}$ and $4\beta t\||\u|^{\frac{r-1}{2}}\u_t\|_{\H}\|\u\|_{\H}$ as follows:
	\begin{align*}
	4 t\|\u\|_{\L^4}^2\|\nabla\u_t\|_{\H}&\leq 4\sqrt{2} t \|\u\|_{\H}\|\nabla\u\|_{\H}\|\nabla\u_{t}\|_{\H}\leq\e\mu t^2\|\nabla\u_t\|_{\H}^2+\frac{8}{\e\mu}\|\u\|_{\H}^2\|\nabla\u\|_{\H}^2,\\
	4\gamma t\|\f\|_{\V'}\|\nabla\u_t\|_{\H}&\leq \e\mu t^2\|\nabla\u_t\|_{\H}^2+\frac{4}{\e\mu}\|\f\|_{\V'}^2, \\
	4\beta t\||\u|^{\frac{r-1}{2}}\u_t\|_{\H}\|\u\|_{\H}&\leq \beta t^2\||\u|^{\frac{r-1}{2}}\u_t\|_{\H}^2+4\|\u\|_{\H}^2,
	\end{align*}
	where we used Ladyzhenskaya's and Young's inequalities. Using the above estimates in \eqref{394} and choosing $\e=\frac{1}{4}$, we deduce that 
	\begin{align}\label{395}
	&\frac{\d}{\d t}\left[t^2\|\u_t(t)\|_{\H}^2+2\mu t\|\nabla\u(t)\|_{\H}^2\right]+\frac{\mu t^2}{2}\|\nabla\u_t(t)\|_{\H}^2+2\alpha t^2\|\u_t(t)\|_{\H}^2+\beta t^2\||\u(t)|^{\frac{r-1}{2}}\u_t(t)\|_{\H}^2\nonumber\\&\quad+2t\|\u_t(t)\|_{\H}^2\nonumber\\&\leq 2\mu\|\nabla\u(t)\|_{\H}^2+\frac{t^2}{\mu}\|\nabla\u(t)\|_{\H}^2\|\u_t(t)\|_{\H}^2+\frac{32}{\mu}\|\u(t)\|_{\H}^2\|\nabla\u(t)\|_{\H}^2+\frac{16}{\mu}\|\f\|_{\V'}^2+4\|\u(t)\|_{\H}^2.
	\end{align}
	Integrating the inequality \eqref{395}, we find 
	\begin{align}\label{396}
	&t^2\|\u_t(t)\|_{\H}^2+2\mu t\|\nabla\u(t)\|_{\H}^2\nonumber\\&\quad+\int_0^t\left[\frac{\mu s^2}{2}\|\nabla\u_t(s)\|_{\H}^2+\beta s^2\||\u(s)|^{\frac{r-1}{2}}\u_t(s)\|_{\H}^2+2s\|\u_t(s)\|_{\H}^2\right]\d s\nonumber\\&\leq 2\mu\int_0^t\|\nabla\u(s)\|_{\H}^2\d s+\frac{1}{\mu}\int_0^ts^2\|\nabla\u(s)\|_{\H}^2\|\u_t(s)\|_{\H}^2\d s+\frac{32}{\mu}\int_0^t\|\u(s)\|_{\H}^2\|\nabla\u(s)\|_{\H}^2\d s\nonumber\\&\quad+\frac{16t}{\mu}\|\f\|_{\V'}^2+4\int_0^t\|\u(s)\|_{\H}^2\d s.
	\end{align}
	An application of Gronwall's inequality in \eqref{396} yields 
	\begin{align}\label{397}
	&t^2\|\u_t(t)\|_{\H}^2+2\mu t\|\nabla\u(t)\|_{\H}^2\nonumber\\&\leq \bigg\{2\mu\int_0^t\|\u(s)\|_{\V}^2\d s+\frac{16t}{\mu}\|\f\|_{\V'}^2+4\int_0^t\|\u(s)\|_{\H}^2\d s+\frac{32}{\mu}\int_0^t\|\u(s)\|_{\H}^2\|\u(s)\|_{\V}^2\d s\bigg\}\nonumber\\&\quad\times\exp\left\{\frac{1}{\mu}\int_0^t\|\u(s)\|_{\V}^2\d s\right\} .
	\end{align}
	Thus, from \eqref{397}, it is immediate that 
	\begin{align}\label{398}
	t\|\u(t)\|_{\V}^2&\leq \frac{1}{\mu} \bigg\{\left(M_1^2+\frac{t}{\mu}\|\f\|_{\V'}^2\right)+\frac{8t}{\mu}\|\f\|_{\V'}^2+2tM_1^2+\frac{16M_1^2}{\mu^2}\left(M_1^2+\frac{t}{\mu}\|\f\|_{\V'}^2\right)^2\bigg\}\nonumber\\&\quad\times\exp\left\{\frac{1}{\mu^2}\left(M_1^2+\frac{t}{\mu}\|\f\|_{\V'}^2\right)\right\} .
	\end{align}
	where we used  \eqref{3p8} for $\u_0\in\mathcal{B}_1$ and $t\geq t_{\B}$.  It can be further deduced that 
	\begin{align}\label{348}
	\sup_{0<\e<t\leq 1}\|\u(t)\|_{\V}^2& \leq \frac{1}{\mu\e} \bigg\{\left(M_1^2+\frac{1}{\mu}\|\f\|_{\V'}^2\right)+\frac{8}{\mu}\|\f\|_{\V'}^2+2M_1^2+\frac{16M_1^2}{\mu^2}\left(M_1^2+\frac{1}{\mu}\|\f\|_{\V'}^2\right)^2\bigg\}\nonumber\\&\quad\times\exp\left\{\frac{1}{\mu^2}\left(M_1^2+\frac{1}{\mu}\|\f\|_{\V'}^2\right)\right\}=:M_3,
	\end{align}
	Since our problem \eqref{kvf} is invariant under $t$-translations, from \eqref{348}, we get a uniform estimate for $\|\u(t)\|_{\V}^2$, $t>\e>0$, that is, 
	\begin{align}\label{3100}
	&\sup_{0<\e<t<\infty}\|\u(t)\|_{\V}^2 \leq M_3, \ \e\leq 1,
	\end{align}
	for all $t\geq t_{\B}+\theta,$ for some $\theta>0$. 
\end{remark}

	\begin{theorem}\label{thm3.4}
		The dynamical system associated with the 2D CBF equations \eqref{kvf} possesses an attractor $\mathscr{A}_{\mathrm{glob}}$ that is compact, connected, and global in $\H$. This attractor is a bounded set in $\V$.  Moreover, $\mathscr{A}_{\mathrm{glob}}$ 	attracts the bounded sets of $\H$ and  is also maximal among the functional 	invariant sets bounded in $\H$. 
	\end{theorem}
\begin{proof}
	Let us denote the right hand side of \eqref{3.36} as $M_2$. We then conclude that the ball $\mathcal{B}_2=\B(0,M_2)$ of $\V$ is an absorbing set in $\V$, which is compact in $\H$, for the semigroup $\S(t)$. Moreover, if $\B$ is any bounded set in $\H$, then $\S(t)\B\subset\mathcal{B}_2$, for $t\geq t_{\B}+\theta$. This shows the existence of an absorbing set	in $\V$  and also that the operators $\S(t)$ are uniformly compact, that is, for every bounded set $\B$, there exists $t_0$ which is dependent on
	$\B$ such that
$\bigcup\limits_{t\geq t_0}\S(t)\B$
	is relatively compact in $\H$. Thus, using the Theorem 1.1, Chapter I, \cite{Te2}, we obtain that the dynamical system associated with the 2D CBF equations \eqref{kvf} 	possesses an attractor $\mathscr{A}_{\mathrm{glob}}$ that is compact, connected, and global in $\H$. Also, $\mathscr{A}_{\mathrm{glob}}$	attracts the bounded sets of $\H$ and $\mathscr{A}_{\mathrm{glob}}$ is maximal among the functional	invariant sets bounded in $\H$. 
\end{proof}

\section{Global Attractor (Poincar\'e Domains)}\label{sec4}\setcounter{equation}{0} 
Let us now consider the case of Poincar\'e domains. We first show a result on the weak continuity of the semigroup $\{\S(t)\}_{t\geq 0}$. Similar results for 2D Navier-Stokes equations have been obtained in  \cite{RR} and we follow this work for our model also.
\begin{lemma}\label{lem4.1}
	Let $\{\u_0^n\}_{n\in\N}$ be a weakly convergent sequence in $\H$ converging to $\u_0\in\H$. Then 
	\begin{align}\label{3.34}
\S(t)\u_0^n&\xrightharpoonup{w} \S(t)\u_0\ \text{ in }\ \H\ \text{ and } \ t\geq 0, \\
	\S(\cdot)\u_0^n &\xrightharpoonup{w}\S(\cdot)\u_0\ \text{ in } \ \mathrm{L}^2(0,T;\V), \ \text{ for all }\ T> 0, \  \text{ and }\ \label{3.35}\\
		\S(\cdot)\u_0^n &\xrightharpoonup{w}\S(\cdot)\u_0\ \text{ in } \ \mathrm{L}^{r+1}(0,T;\wi\L^{r+1}), \ \text{ for all }\ T> 0.\label{3p36}
	\end{align}
	\end{lemma}
\begin{proof}
	Let $\u^n(t)=\S(t)\u_0^n$ and $\u(t)=\S(t)\u_0$, for $t\geq 0$. From \eqref{54} and \eqref{3p8}, we have 
	\begin{align}\label{3.33}
	\{\u_n\}_{n\in\N}\ \text{ is bounded in } \ \mathrm{L}^{\infty}(0,T;\H)\cap\mathrm{L}^2(0,T;\V)\cap\mathrm{L}^{r+1}(0,T;\wi\L^{r+1}), \ \text{ for all } \ T> 0,
	\end{align}
and the estimate 
	\begin{align}\label{3.37}
&	\sup_{0\leq t\leq T}\|\u^n(t)\|_{\H}^2+2\mu\int_0^T\|\u^n(t)\|_{\V}^2\d t+2\alpha\int_0^T\|\u^n(t)\|_{\H}^2\d t+2\beta\int_0^T\|\u^n(t)\|_{\wi\L^{r+1}}^{r+1}\d t\nonumber\\&\leq C\left(\|\u_0\|_{\H},T,\|\f\|_{\V'}\right). 
	\end{align}
	Remember that $\frac{\d{\u^n}}{\d t}=\mathbf{f}-\mu\A\u^n-\alpha\u^n-\B(\u^n)-\beta\mathcal{C}(\u^n)$ in $\V'$, for a.e. $t\in(0,T)$.   For  all $\psi\in\mathrm{L}^{r}(0,T;\V)$,  it can be easily seen that 
	\begin{align}\label{3.39}
&	\int_0^T\left|\left<\frac{\d{\u^n(t)}}{\d t},\psi(t)\right>\right|\d t\nonumber\\& \leq\mu\int_0^T|\langle\A\u^n(t),\psi(t)\rangle|\d t+\int_0^T|\langle\B(\u^n(t)),\psi(t)\rangle|\d t+\alpha\int_0^T|(\u^n(t),\psi(t))|\d t\nonumber\\&\quad+\beta\int_0^T|\langle\mathcal{C}(\u^n(t)),\psi(t)\rangle|\d t+\int_0^T|\langle\f,\psi(t)\rangle|\d t\nonumber\\&\leq\mu\int_0^T\|\u^n(t)\|_{\V}\|\psi(t)\|_{\V}\d t+\int_0^T\|\B(\u^n(t))\|_{\V'}\|\psi(t)\|_{\V}\d t+\alpha\int_0^T\|\u^n(t)\|_{\H}\|\psi(t)\|_{\H}\d t\nonumber\\&\quad+\beta\int_0^T\|\u^n(t)\|_{\wi\L^{r+1}}^r\|\psi(t)\|_{\wi\L^{r+1}}\d t+\int_0^T\|\f\|_{\V'}\|\psi(t)\|_{\V}\d t\nonumber\\&\leq\bigg\{\mu\left(\int_0^T\|\u^n(t)\|_{\V}^2\d t\right)^{1/2}+\sqrt{2}\sup_{t\in[0,T]}\|\u^n(t)\|_{\H}\left(\int_0^T\|\u_n(t)\|_{\V}^2\d t\right)^{1/2}+\sqrt{T}\|\f\|_{\V'}\nonumber\\&\qquad+C\left(\int_0^T\|\u^n(t)\|_{\H}^2\d t\right)^{1/2}\bigg\}\left(\int_0^T\|\psi(t)\|_{\V}^2\d t\right)^{1/2}\nonumber\\&\quad+C\beta\left(\int_0^T\|\u^n(t)\|_{\wi\L^{r+1}}^{r+1}\d t\right)^{\frac{r}{r+1}}\left(\int_0^T\|\psi(t)\|_{\V}^{r+1}\d t\right)^{\frac{1}{r+1}}\nonumber\\&\leq C\left(\|\u_0\|_{\H},T,\mu,\beta,\|\f\|_{\V'}\right),
	\end{align}
since $\V\subset\wi\L^{r+1}$, for all $r\geq 1$.	Thus, it is immediate that
	\begin{align}
	\left\{\frac{\d{\u^n}}{\d t}\right\}_{n\in\N} \ \text{ is bounded in }\ \mathrm{L}^{\frac{r+1}{r}}(0,T;\V'), \ \text{ for all }\ T>0.
	\end{align}
	For all $\psi\in\V$ and $0\leq t\leq t+\tau\leq T$, with $T>0$, we have 
	\begin{align}\label{335} 
	(\u^n(t+\tau)-\u^n(t),\psi)&=\int_t^{t+\tau}\left<\frac{\d\u^n}{\d s}(s), \psi\right>\d s\leq \|\psi\|_{\V}\int_t^{t+\tau}\left\|\frac{\d\u^n}{\d t}(s)\right\|_{\V'}\d s\nonumber\\&\leq \left\|\frac{\d\u^n}{\d t}\right\|_{\mathrm{L}^{\frac{r+1}{r}}(0,T;\V')}\tau^{\frac{1}{r+1}}\|\psi\|_{\V}\leq C(T)\tau^{\frac{1}{r+1}}\|\psi\|_{\V},
	\end{align}
	where $C(\cdot)$ is a positive constant independent of $n$. Since $\v=\u^n(t+\tau)-\u^n(t)\in\V$, for a.e. $t\in(0,T)$, choosing it in \eqref{335}, we obtain 
	\begin{align}
	\|\u^n(t+\tau)-\u^n(t)\|_{\H}^2\leq C(T)\tau^{\frac{1}{r+1}}\|\u^n(t+\tau)-\u^n(t)\|_{\V}.
	\end{align}
	Integrating from $0$ to $T-\tau$, we further find 
	\begin{align}
	\int_0^{T-\tau}	\|\u^n(t+\tau)-\u^n(t)\|_{\H}^2\d t&\leq C(T)\tau^{\frac{1}{r+1}}\int_0^{T-\tau}\|\u^n(t+\tau)-\u^n(t)\|_{\V}\d t\nonumber\\&\leq C(T)\tau^{\frac{1}{r+1}}(T-\tau)^{1/2}\left(\int_0^{T-\tau}\|\u^n(t+\tau)-\u^n(t)\|_{\V}^2\d t\right)^{1/2}\nonumber\\&\leq \widetilde{C}(T)\tau^{\frac{1}{r+1}},
	\end{align}
	where we used the Cauchy-Schwarz inequality and \eqref{3.33}. Also, $\widetilde{C}(T)$ is an another positive constant independent of $n$. Furthermore, we have 
		\begin{align}\label{3p43}
\lim_{\tau\to 0}\sup_n	\int_0^{T-\tau}	\|\u^n(t+\tau)-\u^n(t)\|_{\L^2(\Omega_R)}^2\d t=0,
	\end{align}
	for all $R>0$, where $\Omega_R= \Omega\cap\{x\in\R^2: |x|\leq R\}$. 
	  
	  Let us now consider a truncation function $\chi\in\C^1(\R^+)$ with $\chi(s)=1$ for $s\in[0,1]$ and $\chi(s)=0$ for $s\in[2,\infty)$. For each $R>0$, let us define $$\u^{n,R}(x):=\chi\left(\frac{|x|^2}{R^2}\right)\u^n(x), \ \text{ for }x\in\Omega_{2R}.$$ It can easily be seen from \eqref{3p43} that 
	  	\begin{align}
	  \lim_{\tau\to 0}\sup_n	\int_0^{T-\tau}	\|\u^{n,R}(t+\tau)-\u^{n,R}(t)\|_{\L^2(\Omega_{2R})\times \H^1(\Omega_{2R})}^2\d t=0,
	  \end{align}
	  for all $T,R>0$. Moreover, from \eqref{3.33},    for all $T,R>0$, we infer that 
	  	\begin{align}
	&  \{\u^{n,R}\}_{n\in\N}\ \text{ is bounded in }\ \mathrm{L}^{\infty}(0,T;\L^2(\Omega_{2R}))\cap\mathrm{L}^2(0,T;\H^1_0(\Omega_{2R})).
	  \end{align}
	  Since the injection $\H_0^1(\Omega_{2R})\subset\L^2{(\Omega_{2R})}$ is compact, we can apply Theorem 16.3, \cite{MMR} (see Theorem 13.3, \cite{Te1}) to obtain 
	    	\begin{align}\label{3.47}
	    \{\u^{n,R}\}_{n\in\N}\ \text{ is relatively compact in } \ \mathrm{L}^{2}(0,T;\L^2(\Omega_{2R})),
	    \end{align}
	    for all $T,R>0$. From \eqref{3.47}, we further infer that 
	    \begin{align}\label{3.48}
	    \{\u^{n}\}_{n\in\N}\ \text{ is relatively compact in } \ \mathrm{L}^{2}(0,T;\L^2(\Omega_{R})),
	    \end{align}
	    for all $T,R>0$. Using the estimates \eqref{3.33} and \eqref{3.48}, by a diagonal argument, we can extract a subsequence $\{\u_{n_j}\}_{j\in\N}$ of $\{\u_n\}_{n\in\N}$ such that 
	    \begin{equation}\label{3.49}
	    \left\{
	    \begin{aligned}
	    \u_{n_j}&\xrightharpoonup{w^*}\widetilde{\u}, \ \text{   in }\ \mathrm{L}^{\infty}(0,T;\H), \\
	     \u_{n_j}&\xrightharpoonup{w}\widetilde{\u}, \ \text{  in }\ \mathrm{L}^{2}(0,T;\V),\\
	      \u_{n_j}&\xrightharpoonup{w}\widetilde{\u}, \ \text{  in }\ \mathrm{L}^{r+1}(0,T;\wi\L^{r+1}),\\
	      \u_{n_j}&\to\widetilde{\u}, \ \text{  in }\ \mathrm{L}^{2}(0,T;\L^2(\Omega_R)),
	    \end{aligned}\right.
	    \end{equation}
	    as $j\to\infty$, for all $T, R>0$, for some $$\widetilde{\u}\in \mathrm{L}^{\infty}(0,T;\H)\cap\mathrm{L}^{2}(0,T;\V)\cap \mathrm{L}^{r+1}(0,T;\wi\L^{r+1}). $$ Using the convergence given in \eqref{3.49}, one can easily pass limit in the equation for $\u^n$ to obtain that $\widetilde{\u}$ is a solution of \eqref{kvf} (in the weak sense) with $\widetilde{\u}(0)=\u_0$ (cf. \cite{KKMT}). Since the solution of \eqref{kvf} is unique, we further have $\widetilde{\u}=\u$. Uniqueness of the weak solution also suggests us that the whole sequence $\{\u^n\}_{n\in\mathbb{N}}$ converges to $\u$. From the second and third convergences given in \eqref{3.49}, we have 
	    \begin{align*}
	     \S(\cdot)\u_0^n=\u_{n}&\xrightharpoonup{w}{\u}=\S(\cdot)\u_0 \ \text{ in }\ \mathrm{L}^{2}(0,T;\V)\ \text{ and }\ \mathrm{L}^{r+1}(0,T;\wi\L^{r+1}),
	    \end{align*}
	    which proves \eqref{3.35} and \eqref{3p36}. 
	    
	    Let us now show the convergence given in \eqref{3.34}. From the final convergence given \eqref{3.49}, we immediately have $\u^n(t)$ converges strongly to $\u(t)$ in $\L^2(\Omega_R)$, for a.e. $t\geq 0$ and all $R>0$. Thus, for all $\psi\in \mathscr{V}$ and $  \text{ a.e. } t\in\R^+$, we  have  \begin{align}(\u^n(t),\psi)\to (\u(t),\psi), \  \text{ a.e. }\ t\in\R^+. \end{align} Furthermore, using \eqref{3.37} and \eqref{3.39}, one can easily show that $\{(\u^n(t),\psi)\}_{n\in\N}$ is equibounded and equicontinuous on $[0,T],$ for all $T>0$. Therefore, 
	    for all $\psi\in \mathscr{V}$, we get 
	    \begin{align}\label{3.52}
	    (\u^n(t),\psi)\to(\u(t),\psi), \ \text{ for all } \ t\in\R^+.
	    \end{align}
	    The convergence in \eqref{3.34} follows easily from \eqref{3.52}, making use of the fact that $\mathscr{V}$ is dense in $\H$. 
	\end{proof}
	
	    \subsection{Asymptotic compactness} In this subsection, we prove the existence of a global attractor for the system \eqref{kvf} in Poincar\'e domains.  This follows from the following general result. 
	    \begin{theorem}[Theorem I.1.1, \cite{Te2}, \cite{OAL1,RR}]\label{thm4.2} Let $\mathscr{E}$ be a complete metric space and let $\{\S(t)\}_{t\geq 0}$ be a semigroup of  	continuous (nonlinear) operators in $\mathscr{E}$. If (and only if) $\{\S(t)\}_{t\geq 0}$ possesses an absorbing set  	$\mathcal{B}$ bounded in $\mathscr{E}$ and is asymptotically compact in $\mathscr{E},$ then $\{\S(t)\}_{t\geq 0}$ possesses a (compact) global attractor $\mathscr{A}_{\mathrm{glob}}=\omega(\mathcal{B})$. Furthermore, if $t\mapsto\S(t)\u_0$ is continuous from $\R^+$ into $\mathscr{E}$ and $\mathcal{B}$	is connected in $\mathscr{E}$, then $\mathscr{A}_{\mathrm{glob}}$ is connected in $\mathscr{E}$. 
	    \end{theorem}
	    Note that $\mathcal{B}_1$	is a bounded absorbing set  in $\H$ for the semigroup $\S(t)$ (see Proposition \ref{prop32}). Let us now show the asymptotic compactness of the semigroup $\{\S(t)\}_{t\geq 0}$ using the energy equation \eqref{3p5}. We say that the semigroup $\{\S(t)\}_{t\geq 0}$ is \emph{asymptotically compact} in a given metric space if $$\{\S(t_n)\u_n\} \ \text{ is precompact},$$ whenever  $$\{\u_n\}_n \ \text{ is bounded and }\ t_n\to\infty.$$ We follow the work \cite{RR} for getting asymptotic compactness of the semigroup $\{\S(t)\}_{t\geq 0}$. 
	    
	    Let us first define $\langle\!\langle\cdot,\cdot\rangle\!\rangle:\V\times\V\to\R$ by $$\langle\!\langle\u_1,\u_2\rangle\!\rangle=\mu[\u_1,\u_2]-\mu\frac{\lambda_1}{2}(\u_1,\u_2),\ \text{ for all } \ \u_1,\u_2\in\V,$$ where $(\cdot,\cdot)$ and $[\cdot,\cdot]$ denote the inner products in $\H$ and $\V$, respectively. Let us take $\u_1=\u_2=\u$ in the above expression to find 
	    \begin{align*}
	    \langle\!\langle\u\rangle\!\rangle^2=\langle\!\langle\u,\u\rangle\!\rangle=\mu\|\u\|_{\V}^2-\mu\frac{\lambda_1}{2}\|\u\|_{\H}^2\geq\mu\|\u\|_{\V}^2-\frac{\mu}{2}\|\u\|_{\V}^2=\frac{\mu}{2}\|\u\|_{\V}^2. 
	    \end{align*}
	    Moreover, we have 
	    \begin{align*}
	    \frac{\mu}{2}\|\u\|_{\V}^2\leq\langle\!\langle\u\rangle\!\rangle^2\leq\mu\|\u\|_{\V}^2, \ \text{ for all }\ \u\in\V,
	    \end{align*}
	    and hence $\langle\!\langle\cdot,\cdot\rangle\!\rangle$ defines an  inner product in $\V$ with norm $\langle\!\langle\cdot\rangle\!\rangle=\langle\!\langle\cdot,\cdot\rangle\!\rangle^{1/2}$, and is equivalent to $\|\cdot\|_{\V}$. 
	    \begin{proposition}\label{prop4.4}
	    	The semigroup $\{\S(t)\}_{t\geq 0}$ is asymptotically compact in $\H$. 
	    	\end{proposition}
    	\begin{proof}
    		Let us add and subtract $\mu\lambda_1\|\u\|_{\H}^2$ in \eqref{3p5} to find  
    		\begin{align}\label{41}
    		\frac{\d}{\d t}\|\u(t)\|_{\H}^2+\mu\lambda_1\|\u(t)\|_{\H}^2+2\langle\!\langle\u(t)\rangle\!\rangle^2+2\alpha\|\u(t)\|_{\H}^2+2\beta\|\u(t)\|_{\wi\L^{r+1}}^{r+1} =2\langle \f,\u(t)\rangle,
    		\end{align}
    		for any solution $\u(t)=\S(t)\u_0$, $\u_0\in\H$. Then, using the variation of constant formula, we obtain 
    		\begin{align}\label{42}
    		\|\u(t)\|_{\H}^2= \|\u_0\|_{\H}^2e^{-\mu\lambda_1 t}+2\int_0^te^{-\mu\lambda_1(t-s)}\left[\langle \f,\u(s)\rangle-\left(\langle\!\langle\u(s)\rangle\!\rangle^2+\alpha\|\u(t)\|_{\H}^2+\beta\|\u(s)\|_{\wi\L^{r+1}}^{r+1}\right)\right]\d s.
    		\end{align} 
    		Since $\u(t)=\S(t)\u_0$, one can rewrite \eqref{42} as 
    		\begin{align}\label{44}
    		&\|\S(t)\u_0\|_{\H}^2\nonumber\\&= \|\u_0\|_{\H}^2e^{-\mu\lambda_1 t}+2\int_0^te^{-\mu\lambda_1(t-s)}\left[\langle \f,\S(s)\u_0\rangle-\langle\!\langle\S(s)\u_0\rangle\!\rangle^2-\alpha\|\u(t)\|_{\H}^2-\beta\|\S(s)\u_0\|_{\wi\L^{r+1}}^{r+1}\right]\d s,
    		\end{align}
    		for all $\u_0\in\H$ and $t\geq 0$. 
    		
    		Let us now show the asymptotic compactness of the semigroup $\{\S(t)\}_{t\geq 0}$. Let $\B$ be a bounded subset of $\H$, and consider $\{\u_n\}_n\subset\B$ and $\{t_n\}_n$, $t_n\geq 0$, $t_n\to\infty$. Since $\mathcal{B}_1$ defined in \eqref{314b} is absorbing, there exists a time $t_{\B}>0$ such that $\S(t)\B\subset\mathcal{B}_1, \ \text{ for all }\ t\geq t_\B,$ so that for large enough $t_n$, say $t_n\geq t_{\B},$ $\S(t_n)\u_n\in\mathcal{B}_1.$ Hence, the sequence $\{\S(t_{n_k})\u_{n_k}\}_{n_k}$ is weakly precompact in $\H$. Thus, since $\mathcal{B}_1$ is closed and convex, we have  \begin{align}\label{4p5}\{\S(t_{n_k})\u_{n_k}\}_{n_k}\xrightharpoonup{w}\W\ \text{  in } \ \H,\end{align} for some subsequence $\{\S(t_{n_k})\u_{n_k}\}_{n_k}$ of $\{\S(t_{n_k})\u_n\}_n$ and $\W\in\mathcal{B}_1$. Similarly, for each $T>0$, one can show that $\S(t_n-T)\u_n\in\mathcal{B}_1,$ for all $t_n\geq T+t_{\B}$. Thus, we obtain $\S(t_n-T)\u_n$ is precompact in $\H$, and by using a diagonal argument and passing to a further subsequence (if necessary), we can assume that $$\{\S(t_{n_k}-T)\u_{n_k}\}_{n_k}\xrightharpoonup{w}\W_T, \ \text{ in } \ \H,$$ for all $T\in\mathbb{N}$ with $\W_T\in\mathcal{B}_1$. From \eqref{4p5}, we know that $(\S(t_{n_k})\u_{n_k},\psi)\to(\W,\psi)$, for all $\psi\in\H$. Using the weak continuity of $\S(t)$ established in Lemma \ref{lem4.1} (see \eqref{3.34}),  we have 
    		\begin{align*}
    		\W&=\lim_{k\to\infty}(\S(t_{n_k})\u_{n_k},\v)=\lim_{k\to\infty}(\S(T)\S(t_{n_k}-T)\u_{n_k},\v)\nonumber\\&=(\S(T)\lim_{k\to\infty}\S(t_{n_k}-T)\u_{n_k},\v)=\S(T)\W_T.
    		\end{align*}
    		Hence, we get $\W=\S(T)\W_T,$ for all $T\in\N$. Since $\{\S(t_{n_k})\u_{n_k}\}_{n_k}\xrightharpoonup{w}\W$ weakly in $\H$ and using the weakly lower-semicontinuity property of $\H$, we also have \begin{align}\label{4p6}\|\W\|_{\H}\leq \liminf\limits_{k\to\infty}\|\S(t_{n_k})\u_{n_k}\|_{\H}.\end{align} Our next aim is to show that \begin{align}\label{45}\limsup\limits_{k\to\infty}\|\S(t_{n_k})\u_{n_k}\|_{\H}\leq \|\W\|_{\H}.\end{align} For $T\in\N$ and $t_n>T$, from \eqref{44}, we have 
    		\begin{align}\label{4.5}
    	\|\S(t_n)\u_n\|_{\H}^2&=\|\S(T)\S(t_n-T)\u_n\|_{\H}^2\nonumber\\&=\|\S(t_n-T)\u_n\|_{\H}^2e^{-\mu\lambda_1T} +2\int_0^Te^{-\mu\lambda_1(T-s)}\Big[\langle \f,\S(s)\S(t_n-T)\u_n\rangle\nonumber\\&\qquad-\langle\!\langle\S(s)\S(t_n-T)\u_n\rangle\!\rangle^2-\alpha\|\S(s)\S(t_n-T)\u_n\|_{\H}^2-\beta\|\S(s)\S(t_n-T)\u_n\|_{\wi\L^{r+1}}^{r+1}\Big]\d s.
    		\end{align}
    		Since $\S(t_n-T)\u_n\in\mathcal{B}_1$, we can easily get 
    		\begin{align}\label{4.6}
    		\limsup_{k\to\infty}e^{-\mu\lambda_1T}\|\S(t_n-T)\u_n\|_{\H}^2\leq M_1^2e^{-\mu\lambda_1T}.
    		\end{align}
    		Using the weak continuity result given in \eqref{3.35}  and the convergence  $\{\S(t_{n_k}-T)\u_{n_k}\}_{n_k}\xrightharpoonup{w}\W_T$ in $\H$, we have 
    		\begin{align}\label{4.7}
    		\S(\cdot)\S(t_{n_k}-T)\u_{n_k}\xrightharpoonup{w}\S(\cdot)\W_T \ \text{ weakly in } \ \mathrm{L}^2(0,T;\V)\text{ and }\\
    		\S(\cdot)\S(t_{n_k}-T)\u_{n_k}\xrightharpoonup{w}\S(\cdot)\W_T \ \text{ weakly in } \ \mathrm{L}^{r+1}(0,T;\wi\L^{r+1}). \label{4.8}
    		\end{align}
    		Now, we consider
    		\begin{align}
    		\int_0^T\|e^{-\mu\lambda_1(T-s)}\f\|_{\V'}^2\d s =\int_0^Te^{-2\mu\lambda_1(T-s)}\|\f\|_{\V'}^2\d s=\|\f\|_{\V'}^2\left(\frac{1-e^{-2\mu\lambda_1T}}{2\mu\lambda_1}\right)\leq\frac{\|\f\|_{\V'}^2}{2\mu\lambda_1}<+\infty,
    		\end{align}
    		and hence the  mapping $s\mapsto e^{-\mu\lambda_1(T-s)}\f\in\mathrm{L}^2(0,T;\V')$. Thus, we find 
    		\begin{align}
    		\lim\limits_{k\to\infty}\int_0^Te^{-\mu\lambda_1(T-s)}\langle \f,\S(s)\S(t_{n_k}-T)\u_{n_k}\rangle\d s = \int_0^Te^{-\mu\lambda_1(T-s)}\langle \f,\S(s)\W_T\rangle\d s.
    		\end{align}
    		Furthermore, since $\langle\!\langle\cdot\rangle\!\rangle$ defines a norm on $\V$ equivalent to the norm $\|\cdot\|_{\V}$ and $0<e^{-\mu\lambda_1T}\leq e^{-\mu\lambda_1(T-s)}\leq 1,$ for all $s\in[0,T]$, one can easily see that $\left(\int_0^Te^{-\mu\lambda_1(T-s)}\langle\!\langle\cdot\rangle\!\rangle^2\d s\right)^{1/2}$ defines a norm on $\mathrm{L}^2(0,T;\V)$ equivalent to the norm $\left(\int_0^T\|\cdot\|_{\V}^2\d s\right)^{1/2}$. Using \eqref{4.7} and the weakly lower-semicontinuity property of the norm, we get 
    		\begin{align*}
    		\int_0^Te^{-\mu\lambda_1(T-s)}\langle\!\langle\S(s)\W_T\rangle\!\rangle^2\d s\leq\liminf\limits_{k\to\infty}\int_0^Te^{-\mu\lambda_1(T-s)}\langle\!\langle\S(s)\S(t_{n_k}-T)\u_{n_k}\rangle\!\rangle^2\d s.
    		\end{align*}
    		Thus, we have 
    		\begin{align}\label{411}
    		\limsup\limits_{k\to\infty}&\left[-2\int_0^Te^{-\mu\lambda_1(T-s)}(\!(\!(\S(s)\S(t_{n_k}-T)\u_{n_k})\!)\!)^2\d s\right]\nonumber\\&=-2\liminf\limits_{k\to\infty}\left[\int_0^Te^{-\mu\lambda_1(T-s)}(\!(\!(\S(s)\S(t_{n_k}-T)\u_{n_k})\!)\!)^2\d s\right]\nonumber\\&\leq-2\int_0^Te^{-\mu\lambda_1(T-s)}(\!(\!(\S(s)\W_T)\!)\!)^2\d s.
    		\end{align}
    		Note  that $0<e^{-\mu\lambda_1T}\leq e^{-\mu\lambda_1(T-s)}\leq 1,$ for all $s\in[0,T]$, and thus it is immediate that $\left(\int_0^Te^{-\mu\lambda_1(T-s)}\|\cdot\|_{\H}^{2}\d s\right)^{1/2}$ defines a norm on $\mathrm{L}^{2}(0,T;\H)$, which is equivalent to the standard norm. Using \eqref{4.8} and the weakly lower-semicontinuity property of the norm, we get 
    		\begin{align}
    		\int_0^Te^{-\mu\lambda_1(T-s)}\|\S(s)\W_T\|_{\H}^{2}\d s\leq\liminf\limits_{k\to\infty}\int_0^Te^{-\mu\lambda_1(T-s)}\|\S(s)\S(t_{n_k}-T)\u_{n_k}\|_{\H}^{2}\d s.
    		\end{align}
    		Once again using the fact that $0<e^{-\mu\lambda_1T}\leq e^{-\mu\lambda_1(T-s)}\leq 1,$ for all $s\in[0,T]$, one can easily see that $\left(\int_0^Te^{-\mu\lambda_1(T-s)}\|\cdot\|_{\wi\L^{r+1}}^{r+1}\d s\right)^{\frac{1}{r+1}}$ defines a norm on $\mathrm{L}^{r+1}(0,T;\wi\L^{r+1})$ equivalent to the norm $\left(\int_0^T\|\cdot\|_{\wi\L^{r+1}}^2\d s\right)^{\frac{1}{r+1}}$. Using \eqref{4.8} and the weakly lower-semicontinuity property of the norm, we get 
    		\begin{align}\label{4.11}
    		\int_0^Te^{-\mu\lambda_1(T-s)}\|\S(s)\W_T\|_{\wi\L^{r+1}}^{r+1}\d s\leq\liminf\limits_{k\to\infty}\int_0^Te^{-\mu\lambda_1(T-s)}\|\S(s)\S(t_{n_k}-T)\u_{n_k}\|_{\wi\L^{r+1}}^{r+1}\d s.
    		\end{align}
    		Using \eqref{4.6}-\eqref{4.11} in  \eqref{4.5} and then taking $\limsup$ in \eqref{4.5}, we obtain 
    		\begin{align}\label{412}
    	&	\limsup\limits_{k\to\infty}	\|\S(t_{n_k})\u_{n_k}\|_{\H}^2\nonumber\\&\leq {M_1^2}e^{-\mu\lambda_1T}+2\int_0^Te^{-\mu\lambda_1(T-s)}\left[\langle \f,\S(s)\W_T\rangle-\langle\!\langle\S(s)\W_T\rangle\!\rangle^2-\alpha\|\S(s)\W_T\|_{\H}^2-\beta\|\S(s)\W_T\|_{\wi\L^{r+1}}^{r+1}\right]\d s.
    		\end{align}
    		Let us apply  $\W=\S(T)\W_T$ in \eqref{44} to deduce that 
    		\begin{align}\label{413}
    		\|\W\|_{\H}^2&=\|\S(T)\W_T\|_{\H}^2=e^{-\mu\lambda_1 T}\|\W_T\|_{\H}^2\nonumber\\&\quad+2\int_0^Te^{-\mu\lambda_1(T-s)}\left[\langle \f,\S(s)\W_T\rangle-\langle\!\langle\S(s)\W_T\rangle\!\rangle^2-\alpha\|\S(s)\W_T\|_{\H}^2-\beta\|\S(s)\W_T\|_{\wi\L^{r+1}}^{r+1}\right]\d s.
    		\end{align}
    		Combining \eqref{412} and \eqref{413}, we infer that 
    		\begin{align}\label{414}
    			\limsup\limits_{k\to\infty}\|\S(t_{n_k})\u_{n_k}\|_{\H}^2&\leq\|\W\|_{\H}^2+\left(M_1^2-\|\W_T\|_{\H}^2\right)e^{-\mu\lambda_1T}\leq\|\W\|_{\H}^2+M_1^2e^{-\mu\lambda_1 T},
    		\end{align}
    		for all $T\in\N$. Let us take $T\to\infty$ in \eqref{414} to find 
    		\begin{align}\label{417}
    		\limsup\limits_{k\to\infty}\|\S(t_{n_k})\u_{n_k}\|_{\H}^2\leq\|\W\|_{\H}^2,
    		\end{align}
    		and \eqref{45} follows. Since $\H$ is a Hilbert space, \eqref{4p6}-\eqref{45} along with \eqref{4p5} imply 
    		\begin{align}\label{419}\{\S(t_{n_k})\u_{n_k}\}_{n_k}\to\W\ \text{ strongly in } \ \H,\end{align}
    		as $k\to\infty$. This establishes that $\{\S(t_{n})\u_{n}\}_{n}$ is precompact in $\H$ and hence that $\{\S(t)\}_{t\geq 0}$ is asymptotically	compact in $\H$. 
    	\end{proof}
	    Using the asymptotic compactness of the semigroup $\{\S(t)\}_{t\geq 0}$ and Theorem \ref{thm4.2}, we have the following result:
	    \begin{theorem}\label{thm4.3}
	    	Let $\Omega$ be an open Poincar\'e domain. Assume that $\mu>0$ and $\f\in\V'$. Then the semigroup $\{\S(t)\}_{t\geq 0}$ associated with the 2D CBF system \eqref{kvf} possesses a global attractor $\mathscr{A}_{\mathrm{glob}}$ in $\H$, that is, a compact invariant set  in $\H$, which attracts all bounded sets 	in $\H$. Moreover, $\mathscr{A}_{\mathrm{glob}}$ is connected in $\H$ and is maximal for the inclusion relation among all the functional invariant sets bounded in $\H$. 
	    \end{theorem}
    \begin{remark}
    For $\f\in\V'$, as we have proved in Remark \ref{rem3.7}, one can show that the global attractor $\mathscr{A}_{\mathrm{glob}}$ obtained in  Theorem \ref{thm4.3} is  included and bounded in $\V$. 
    \end{remark}

	    \section{Dimension of the Attractor}\label{sec5}\setcounter{equation}{0} In this section, we analyze the dimension of the global attractor $\mathscr{A}_{\mathrm{glob}}$ obtained in section \ref{sec4} (Poincar\'e domains). We estimate the bounds for the Hausdorff as well as fractal dimensions of the global attractor  $\mathscr{A}_{\mathrm{glob}}$.  We are able to prove the Fr\'echet
	    differentiability of the map $\S(t) : \H \to \H$ with respect to the initial data for $r=1,2,3$ only.  Due to this technical difficulty, we consider the cases $r=1,2$ and $3$ only in this section. 
	    
	    Let $\u_0\in\H$ and $\u(t)=\S(t)\u_0$, for $t\geq 0$ be the unique weak solution of the system \eqref{kvf}. The linearized flow around $\u(\cdot)$ is given by the following equation: 
	  	\begin{equation}\label{5.1}
	  \left\{
	  \begin{aligned}
	  \frac{\d}{\d t}\xi(t)+\mu\A\xi(t)+\B(\xi(t),\u(t))+\B(\u(t),\xi(t))+\alpha\xi(t)+\beta\mathcal{C}'(\u(t))\xi(t)&=\mathbf{0},  \\  
	  \xi(0)&=\xi_0,
	  \end{aligned}
	  \right.
	  \end{equation}
	  for a.e. $t\in(0,T)$, where $\mathcal{C}'(\cdot)$ is defined in \eqref{29}. 
	As in the case of nonlinear problem, one can show that there exists a unique solution $\xi\in\mathrm{L}^{\infty}(0,T;\H)\cap\mathrm{L}^2(0,T;\V)\cap\mathrm{L}^{r+1}(0,T;\wi\L^{r+1}),$ with $\xi_t\in\mathrm{L}^{\frac{r+1}{r}}(0,T;\V')$ for all $T>0$ satisfying the energy equality:
	\begin{align}\label{52}
&	\|\xi(t)\|_{\H}^2+2\mu\int_0^t\|\nabla\xi(s)\|_{\H}^2\d s+2\alpha\int_0^t\|\xi(s)\|_{\H}^2\d s+2\beta\int_0^t\||\u(s)|^{\frac{r-1}{2}}\xi(s)\|_{\H}^2\d s\nonumber\\&\quad+2(r-1)\beta\int_0^t\||\u(s)|^{\frac{r-3}{2}}(\u(s)\cdot\xi(s))\|_{\H}^2\d s\nonumber\\&=\|\xi_0\|_{\H}^2-2\int_0^t\langle\B(\xi(s),\u(s)),\xi(s)\rangle\d s,
	\end{align}
	for all $t\in[0,T]$ and $r\geq 3$. For $r=1,2$, one has to make proper changes in the final term appearing in the left hand side of the equality \eqref{52}, according to the definition given in \eqref{29}. 
Furthermore, using Remark \ref{rem5.10}, we infer that $\xi\in\C([0,T];\H),$ for all $T>0$. We define a map $\Lambda(t;\u_0):\H\to\H$ by setting $\Lambda(t;\u_0)\xi_0=\xi(t)$. In the next Lemma, we show that the map $\Lambda(t;\u_0)$ is bounded and the semigroup $\{\S(t)\}_{t\geq 0}$ is uniformly differentiable on $\mathscr{A}_{\mathrm{glob}}$, that is, 
	  \begin{align}\label{5.2}
	  \lim\limits_{\e\to 0}\sup\limits_{\substack{\u_0,\v_0\in\mathscr{A}_{\mathrm{glob}}\\ 0<\|\u_0-\v_0\|_{\H}\leq \e}}\frac{\|\S(t)\u_0-\S(t)\v_0-\Lambda(t;\u_0)(\v_0-\u_0)\|_{\H}}{\|\v_0-\u_0\|_{\H}}=0.
	  \end{align}
	  \begin{theorem}\label{thm5.1}
	  	Let $\u_0$ and $\v_0$ be two members of $\H$. Then for $r=1,2,3,$ there exists a constant $K=K(\|\u_0\|_{\H},\|\v_0\|_{\H})$ such that 
	  	\begin{align}\label{53}
	  	\|\S(t)\v_0-\S(t)\u_0-\Lambda(t;\u_0)(\v_0-\u_0)\|_{\H}\leq K\|\v_0-\u_0\|_{\H}
	  	,	\end{align}
	  	where the linear operator $\Lambda(t;\u_0):\H\to\H$, for $t>0$ is the solution operator of the problem \eqref{5.1} with $\u(t)=\S(t)\u_0$. Or in other words,  for every $t > 0$, the map $\S(t)\u_0$, as a map $\S(t) : \H \to \H$ is Fr\'echet
	  	differentiable with respect to the initial data, and its Fr\'echet derivative $\D_{\u_0}
	  	(\S(t)\u_0)\xi_0 =\Lambda(t;\u_0)\xi_0$. Moreover, \eqref{5.2} is satisfied. 
	  \end{theorem}
	  \begin{proof}
	  Let us define $$\eta(t):=\u(t)-\v(t)-\xi(t)=\S(t)(\u_0-\v_0)-\xi(t).$$  Then $\eta(t)$ satisfies: 
	  	\begin{equation}\label{5.3}
	  	\left\{
	  	\begin{aligned}
	  	\frac{\d}{\d t}\eta(t)+\mu\A\eta(t)+\B(\eta(t),\u(t))+\B(\u(t),\eta(t))&-\B(\w(t),\w(t))+\alpha\eta(t)\\+\beta[\mathcal{C}(\u(t))-\mathcal{C}(\v(t))-\mathcal{C}'(\u(t))\xi(t)]&=\mathbf{0}, \\
	  	\eta(0)&=\mathbf{0},
	  	\end{aligned}
	  	\right.
	  	\end{equation}
	  for a.e. $t\in(0,T)$ in $\V'$,	where $\w(t)=\u(t)-\v(t)$. Let us take the inner product with $\eta(\cdot)$ to the first equation in \eqref{5.3} to obtain 
	  	\begin{align}\label{5.4}
	  	&\frac{1}{2}\frac{\d }{\d t}\|\eta(t)\|_{\H}^2+\mu\|\eta(t)\|_{\V}^2+\alpha\|\eta(t)\|_{\H}^2\nonumber \\&= -b(\eta(t),\u(t),\eta(t))+b(\w(t),\w(t),\eta(t))-\beta\langle\mathcal{C}(\u(t))-\mathcal{C}(\v(t))-\mathcal{C}'(\u(t))\xi(t),\eta(t)\rangle.
	  	\end{align}
	  	We estimate $|b(\eta,\u,\eta)|$ using H\"older's, Ladyzhenskaya's and Young's inequalities as 
	  	\begin{align*}
	  	|b(\eta,\u,\eta)|&\leq \|\u\|_{\V}\|\eta\|_{\L^4}^2\leq \sqrt{2}\|\u\|_{\V}\|\eta\|_{\H}\|\eta\|_{\V}\leq \frac{\mu}{4}\|\eta\|_{\V}^2+\frac{2\|\u\|_{\V}^2}{\mu}\|\eta\|_{\H}^2,\\
	  	|b(\w,\w,\eta)|&=|b(\w,\eta,\w)|\leq \|\eta\|_{\V}\|\w\|_{\L^4}^2\leq\sqrt{2}\|\eta\|_{\V}\|\w\|_{\H}\|\w\|_{\V}\leq\frac{\mu}{4}\|\eta\|_{\V}^2+\frac{2}{\mu}\|\w\|_{\H}^2\|\w\|_{\V}^2.
	  	\end{align*}
	  	In order  to estimate the term $-\beta\langle\mathcal{C}(\u)-\mathcal{C}(\v)-\mathcal{C}'(\u)\xi,\eta\rangle$, we consider the cases $r=1$, $r=2$ and $r= 3$ separately. For $r=1$, its can be easily seen that 
	  	\begin{align*}
	  	-\beta\langle\mathcal{C}(\u)-\mathcal{C}(\v)-\mathcal{C}'(\u)\xi,\eta\rangle=-\beta\|\eta\|_{\H}^2. 
	  	\end{align*}
	  	It should be noted that 
	  	\begin{align*}
	&  	\left<\frac{\u_1}{|\u_1|}(\u_1\cdot\v)-\frac{\u_2}{|\u_2|}(\u_2\cdot\v),\w\right>\nonumber\\&=\left<\frac{\u_1}{|\u_1|}((\u_1-\u_2)\cdot\v),\w\right>+\left<\left(\frac{\u_1}{|\u_1|}-\frac{\u_2}{|\u_2|}\right)(\u_2\cdot\v),\w\right>\nonumber\\&= \left<\frac{\u_1}{|\u_1|}((\u_1-\u_2)\cdot\v),\w\right>+\left<\frac{\u_1(|\u_1|-|\u_2|)+(\u_1-\u_2)|\u_1|}{|\u_1||\u_2|}(\u_2\cdot\v),\w\right>\nonumber\\&\leq 2\langle|\u_1-\u_2||\v|,|\w|\rangle,
	  	\end{align*}
for all $\u_1\neq \mathbf{0},\u_2\neq\mathbf{0},\v,\w\in\wi\L^{3}$ (one can also obtain same estimates for $\u_1=\mathbf{0}$ or $\u_2=\mathbf{0}$).	  	For $r=2$, using \eqref{2.9}, Taylor's formula (see Theorem 7.9.1, \cite{PGC}), H\"older's, Ladyzhenskaya's and Young's inequalities, we obtain 
	  	\begin{align*}
	  	-&\beta\langle\mathcal{C}(\u)-\mathcal{C}(\v)-\mathcal{C}'(\u)\xi,\eta\rangle\nonumber\\&=-\beta\left<\int_0^1\mathcal{C}'(\theta\u+(1-\theta)\v)\d\theta(\u-\v)-\mathcal{C}'(\u)\xi,\eta\right>\nonumber\\&=-\beta\langle\mathcal{C}'(\u)\eta,\eta\rangle+\beta\left<\int_0^1[\mathcal{C}'(\u)-\mathcal{C}'(\theta\u+(1-\theta)\v)](\u-\v)\d\theta,\eta\right>\nonumber\\&=-\beta\langle\mathcal{C}'(\u)\eta,\eta\rangle+\beta\int_0^1\langle[|\u|-|\theta\u+(1-\theta)\v|](\u-\v),\eta\rangle\d\theta\nonumber\\&\quad+\beta\int_0^1\left<\left[\frac{\u}{|\u|}(\u\cdot(\u-\v))\right]-\left[\frac{\theta\u+(1-\theta\v)}{|\theta\u+(1-\theta\v)|}((\theta\u+(1-\theta\v))\cdot(\u-\v))\right],\eta\right>\d\theta\nonumber\\&\leq 3\beta\int_0^1(1-\theta)\langle|\u-\v|^2,|\eta|\rangle\d\theta\leq \frac{3\beta}{2}\|\u-\v\|_{\wi\L^4}^2\|\eta\|_{\H}\nonumber\\&\leq\frac{\mu}{2}\|\eta\|_{\H}^2+\frac{9\beta^2}{8\mu}\|\w\|_{\H}^2\|\w\|_{\V}^2.
	  	\end{align*}
	  	For the case $r=3$, once again using \eqref{2.9}, Taylor's formula, H\"older's, Ladyzhenskaya's and Young's inequalities, we find 
	  	\begin{align*}
	  -&\beta\langle\mathcal{C}(\u)-\mathcal{C}(\v)-\mathcal{C}'(\u)\xi,\eta\rangle\nonumber\\&=-\beta\left<\mathcal{C}'(\u)(\u-\v)+\frac{1}{2}\int_0^1\mathcal{C}''(\theta\u+(1-\theta)\v)\d\theta(\u-\v)\otimes(\u-\v)-\mathcal{C}'(\u)\xi,\eta\right>\nonumber\\&=-\beta\langle\mathcal{C}'(\u)\eta,\eta\rangle+3\beta\int_0^1\langle[(\theta\u+(1-\theta)\v)\cdot(\u-\v)](\u-\v),\eta\rangle\d\theta\nonumber\\&\quad +\frac{3\beta}{2}\int_0^1\langle|\u-\v|^2(\theta\u+(1-\theta)\v),\eta\rangle\d\theta\nonumber\\&\leq \frac{9\beta}{2}\int_0^1\|\theta\u+(1-\theta)\v\|_{\wi\L^{4}}\|\u-\v\|_{\wi\L^{4}}^{2}\|\eta\|_{\wi\L^{4}}\d\theta\nonumber\\&\leq \frac{9\beta}{2^{1/4}}\left(\|\u\|_{\wi\L^{4}}+\|\v\|_{\wi\L^{4}}\right)\|\w\|_{\H}\|\w\|_{\V}\|\eta\|_{\H}^{1/2}\|\eta\|_{\V}^{1/2}\nonumber\\&\leq \frac{\mu}{4}\|\eta\|_{\V}^2+\frac{3}{4\mu^{1/3}}\left(\frac{9\beta}{2^{1/4}}\right)^{4/3}\left(\|\u\|_{\wi\L^4}+\|\v\|_{\wi\L^4}\right)^{4/3}\|\w\|_{\H}^{4/3}\|\w\|_{\V}^{4/3}\|\eta\|_{\H}^{2/3}\nonumber\\&\leq \frac{\mu}{4}\|\eta\|_{\V}^2+\frac{\mu}{16}\left(\|\u\|_{\wi\L^4}+\|\v\|_{\wi\L^4}\right)^4\|\eta\|_{\H}^2+\frac{27\beta^2}{\mu}\|\w\|_{\H}^2\|\w\|_{\V}^2.
	  	\end{align*}
	  	Applying the above estimates  in \eqref{5.4}, we find
	  	\begin{align}\label{5.7}
	  	&\frac{\d }{\d t}\|\eta(t)\|_{\H}^2+\frac{\mu}{2}\|\eta(t)\|_{\V}^2+2\alpha\|\eta(t)\|_{\H}^2\nonumber\\&\leq\left\{\begin{array}{l}\frac{4\|\u(t)\|_{\V}^2}{\mu}\|\eta(t)\|_{\H}^2+\frac{4}{\mu}\|\w(t)\|_{\H}^2\|\w(t)\|_{\V}^2, \\ \frac{4\|\u(t)\|_{\V}^2}{\mu}\|\eta(t)\|_{\H}^2+\frac{2}{\mu}\left(2+\frac{9}{8}\beta^2\right)\|\w(t)\|_{\H}^2\|\w(t)\|_{\V}^2+\mu\|\eta(t)\|_{\H}^2, \\ \frac{4\|\u(t)\|_{\V}^2}{\mu}\|\eta(t)\|_{\H}^2+\frac{2}{\mu}(2+27\beta^2)\|\w(t)\|_{\H}^2\|\w(t)\|_{\V}^2+\mu\left(\|\u(t)\|_{\wi\L^4}^4+\|\v(t)\|_{\wi\L^4}^4\right)\|\eta(t)\|_{\H}^2,  
	  	 \end{array}\right.
	  	\end{align} 
	  	for $r=1,2$ and $3$, respectively. 
	  	From the estimate \eqref{311} (see Lemma \ref{lem3.1}), we find 
	  	\begin{align}\label{518}
	  	\|\w(t)\|_{\H}^2+\mu\int_0^t\|\w(s)\|_{\V}^2\d s&\leq\|\w_0\|_{\H}^2\exp\left\{\frac{4}{\mu^2}\left(\|\u_0\|_{\H}^2+\frac{t}{\mu}\|\f\|_{\V'}^2\right)\right\},
	  	\end{align}
	 for all $t\in[0,T]$, where we used \eqref{3p8}. 
	  	  For $r=3$,	using \eqref{518} in \eqref{5.7} and then integrating from $0$ to $t$, we find 
	  	\begin{align}\label{5.15}
	  	&\|\eta(t)\|_{\H}^2+\frac{\mu}{2}\int_0^t\|\eta(s)\|_{\V}^2\d s+2\alpha\int_0^t\|\eta(s)\|_{\H}^2\d s\nonumber\\&\leq  \frac{4}{\mu}\int_0^t\|\u(s)\|_{\V}^2\|\eta(s)\|_{\H}^2\d s+\mu\int_0^t\left(\|\u(s)\|_{\wi\L^4}^4+\|\v(s)\|_{\wi\L^4}^4\right)\|\eta(s)\|_{\H}^2\d s\nonumber\\&\quad+\frac{2}{\mu}(2+27\beta^2)\int_0^t\|\w(s)\|_{\H}^2\|\w(s)\|_{\V}^2\d s\nonumber\\&\leq  \frac{4}{\mu}\int_0^t\|\u(s)\|_{\V}^2\|\eta(s)\|_{\H}^2\d s+\mu\int_0^t\left(\|\u(s)\|_{\wi\L^4}^4+\|\v(s)\|_{\wi\L^4}^4\right)\|\eta(s)\|_{\H}^2\d s\nonumber\\&\quad+\frac{2}{\mu^2}(2+27\beta^2)\|\w_0\|_{\H}^4\exp\left\{\frac{8}{\mu^2}\left(\|\u_0\|_{\H}^2+\frac{t}{\mu}\|\f\|_{\V'}^2\right)\right\}.
	  	\end{align}
	  	An application of Gronwall's inequality in \eqref{5.15} gives
	  	\begin{align}\label{510}
	  	\|\eta(t)\|_{\H}^2&\leq \frac{2}{\mu^2}(2+27\beta^2)\exp\left\{\frac{4}{\mu^2}\left(\|\u_0\|_{\H}^2+\frac{t}{\mu}\|\f\|_{\V'}^2\right)\right\}\|\w_0\|_{\H}^4\nonumber\\&\quad\times \exp\left(\frac{8}{\mu}\int_0^t\|\u(s)\|_{\V}^2\d s+\mu\int_0^t\left(\|\u(s)\|_{\wi\L^4}^4+\|\v(s)\|_{\wi\L^4}^4\right)\d s\right)\nonumber\\&\leq \frac{2}{\mu^2}(2+27\beta^2)\exp\left\{\left(\frac{12}{\mu^2}+\frac{\mu}{\beta}\right)\left(\|\u_0\|_{\H}^2+\|\v_0\|_{\H}^2+\frac{t}{\mu}\|\f\|_{\V'}^2\right)\right\}\|\w_0\|_{\H}^4,
	  	\end{align}
	  	where we used \eqref{3p8}. 	Thus, by the definition of $\eta$, it is immediate that 
	  	\begin{align}\label{620a}
	  	\frac{\|\u(t)-\v(t)-\xi(t)\|_{\H}}{\|\u_0-\v_0\|_{\H}}\leq \vartheta(t)\|\u_0-\v_0\|_{\H},
	  	\end{align}
	  where $\vartheta(t)=\frac{1}{\mu}\sqrt{2(2+27\beta^2)}\exp\left\{\left(\frac{6}{\mu^2}+\frac{\mu}{2\beta}\right)\left(\|\u_0\|_{\H}^2+\|\v_0\|_{\H}^2+\frac{t}{\mu}\|\f\|_{\V'}^2\right)\right\}$,	and hence the differentiability of the semigroup $\S(t)$ with respect to the initial data as well as \eqref{5.2} and \eqref{53} follows. 
  The cases of $r=1,2$ can be proved in a similar way. 
	  \end{proof}

  In the next Theorem, we show that the global attractor obtained in Theorem \ref{thm4.3} has finite Hausdorff and fractal dimensions, and we find their bounds also. 
  \begin{theorem}\label{thm52}
  For $r=1,2,3$,	the global attractor obtained in Theorem \ref{thm4.3} has finite Hausdorff and fractal dimensions, which can be estimated as 
  	 \begin{align}\label{515}
  	\mathrm{dim}^{\H}_{\mathcal{H}}(\mathscr{A}_{\mathrm{glob}})\leq  1+ \frac{\widetilde{\kappa}\|\f\|_{\V'}^2}{\mu^4\lambda_1}
  	\end{align}
  	and 
  	  \begin{align}\label{5.23}
  	\mathrm{dim}^{\H}_{\mathcal{F}}(\mathscr{A}_{\mathrm{glob}})\leq 2\left(1+\frac{2\widetilde{\kappa}\|\f\|_{\V'}^2}{\mu^4\lambda_1^4}\right),
  	\end{align}
  	  for some absolute constant $\widetilde{\kappa}$. 
  \end{theorem}
  \begin{proof}
  We can rewrite \eqref{5.1} as 
  \begin{align}
  \frac{\d\xi}{\d t}=\mathcal{F}'(\u)\xi=-[\mu\A\xi+\B(\xi,\u)+\B(\u,\xi)+\alpha\xi+\beta\mathcal{C}'(\u)]\xi.
  \end{align}
  We define the numbers $q_m$, $m\in\N$ by 
  \begin{align*}
  q_m=\limsup_{t\to\infty}\sup_{\u_0\in\mathscr{A}_{\mathrm{glob}}}\sup_{\substack{_{\varrho_i\in\H}, \ \|\varrho_i\|_{\H}\leq 1\\ i=1,\ldots,n}}\frac{1}{t}\int_0^t\Tr(\mathcal{F}'(\S(s)\u_0)\circ q_m(s))\d s,
  \end{align*}
  where $q_m(s)=q_m(s;\u_0,\varrho_1,\ldots,\varrho_m)$ is the orthogonal projector of $\H$ onto the space spanned by $\Lambda(t;\u_0)\varrho_1,\ldots,\Lambda(t;\u_0)\varrho_m$. From section V.3.4, \cite{Te2} (see Proposition V.2.1 and Theorem V.3.3), we infer that if $q_m<0$, for some $m\in\N$, then the global attractor $\mathscr{A}_{\mathrm{glob}}$ has finite Hausdorff and fractal dimensions estimated respectively as
  \begin{align}
  \mathrm{dim}^{\H}_{\mathcal{H}}(\mathscr{A}_{\mathrm{glob}})&\leq m,\label{523}\\
  \mathrm{dim}^{\H}_{\mathcal{F}}(\mathscr{A}_{\mathrm{glob}})&\leq m\left(1+\max_{1\leq j\leq m}\frac{(q_j)_{+}}{|q_m|}\right). \label{524}
  \end{align}
  Our next aim is to estimate the number $q_m$. Let $\u_0\in\mathscr{A}_{\mathrm{glob}}$, $\varrho_1,\ldots,\varrho_m\in\H$ and set $\u(t)=\S(t)\u_0$ and $\xi_j(t)=\Lambda(t;\u_0)\varrho_j$, $t\geq 0$ and $1\leq j\leq m$. Let us consider  $\{\phi_1(t),\ldots,\phi_m(t)\}$ as an orthonormal basis in $\H$ for $\text{span}\{\xi_1(t),\ldots,\xi_m(t)\}$. Note that $\|\phi_j\|_{\H}^2=1$, for all $1\leq j\leq m$. Since $\xi_j\in \mathrm{L}^2(0,T;\V)$, we know that $\xi_j(t)\in\V$, for a.e. $t\in(0,T)$. By the Gram-Schmidt orthogonalization process, we can assume that $\phi_j(t)\in\V$. Then, one can see that 
  \begin{align}\label{525} \Tr(\mathcal{F}'(\u(s))\circ q_m(s))&=\sum_{j=1}^m(\mathcal{F}'(\u(s))\phi_j,\phi_j)\nonumber\\&=\sum_{j=1}^m\left\{\langle -[\mu\A\phi_j+\B(\phi_j,\u)+\B(\u,\phi_j)+\alpha\phi_j+\beta\mathcal{C}'(\u)\phi_j,\phi_j\rangle\right\}\nonumber\\&=\sum_{j=1}^m\left\{-\mu\|\phi_j\|_{\V}^2-\alpha\|\phi_j\|_{\H}^2-\beta\||\u|^{\frac{r-1}{2}}\phi_j\|_{\H}^2-\langle\B(\phi_j,\u),\phi_j\rangle\right\}.
  \end{align}
 Calculating similarly as in page 82, \cite{RR}, we find
  \begin{align}
  \left|\sum_{j=1}^m\langle\B(\phi_j,\u),\phi_j\rangle\right|\leq\frac{\mu}{2}\sum_{j=1}^m\|\phi_j\|_{\V}^2+\frac{\widetilde{\kappa}\|\u\|_{\V}^2}{2\mu},
  \end{align}
  for an absolute constant $\widetilde{\kappa}$. Thus, from \eqref{525}, we infer that
  \begin{align}\label{527}
   \Tr(\mathcal{F}'(\u(s))\circ q_m(s))&\leq-\frac{\mu}{2}\sum_{j=1}^m\|\phi_j\|_{\V}^2-\alpha\sum_{j=1}^m\|\phi_j\|_{\H}^2-\beta\sum_{j=1}^m\||\u(s)|^{\frac{r-1}{2}}\phi_j(s)\|_{\H}^2+\frac{\widetilde{\kappa}}{2\mu}\|\u(s)\|_{\V}^2\nonumber\\&\leq -\frac{\mu\lambda_1}{2}\sum_{j=1}^m\|\phi_j\|_{\H}^2+\frac{\widetilde{\kappa}}{2\mu}\|\u(s)\|_{\V}^2\leq -\frac{\mu\lambda_1}{2}m+\frac{\widetilde{\kappa}}{2\mu}\|\u(s)\|_{\V}^2.
  \end{align}
  Let us define the energy dissipation flux as 
  \begin{align}
  \mathcal{E}=\mu\lambda_1\limsup\limits_{t\to\infty}\sup_{\u_0\in\mathscr{A}_{\mathrm{glob}}}\frac{1}{t}\int_0^t\|\S(s)\u_0\|_{\V}^2\d s.
  \end{align}
  Note that the energy dissipation flux $\mathcal{E}$ is finite due to the estimate \eqref{3p9} and 
  $\mathcal{E}\leq \frac{\lambda_1}{\mu}\|\f\|_{\V'}^2.$
  Then, from \eqref{527}, we have 
  \begin{align}
  q_m&=\limsup_{t\to\infty}\sup_{\u_0\in\mathscr{A}_{\mathrm{glob}}}\sup_{\substack{_{\varrho_i\in\H}, \ \|\varrho_i\|_{\H}\leq 1\\ i=1,\ldots,n}}\frac{1}{t}\int_0^t\Tr(\mathcal{F}'(\S(s)\u_0)\circ q_m(s))\d s\leq-\frac{\mu\lambda_1}{2}m+\frac{\widetilde{\kappa}}{2\mu^2\lambda_1}\mathcal{E}, 
  \end{align}
  for all $m\in\N$. Hence, if $m'\in\N$ is defined by $m'-1\leq \frac{\widetilde{\kappa}}{\mu^3\lambda_1^2}\mathcal{E}<m',$ then $q_{m'}<0$ and thus from \eqref{523}, we find 
  \begin{align}
   \mathrm{dim}^{\H}_{\mathcal{H}}(\mathscr{A}_{\mathrm{glob}})&\leq m'\leq 1+\frac{\widetilde{\kappa}}{\mu^3\lambda_1^2}\mathcal{E}\leq 1+ \frac{\widetilde{\kappa}\|\f\|_{\V'}^2}{\mu^4\lambda_1}.
  \end{align}
  Furthermore, if $m''\in\N$ is defined by 
  $m''-1<\frac{2\widetilde{\kappa}}{\mu^3\lambda_1^2}\mathcal{E}\leq m'',$ then using Lemma VI.2.2, \cite{Te2}, we have 
  \begin{align}
  q_{m''}<0\ \text{ and }\ \frac{(q_j)_{+}}{|q_{m''}|}\leq 1, \ \text{ for all }\ j=1,\ldots,m''. 
  \end{align}
  Hence from \eqref{524}, we obtain 
  \begin{align}
   \mathrm{dim}^{\H}_{\mathcal{F}}(\mathscr{A}_{\mathrm{glob}})\leq 2m''\leq 2\left(1+\frac{2\widetilde{\kappa}}{\mu^3\lambda_1^2}\mathcal{E}\right)\leq 2\left(1+\frac{2\widetilde{\kappa}\|\f\|_{\V'}^2}{\mu^4\lambda_1^4}\right),
  \end{align}
  which completes the proof. 
  \end{proof}

\begin{remark}
	For the force $\f\in\V'$, if we take $\lambda_1^{1/2}$ for a characteristic length for the problem, one can regard $\|\f\|_{\V'}^{1/2}\lambda_1^{1/4}$ as a characteristic velocity and define the Reynolds number $\mathrm{Re}=\frac{\|\f\|_{\V'}^{1/2}}{\nu\lambda_1^{1/4}}$. We can also define the generalized Grashof number $G=\frac{\|\f\|_{\V'}}{\nu^2\lambda_1^{1/2}}=\mathrm{Re}^2$. Thus, from \eqref{515} and \eqref{5.23}, we infer that 
	\begin{align*}
		\mathrm{dim}^{\H}_{\mathcal{H}}(\mathscr{A}_{\mathrm{glob}})\leq 1+\widetilde{\kappa}G^2=1+\widetilde{\kappa}\mathrm{Re}^4 \ \text{ and }\ 	\mathrm{dim}^{\H}_{\mathcal{F}}(\mathscr{A}_{\mathrm{glob}})\leq 2(1+2\widetilde{\kappa}G^2)=2(1+2\widetilde{\kappa}\mathrm{Re}^4),
	\end{align*}
	for an absolute constant $\widetilde{\kappa}$. 
\end{remark}

	    \section{Upper Semicontinuity of Global Attractors} \label{sec7}\setcounter{equation}{0} In this section, we   establish an upper semicontinuity of global attractors for the 2D CBF equations. Let $\{\Omega_m\}_{m=1}^{\infty}$ be an expanding sequence of simply connected, bounded and smooth subdomains of $\Omega$ (for example, one can take $\Omega=\R\times(-L,L)$) such that $\Omega_m\to\Omega$ as   $m \to\infty$. We consider 
	    \begin{eqnarray}\label{4.1}
	    \left\{
	    \begin{aligned}
	    \frac{\partial \u_m(x,t)}{\partial t}-\mu\Delta\u_m(x,t)+(\u_m(x,t)\cdot\nabla)\u_m(x,t)&+\alpha\u_m(x,t)+\beta|\u_m(x,t)|^{r-1}\u_m(x,t)\\+\nabla p_m(x,t)&=\mathbf{f}_m(x), \\ \nabla\cdot\u_m(x,t)&=0,\\ \u_m(x,0)&=\u_{0,m}(x), 
	    \end{aligned}
	    \right.
	    \end{eqnarray}  for $(x,t)\in\Omega\times(0,T)$. In \eqref{4.1}, $\f_m(\cdot)$ and $\u_{0,m}(\cdot)$ are given by 
	    \begin{align}\label{72}
	    \f_m(x)=\left\{\begin{array}{cl}\f(x),&x\in\Omega_m, \\
	    \mathbf{0},& x\in\Omega\backslash\Omega_m,\end{array}\right. \ \text{ and }  \ \u_{0,m}(x)=\left\{\begin{array}{cl}\u_0(x),&x\in\Omega_m, \\
	    \mathbf{0},& x\in\Omega\backslash\Omega_m.\end{array}\right.
	    \end{align}
	     If we take orthogonal projection $\mathcal{P}$ onto the system \eqref{4.1}, we get 
	    \begin{equation}\label{7.24}
	    \left\{
	    \begin{aligned}
	    \frac{\d\u_m(t)}{\d t}+\mu\A\u_m(t)+\B(\u_m(t))+\alpha\u_m(t)+\beta\mathcal{C}(\u_m(t))&=\f_m, \ t\in (0,T),\\
	    \u_m(0)=\u_{0,m}, \v_m(0)=\v_{0,m}&=\mathbf{0},
	    \end{aligned}
	    \right.
	    \end{equation}
	    where $\f_m$ and $\u_{m,0}$ are defined as in \eqref{72}  (projected). Throughout this section, we differentiate the $\H$-spaces defined in $\Omega$ and $\Omega_m$ as $\H_{\Omega}$ and $\H_{\Omega_m},$ respectively and similar modifications are made for other spaces also. Furthermore, we assume that $\f\in\H_{\Omega}$.  Let $\u_m\in\mathrm{L}^{\infty}(0,T;\H_{\Omega_m})\cap\mathrm{L}^2(0,T;\V_{\Omega_m})$ be the unique weak solution of the system \eqref{7.24} with $\partial_t\u_m\in\mathrm{L}^{2}(0,T;\V_{\Omega_m}')$ and hence $\u_m\in\C([0,T];\H_{\Omega_m})$.  The 2D CBF model \eqref{7.24} possesses an attractor $\mathscr{A}_m$ that is compact, connected, and global in $\H_{\Omega_m}$ (see Theorem \ref{thm3.4} and Remark \ref{rem3.7}).	Our aim in this section is to check whether the global attractors $\mathscr{A}$ and  $\mathscr{A}_m$ (for simplicity of notations) of the 2D CBF equations corresponding   to $\Omega$ and $\Omega_m$, respectively, have the upper semicontinuity, when $m\to\infty$.   The upper semicontinuity of the  Klein-Gordon-Schr\"odinger equations on $\R^n$, $n\leq 3$ is established in \cite{LW} and 2D Navier-Stokes equations is obtained in \cite{ZD}. We follow the works \cite{LW,ZD}, etc for getting such a result for 2D CBF equations. We are providing a complete proof of the upper semicontinuity results, as we are improving some of the estimates used in \cite{ZD} and the pressure estimates are different in both works.
    \begin{lemma}\label{lem4.2}
    	Assume that $\u_0\in\mathcal{B}_1,$ where $\mathcal{B}_1$ is the absorbing set defined in \eqref{314b}. Then, for any $\e>0$, there exist $T_1(\e)$ and $m_1(\e)$ such that the solution $\u(\cdot)$ of the problem \eqref{1} with the initial condition $\u_0$ satisfies:
    	\begin{align}\label{74}
    	\int_{\Omega\backslash\Omega_m}|\u(x,t)|^2\d x\leq\e, 
    	\end{align}
    	for $t\geq T_1(\e)$ and $m\geq m_1(\e)$, where $T_1(\e)$ and	$m_1(\e)$  are constants depending only on $\e$.
    \end{lemma}
\begin{proof}
	Let us first define a cutoff function $\chi(\cdot)\in\mathrm{C}^{1}(\R^2;\R)$ as 
	\begin{align*}
	\chi(x)=\left\{\begin{array}{cl}0,& \ \text{ if } \ |x|<1, \\ (|x|^2-1)^2(2-|x|^2)^2,&  \ \text{ if } \ 1\leq |x|<2,\\ 
	1,&  \ \text{ if } \ |x|\geq 2,\end{array}\right.
	\end{align*}
and note that $0\leq \chi(x)\leq 1$.	Let us set $\chi_R(x)=\chi\left(\frac{x}{R}\right),$ for sufficiently large $R>1$, then
\begin{align*}
\chi_R(x)=\left\{\begin{array}{cl}0,& \ \text{ if } \ |x|<R, \\ \left(\frac{|x|^2}{R^2}-1\right)^2\left(2-\frac{|x|^2}{R^2}\right)^2,&  \ \text{ if } \ R\leq |x|<2R,\\ 
1,&  \ \text{ if } \ |x|\geq 2R,\end{array}\right.
\end{align*}
and 
\begin{align*}
\nabla \chi_R(x)=\left\{\begin{array}{cl}0,&\ \text{ if } |x|<R \ \text{ or } \  |x|\geq 2R, \\
\frac{4}{R^2}\left(\frac{|x|^2}{R^2}-1\right)\left(2-\frac{|x|^2}{R^2}\right)x&  \ \text{ if } \ R\leq |x|<2R.\end{array}\right.
\end{align*}
 It can be easily seen that 
	\begin{align*}
	\|\nabla\chi_R\|_{\mathbb{L}^{\infty}(\R^2)}\leq \frac{C}{R},
	\end{align*}
where $C=12$.	Taking divergence in the first equation in \eqref{1}, we find 
	\begin{align*}
	-\Delta p&=\nabla\cdot[(\u\cdot\nabla)\u]+\beta\nabla\cdot[|\u|^{r-1}\u]=\nabla\cdot(\nabla\cdot(\u\otimes\u))+\beta\nabla\cdot[|\u|^{r-1}\u]\\&=\sum_{i,j=1}^2\frac{\partial^2}{\partial x_i\partial x_j}(u_iu_j)+\beta\nabla\cdot[|\u|^{r-1}\u],
	\end{align*}
	so that 
	$$p=(-\Delta)^{-1}\left[\sum_{i,j=1}^2\frac{\partial^2}{\partial x_i\partial x_j}(u_iu_j)+\beta\nabla\cdot[|\u|^{r-1}\u]\right]. $$ For $v\in\mathrm{L}^{2}(0,T;\mathrm{L}^2(\Omega))\cap\mathrm{L}^{r+1}(0,T;\mathrm{L}^{r+1}(\Omega))$, we find 
	\begin{align*}
	\int_0^T|\langle p(t),v(t)\rangle|\d t&\leq\int_0^T\left|\left<(-\Delta)^{-1}\sum_{i,j=1}^2\frac{\partial^2}{\partial x_i\partial x_j}(u_i(t)u_j(t)),v(t)\right>\right|\d t\nonumber\\&\quad+\beta\int_0^T\left|\left<(-\Delta)^{-1}\nabla\cdot[|\u(t)|^{r-1}\u(t)],v(t)\right>\right|\d t\nonumber\\&\leq C\int_0^T\|\u(t)\|_{\wi\L^4_\Omega}^2\|v(t)\|_{\mathrm{L}^2(\Omega)}\d t+C\beta\int_0^T\|\u(t)\|_{\wi\L^{r+1}_\Omega}^{r}\|v(t)\|_{\mathrm{L}^{r+1}(\Omega)}\d t\nonumber\\&\leq C\left(\int_0^T\|\u(t)\|_{\wi\L^4_\Omega}^4\d t\right)^{1/2}\left(\int_0^T\|v(t)\|_{\mathrm{L}^2(\Omega)}^2\d t\right)^{1/2}\nonumber\\&\quad+C\beta\left(\int_0^T\|\u(t)\|_{\wi\L^{r+1}_\Omega}^{r+1}\d t\right)^{\frac{r}{r+1}}\left(\int_0^T\|v(t)\|_{\mathrm{L}^{r+1}(\Omega)}^{r+1}\d t\right)^{\frac{1}{r+1}},
	\end{align*}
so that $p\in\mathrm{L}^2(0,T;\mathrm{L}^2(\Omega))+\mathrm{L}^{\frac{r+1}{r}}(0,T;\mathrm{L}^{\frac{r+1}{r}}(\Omega))$. Similarly for $v\in\mathrm{L}^{r+1}(0,T;\mathrm{L}^2(\Omega)\cap\mathrm{L}^{r+1}(\Omega))$ and all $\theta\geq 0$, we have 
\begin{align*}
\int_t^{t+\theta}|\langle p(s),v(s)\rangle|\d s&\leq C\left\{ \theta^{\frac{r-1}{2(r+1)}}\left(\int_t^{t+\theta}\|\u(s)\|_{\wi\L^4_\Omega}^4\d s\right)^{\frac{1}{2}}+\beta\left(\int_t^{t+\theta}\|\u(s)\|_{\wi\L^{r+1}_\Omega}^{r+1}\d s\right)^{\frac{r}{r+1}}\right\}\nonumber\\&\quad\times\left(\int_0^T\|v(t)\|_{\mathrm{L}^{r+1}(\Omega)}^{r+1}\d t\right)^{\frac{1}{r+1}},
\end{align*}
so that we get 
\begin{align}\label{6p5}
&\left(\int_t^{t+\theta}\|p(s)\|_{\mathrm{L}^2(\Omega)+\mathrm{L}^{\frac{r+1}{r}}(\Omega)}^{\frac{r+1}{r}}\d s \right)^{\frac{r}{r+1}}\nonumber\\&\leq C\left\{ \theta^{\frac{r-1}{2(r+1)}}\sup_{s\in[t,t+\theta]}\|\u(s)\|_{\H_{\Omega}}\left(\int_t^{t+\theta}\|\u(s)\|_{\V_{\Omega}}^2\d s\right)^{\frac{1}{2}}+\beta\left(\int_t^{t+\theta}\|\u(s)\|_{\wi\L^{r+1}_\Omega}^{r+1}\d s\right)^{\frac{r}{r+1}}\right\}\nonumber\\&\leq C\left\{\theta^{\frac{r-1}{2(r+1)}}\frac{M_1}{\sqrt{\mu}}\left(M_1^2+\frac{\theta}{\mu}\|\f\|_{\V_{\Omega}'}^2\right)^{\frac{1}{2}}+\beta^{\frac{1}{r+1}}\left(M_1^2+\frac{\theta}{\mu}\|\f\|_{\V_{\Omega}'}^2\right)^{\frac{r}{r+1}}\right\},
\end{align}
\iffalse 
	\begin{align}
	\int_t^{t+\theta}\|p(s)\|_{\mathrm{L}^2(\Omega)}^2\d s&\leq C\int_t^{t+\theta}\|\u(s)\|_{\wi\L^4_{\Omega}}^4\d s+C\beta\int_t^{t+\theta}\|\u(s)\|_{\wi\L^{2r}_{\Omega}}^{2r}\d s\nonumber\\&\leq C\left(\sup_{s\in[t,t+\theta]}\|\u(s)\|_{\H_{\Omega}}^2+\beta\sup_{s\in[t,t+\theta]}\|\u(s)\|_{\H_{\Omega}}^{2r-2}\right)\int_t^{t+\theta}\|\u(s)\|_{\V_{\Omega}}^2\d s \nonumber\\&\leq \frac{C(M_1^2+\beta M_1^{2r-2}) }{\mu^2}\left(\frac{2}{\mu\lambda_1}+\theta\right)\|\f\|_{\V_{\Omega}'}^2,
	\end{align}
	\fi
where we used \eqref{3p29}.	Taking the inner product with $\chi_R^2(x)\u(t,x)$ to the first equation  in \eqref{1}, we find 
	\begin{align}\label{4.9}
	\frac{1}{2}\frac{\d}{\d t}\|\chi_R\u(t)\|_{\L^2(\Omega)}^2& = -\mu(\nabla\u(t),\nabla(\chi_R^2\u(t)))-\langle(\u(t)\cdot\nabla)\u(t),\chi_R^2\u(t)\rangle-\alpha\|\chi_R\u(t)\|_{\L^2(\Omega)}^2\nonumber\\&\quad -\beta\langle|\u(t)|^{r-1}\u(t),\chi_R^2\u(t)\rangle-\langle\nabla p(t),\chi_R^2\u(t)\rangle-(\f,\chi_R^2\u(t)).
	\end{align}
	Using H\"older's and Young's inequalities, we estimate $-\mu(\nabla\u,\nabla(\chi_R^2\u))$ as 
	\begin{align}
	-\mu(\nabla\u,\nabla(\chi_R^2\u))&=-\mu(\nabla\u,\chi_R\nabla(\chi_R\u))-\mu(\nabla\u,\nabla(\chi_R)\chi_R\u)\nonumber\\&=-\mu(\nabla(\chi_R\u)-\nabla(\chi_R)\u,\nabla(\chi_R\u))-\mu(\nabla\u,\nabla(\chi_R)\chi_R\u)\nonumber\\&=-\mu\|\nabla(\chi_R\u)\|_{\L^2(\Omega)}^2+\mu(\nabla(\chi_R)\u,\nabla(\chi_R\u))-\mu(\nabla\u,\nabla(\chi_R)\chi_R\u) \nonumber\\&\leq -\mu\|\nabla(\chi_R\u)\|_{\L^2(\Omega)}^2+\mu\|\nabla\chi_R\|_{\L^{\infty}(\R^2)}\|\u\|_{\H_{\Omega}}\|\nabla(\chi_R\u)\|_{\L^2(\Omega)}\nonumber\\&\quad+\mu\|\nabla\u\|_{\H_{\Omega}}\|\nabla\chi_R\|_{\L^{\infty}(\R^2)}\|\chi_R\u\|_{\L^2(\Omega)}\nonumber\\&\leq-\frac{\mu}{2}\|\nabla(\chi_R\u)\|_{\L^2(\Omega)}^2+\frac{C\mu}{R}\|\u\|_{\H_{\Omega}}^2+\frac{C\mu}{R}\left(\|\nabla\u\|_{\H_{\Omega}}^2+\|\u\|_{\H_{\Omega}}^2\right)\nonumber\\&\leq-\frac{\mu}{2}\|\nabla(\chi_R\u)\|_{\L^2(\Omega)}^2+\frac{C\mu M_1^2}{R}+\frac{C\mu}{R}\|\u\|_{\V_{\Omega}}^2.
	\end{align}
	We consider the term $-\langle(\u\cdot\nabla)\u,\chi_R^2\u\rangle$ from \eqref{4.9} and using an integration by parts, divergence free condition, H\"older's, Ladyzhenskaya's and Young's inequalities, we  estimate it as 
	\begin{align}\label{4.10}
	-\langle(\u\cdot\nabla)\u,\chi_R^2\u\rangle&=-\sum_{i,j=1}^2\int_{\Omega}u_i(x)\frac{\partial u_j(x)}{\partial x_i}\chi_R^2(x)u_j(x)\d x \nonumber\\&=-\frac{1}{2}\sum_{i,j=1}^2\int_{\Omega}\chi_R^2(x)u_i(x)\frac{\partial u_j^2(x)}{\partial x_i}\d x\nonumber\\& =\sum_{i,j=1}^2\int_{\Omega}\chi_R(x)\frac{\partial\chi_R(x)}{\partial x_i}u_i(x)u_j^2(x)\d x\nonumber\\&\leq\|\nabla\chi_R\|_{\L^{\infty}(\R^2)}\|\chi_R\u\|_{\L^2(\Omega)}\|\u\|_{\L^4(\Omega)}^2\nonumber\\&\leq \frac{C}{R}\left(\|\u\|_{\V_{\Omega}}^2+\|\u\|_{\H_{\Omega}}^4\right)\leq \frac{C}{R}\left(\|\u\|_{\V_{\Omega}}^2+M_1^4\right).
	\end{align}
	\iffalse 
	Next, we estimate the term $\beta|(|\u|^{r-1}\u,\chi_R^2\u)|$ as 
	\begin{align}
	\beta|(|\u|^{r-1}\u,\chi_R^2\u)|&= \beta|(\chi_R[|\u|^{r-1}\u],\chi_R\u)|\leq\beta\|\chi_R[|\u|^{r-1}\u]\|_{\L^2(\Omega)}\|\chi_R\u\|_{\L^2}\nonumber\\&\leq\frac{\mu}{4}\|\nabla(\chi_R\u)\|_{\L^2(\Omega)}^2+\frac{\beta^2}{\mu\lambda_1}\|\chi_R[|\u|^{r-1}\u]\|_{\L^2(\Omega)}^2.
	\end{align}
	\fi 
	Once again an integration by parts and divergence free condition yields 
	\begin{align}
	-\langle\nabla p,\chi_R^2\u\rangle&=\langle p,\nabla\cdot(\chi_R^2\u)\rangle=2\langle p,\chi_R\nabla(\chi_R)\cdot\u\rangle\nonumber\\&\leq 2\|p\|_{\mathrm{L}^2(\Omega)+\mathrm{L}^{\frac{r+1}{r}}(\Omega)}\|\chi_R\nabla(\chi_R)\cdot\u\|_{\mathrm{L}^2(\Omega)\cap\mathrm{L}^{r+1}(\Omega)}\nonumber\\&\leq 2\|p\|_{\mathrm{L}^2(\Omega)+\mathrm{L}^{\frac{r+1}{r}}(\Omega)} \|\nabla\chi_R\|_{\L^{\infty}(\R^2)}\left(\|\chi_R\u\|_{\L^2(\Omega)}+\|\chi_R\u\|_{\wi\L^{r+1}(\Omega)}\right)\nonumber\\&\leq \frac{C}{R}\|p\|_{\mathrm{L}^2(\Omega)+\mathrm{L}^{\frac{r+1}{r}}(\Omega)}\left(\|\u\|_{\H_{\Omega}}+\|\u\|_{\wi\L^{r+1}_{\Omega}}\right)\nonumber\\&\leq \frac{C}{R}\left(\|p\|_{\mathrm{L}^2(\Omega)+\mathrm{L}^{\frac{r+1}{r}}(\Omega)}^{\frac{r+1}{r}}+\|\u\|_{\H_{\Omega}}^{r+1}+\|\u\|_{\wi\L^{r+1}_{\Omega}}^{r+1}\right)\nonumber\\&\leq  \frac{C}{R}\left(\|p\|_{\mathrm{L}^2(\Omega)+\mathrm{L}^{\frac{r+1}{r}}(\Omega)}^{\frac{r+1}{r}}+M_1^{r+1}+\|\u\|_{\wi\L^{r+1}_{\Omega}}^{r+1}\right).
	\end{align}
	Using the Cauchy-Schwarz inequality, we estimate $-(\f,\chi_R^2\u)$ as 
	\begin{align}\label{4.16}
	-(\f,\chi_R^2\u)\leq \|\chi_R\f\|_{\L^{2}(\Omega)}\|\chi_R\u\|_{\L^2(\Omega)}\leq M_1\|\chi_R\f\|_{\L^{2}(\Omega)}.
	\end{align}
Combining \eqref{4.10}-\eqref{4.16} and substituting it in \eqref{4.9}, we find 
	\begin{align}\label{4.17}
&	\frac{1}{2}\frac{\d}{\d t}\|\chi_R\u(t)\|_{\L^2(\Omega)}^2+\frac{\mu}{2}\|\nabla(\chi_R\u(t))\|_{\L^2(\Omega)}^2+\alpha\|\chi_R\u(t)\|_{\L^2(\Omega)}^2+\beta\|\chi_R|\u(t)|^{\frac{r-1}{2}}\|_{\mathbb{L}^2(\Omega)}^2\nonumber\\& \leq M_1\|\chi_R\f\|_{\L^{2}(\Omega)}+\frac{C}{R}\left[(\mu+1)\|\u(t)\|_{\V_{\Omega}}^2+\|p(t)\|_{\mathrm{L}^2(\Omega)+\mathrm{L}^{\frac{r+1}{r}}(\Omega)}^{\frac{r+1}{r}}+\|\u(t)\|_{\wi\L^{r+1}_{\Omega}}^{r+1}\right.\nonumber\\&\qquad\qquad\qquad\qquad\qquad\left.+M_1^2(M_1^{r-1}+M_1^2+\mu)\right].
	\end{align}
Using variation of constants formula, we obtain 
\begin{align}\label{4.18}
\|\chi_R\u(t)\|_{\L^2(\Omega)}^2&\leq \|\chi_R\u_0\|_{\L^2(\Omega)}^2e^{-\mu\lambda_1 t}+\frac{2M_1\|\chi_R\f\|_{\L^2(\Omega)}}{\mu\lambda_1}+\frac{CM_1^2}{R\mu\lambda_1}(M_1^{r-1}+M_1^2+\mu)\nonumber\\&\quad+\frac{C(\mu+1)}{R}\int_0^te^{-\mu\lambda_1(t-s)}\left[\|\u(s)\|_{\V_{\Omega}}^2+\|p(s)\|_{\mathrm{L}^2(\Omega)+\mathrm{L}^{\frac{r+1}{r}}(\Omega)}^{\frac{r+1}{r}}+\|\u(s)\|_{\wi\L^{r+1}_{\Omega}}^{r+1}\right]\d s.
\end{align}
Let $k$ be an integer such that $k\leq t\leq k+1$ and consider 
\begin{align}\label{619}
&e^{-\mu\lambda_1 t}\int_0^te^{\mu\lambda_1 s}\left[\|\u(s)\|_{\V_{\Omega}}^2+\|p(s)\|_{\mathrm{L}^2(\Omega)+\mathrm{L}^{\frac{r+1}{r}}(\Omega)}^{\frac{r+1}{r}}+\|\u(s)\|_{\wi\L^{r+1}_{\Omega}}^{r+1}\right]\d s\nonumber\\&\leq e^{-\mu\lambda_1 t}\sum_{j=0}^ke^{\mu\lambda_1(j+1)}\int_j^{j+1}\left[\|\u(s)\|_{\V_{\Omega}}^2+\|p(s)\|_{\mathrm{L}^2(\Omega)+\mathrm{L}^{\frac{r+1}{r}}(\Omega)}^{\frac{r+1}{r}}+\|\u(s)\|_{\wi\L^{r+1}_{\Omega}}^{r+1}\right]\d s\nonumber\\&\leq e^{-\mu\lambda_1 t}\sum_{j=0}^ke^{\mu\lambda_1(j+1)}\Bigg\{\left(\frac{1}{\mu}+\frac{1}{\beta}\right)\left(M_1^2+\frac{1}{\mu}\|\f\|_{\V'_{\Omega}}^2\right)\nonumber\\&\qquad+C\left[\frac{M_1}{\sqrt{\mu}}\left(M_1^2+\frac{1}{\mu}\|\f\|_{\V_{\Omega}'}^2\right)^{\frac{1}{2}}+\beta^{\frac{1}{r+1}}\left(M_1^2+\frac{\theta}{\mu}\|\f\|_{\V_{\Omega}'}^2\right)^{\frac{r}{r+1}}\right]^{\frac{r+1}{r}}\Bigg\}\nonumber\\&=e^{-\mu\lambda_1 t}\Bigg\{\left(\frac{1}{\mu}+\frac{1}{\beta}\right)\left(M_1^2+\frac{1}{\mu}\|\f\|_{\V'_{\Omega}}^2\right)\nonumber\\&\qquad+C\left[\frac{M_1}{\sqrt{\mu}}\left(M_1^2+\frac{1}{\mu}\|\f\|_{\V_{\Omega}'}^2\right)^{\frac{1}{2}}+\beta^{\frac{1}{r+1}}\left(M_1^2+\frac{\theta}{\mu}\|\f\|_{\V_{\Omega}'}^2\right)^{\frac{r}{r+1}}\right]^{\frac{r+1}{r}}\Bigg\}\frac{e^{\mu\lambda_1}(e^{(k+1)\mu\lambda_1}-1)}{e^{\mu\lambda_1}-1}\nonumber\\&\leq \Bigg\{\left(\frac{1}{\mu}+\frac{1}{\beta}\right)\left(M_1^2+\frac{1}{\mu}\|\f\|_{\V'_{\Omega}}^2\right)\nonumber\\&\qquad+C\left[\frac{M_1}{\sqrt{\mu}}\left(M_1^2+\frac{1}{\mu}\|\f\|_{\V_{\Omega}'}^2\right)^{\frac{1}{2}}+\beta^{\frac{1}{r+1}}\left(M_1^2+\frac{\theta}{\mu}\|\f\|_{\V_{\Omega}'}^2\right)^{\frac{r}{r+1}}\right]^{\frac{r+1}{r}}\Bigg\}\frac{e^{2\mu\lambda_1}}{e^{\mu\lambda_1}-1},
\end{align}
where we used  \eqref{3p29} and \eqref{6p5}. 
Thus, from \eqref{4.18}, we immediately have 
\begin{align}\label{4.20}
\|\chi_R\u(t)\|_{\L^2(\Omega)}^2&\leq  \|\chi_R\u_0\|_{\L^2(\Omega)}^2e^{-\mu\lambda_1 t}+\frac{2M_1\|\chi_R\f\|_{\L^2(\Omega)}}{\mu\lambda_1}+\frac{CM_1^2}{R\mu\lambda_1}(M_1^{r-1}+M_1^2+\mu)\nonumber\\&\quad+\frac{C(\mu+1)}{R}\Bigg\{\left(\frac{1}{\mu}+\frac{1}{\beta}\right)\left(M_1^2+\frac{1}{\mu}\|\f\|_{\V'_{\Omega}}^2\right)\nonumber\\&\qquad+C\left[\frac{M_1}{\sqrt{\mu}}\left(M_1^2+\frac{1}{\mu}\|\f\|_{\V_{\Omega}'}^2\right)^{\frac{1}{2}}+\beta^{\frac{1}{r+1}}\left(M_1^2+\frac{\theta}{\mu}\|\f\|_{\V_{\Omega}'}^2\right)^{\frac{r}{r+1}}\right]^{\frac{r+1}{r}}\Bigg\}\frac{e^{2\mu\lambda_1}}{e^{\mu\lambda_1}-1}.
\end{align}
Since $\f\in\H_{\Omega}$, it can be easily seen that $\lim\limits_{R\to\infty}\|\chi_R\f\|_{\L^2(\Omega)}=0$. Hence, for any $\e>0$, there exist $m_1(\e)$ and $T_1(\e)$ such that $$\|\chi_{\frac{m_1(\e)}{2}}\u(t)\|_{\L^2(\Omega)}^2<\e, \ \text{ for all } \ t>T_1(\e).$$ Moreover, for $m>m_1(\e)$, we finally have 
\begin{align}
\|\u(t)\|_{\L^2(\Omega\backslash\Omega_m)}^2\leq \|\chi_{\frac{m_1(\e)}{2}}\u(t)\|_{\L^2(\Omega)}^2\leq \e,
\end{align}
for all $t>T_1(\e)$. 
\end{proof}

One can prove the following Lemma in a similar fashion as in Lemma \ref{lem4.2}. 
\begin{lemma}
		Assume that $\u_{m,0}\in\mathcal{B}_{1,m},$ where $\mathcal{B}_{1,m}$ is the absorbing set defined in \eqref{314b}. Then, for any $\e>0$, there exist $T_2(\e)$ and $m_2(\e)$ such that the solution $\u_m(\cdot)$ of the problem \eqref{4.1} with the initial condition $\u_{0,m}$ satisfies:
	\begin{align}
	\int_{\Omega_m\backslash\Omega_k}|\u_m(x,t)|^2\d x\leq \e, 
	\end{align}
	for $t\geq T_2(\e)$ and $m\geq k\geq m_2(\e)$, where $T_2(\e)$ and
	$m_2(\e)$  are constants depending only on $\e$.
\end{lemma}
\subsection{Upper semicontinuity} Let us now prove the main results of this section,  the upper
semicontinuity of the global attractors $\mathscr{A}_m$ to the global attractor $\mathscr{A}$ as $m\to\infty$.
\begin{lemma}\label{lem7.3}
	If $\u_{m,0}\in\mathscr{A}_{m},$ $m=1,2,\ldots$, then there exits a $\u_0\in\mathscr{A},$ such that up to a subsequence
	\begin{align}\label{724}
	\S_m(\cdot)\u_{m,0}\xrightharpoonup{w}\S(\cdot)\u_0 \ \text{ in } \ \mathrm{L}^2(-T,T;\V_{\Omega}),
	\end{align}
	and 
	\begin{align}\label{725}
	\S_m(t)\u_{m,0}\xrightharpoonup{w}\S(t)\u_0 \ \text{ in } \ \H_{\Omega}, \ \text{ for each } \ t\in\R. 
 	\end{align}
\end{lemma}
\begin{proof}
	Given that $\u_{m,0}\in\mathscr{A}_{m},$ $m=1,2,\ldots$ and hence $\S_m(t)\u_{m,0}\in\mathscr{A}_{m}$, for each $t\in\R$ and using \eqref{314b}, we find 
	\begin{align}\label{726}
	\|\S_m(t)\u_{m,0}\|_{\H}\leq M_1, \ \text{ for all } \ t\in\R. 
	\end{align}
	Estimates like \eqref{3p8} ensures that for each $T>0$, the sequence $\{\S_m(t)\u_{m,0}\}_{m\in\N}$ is uniformly bounded in $\mathrm{L}^{\infty}(-T,T;\H_{\Omega})\cap\mathrm{L}^{2}(-T,T;\V_{\Omega})$ and $\left\{\frac{\partial}{\partial t}\S_m(t)\u_{m,0}\right\}_{m\in\N}$ is bounded in $\mathrm{L}^{2}(-T,T;{\V}_{\Omega}') $. These uniform bounds guarantee us the existence of a subsequence $\{\u_{{m_k},0}\}_{k\in\N}$ of $\{\u_{m,0}\}_{m\in\N}$ such that 
	\begin{equation}\label{727}
	\left\{
	\begin{aligned}
	\S_{m_k}(t)\u_{m_k,0}&\xrightharpoonup{w^*}\u \ \text{ in }\ \mathrm{L}^{\infty}(-T,T;\H_{\Omega}), \\
	\S_{m_k}(t)\u_{m_k,0}&\xrightharpoonup{w}\u \ \text{ in }\ \mathrm{L}^{2}(-T,T;\V_{\Omega}), \\
		\S_{m_k}(t)\u_{m_k,0}&\xrightharpoonup{w}\u \ \text{ in }\ \mathrm{L}^{r+1}(-T,T;\wi\L^{r+1}_{\Omega}), \\
	\frac{\partial}{\partial t}\S_m(t)\u_{m,0}&\xrightharpoonup{w} \frac{\partial\u}{\partial t} \ \text{ in }\ \mathrm{L}^{2}(-T,T;{\V}_{\Omega}').
	\end{aligned}
	\right.
	\end{equation}
	From the first convergence in \eqref{727}, it is clear that 
	\begin{align}
		\S_{m_k}(t)\u_{m_k,0}\xrightharpoonup{w}\u \ \text{ in }\ \mathrm{L}^{2}(-T,T;\H_{\Omega}).
	\end{align}
	Arguing similarly as in Lemma \ref{lem4.1}, one can show that $\u(\cdot)$ is a weak solution to \eqref{kvf} defined on $\R$ and $\u\in\C(\R;\H_{\Omega})$. Thus, we can write 
$
	\u(t)=\S(t)\u(0)
$ and obtain 
	\begin{align}
\S_{m_k}(\cdot)\u_{m_k,0}\xrightharpoonup{w}\S(\cdot)\u(0)\ \text{ in }\ \mathrm{L}^{2}(-T,T;\H_{\Omega}),
\end{align}
which proves \eqref{724}. 

Our next aim is to establish \eqref{725}. Let us now fix $t^*\in[-T,T]$. From the estimate \eqref{726}, we know that the sequence $\{\S_m(t^*)\u_{m,0}\}_{m\in\N}$ is bounded in $\H_{\Omega}$. Therefore there exists $\widetilde{\u}\in\H_{\Omega}$ and a subsequence $\{\S_{m_k}(t^*)\u_{m_k,0}\}_{k\in\N}$ of $\{\S_m(t^*)\u_{m,0}\}_{m\in\N}$ such that $\S_{m_k}(t^*)\u_{m_k,0}\xrightharpoonup{w}\widetilde{\u}$ as $k\to\infty$ in $\H_{\Omega}$. Thus $\u(t)$ is a solution of \eqref{kvf} with $\u(t^*)=\widetilde{\u}$. Since $t^*\in[-T,T]$ is arbitrary, we obtain 
	\begin{align}\label{730}
\S_{m_k}(t)\u_{m_k,0}\xrightharpoonup{w}\S(t)\u(0) \ \text{ in } \ \H, \ \text{ for each } \ t\in\R,
\end{align}
as desired. Now, it is left to show that $\u_0\in\mathscr{A}$. Using the weakly lower-semicontinuity of norm, \eqref{730} and \eqref{726}, we find 
\begin{align}
\|\S(t)\u(0)\|_{\H}\leq\liminf_{k\to\infty}\|\S_{m_k}(t)\u_{m_k,0}\|_{\H}\leq M_1, \ \text{ for all }\ t\in\R,
\end{align}
which implies that the solution $\S(t)\u(0)$ defined on $\R$ and bounded. Hence, $\S(t)\u(0)\in\mathscr{A},$ for all $t\in\R$ and in particular $\u(0)=\u_0\in\mathscr{A}$, which completes the proof.
\end{proof}
The following Lemma improves the result obtained in Lemma \ref{lem7.3}. 
\begin{lemma}\label{lem6.5}
	If $\u_{m,0}\in\mathscr{A}_m$, $m=1,2,\ldots$, then there exists $\u_0\in\mathscr{A}$ such that up to a subsequence \begin{align}\label{7.32}\u_{m,0}\to \u_0\ \text{ strongly   in } \ \H.\end{align}
\end{lemma}
\begin{proof}
The proof is similar to that of Proposition \ref{prop4.4}. A calculation similar to \eqref{44} yields 
	\begin{align}\label{732}
\|\u_m(t)\|_{\H}^2&= \|\u_m(s)\|_{\H}^2e^{-\mu\lambda_1 (t-s)}\nonumber\\&\quad+2\int_s^te^{-\mu\lambda_1(t-s)}\left[\langle \f,\u_m(s)\rangle-\langle\!\langle\u_m(s)\rangle\!\rangle^2-\alpha\|\u_m(s)\|_{\H}^2-\beta\|\u_m(s)\|_{\wi\L^{r+1}}^{r+1}\right]\d s.
\end{align}
Note that $\u_{m,0}=\S_m(T)\S_m(-T)\u_{m,0}$, for each $T\in\R$. Using \eqref{732}, we find 
\begin{align}
\|\u_{m,0}\|_{\H}^2&=\|\S_m(T)\S_m(-T)\u_{m,0}\|_{\H}^2\nonumber\\&= \|\S_m(T_0)\S_m(-T)\u_{m,0}\|_{\H}^2e^{-\mu\lambda_1 (T-T_0)}+2\int_{T_0}^Te^{-\mu\lambda_1(T-s)}\langle \f,\S_m(s)\S_m(-T)\u_{m,0}\rangle\d s\nonumber\\&\quad-2\int_{T_0}^Te^{-\mu\lambda_1(T-s)}\Big[\langle\!\langle\S_m(s)\S_m(-T)\u_{m,0}\rangle\!\rangle^2+\alpha\|\S_m(s)\S_m(-T)\u_{m,0}\|_{\H}^2\nonumber\\&\qquad+\beta\|\S_m(s)\S_m(-T)\u_{m,0}\|_{\wi\L^{r+1}}^{r+1}\Big]\d s.
\end{align}
Using Lemma \ref{lem7.3} (see \eqref{725}), one can easily see that there exists a $\u_0\in\mathscr{A}$ such that up to a subsequence 
\begin{align}\label{734}
\u_{m,0}\xrightharpoonup{w}\u_0\ \text{ weakly in }\ \H_{\Omega},
\end{align}
and the weakly lower-semicontinuity property of the $\H_{\Omega}$ norm gives 
\begin{align}\label{735}
\|\u_0\|_{\H_{\Omega}}\leq\liminf\limits_{m\to\infty}\|\u_{m,0}\|_{\H_{\Omega}}.
\end{align}
Since $\u_{m,0}\in\mathscr{A}_m$, the solution $\S_m(t)\u_{m,0}$ of the problem \eqref{7.24} is bounded on $\R$ and $\S_m(t)\u_{m,0}\in\mathscr{A}_m$, for all $t\in\R$. Hence, for each $T\geq 0$, one can easily get $\S_m(-T)\u_{m,0}\in\mathscr{A}_m$. Once again using Lemma \ref{lem7.3}, we can find a $\u_T\in\mathscr{A}$ such that, up to a subsequence 
\begin{align}
\S_m(t)\S_m(-T)\u_{m,0}\xrightharpoonup{w}\u_T\ \text{ weakly in }\ \H_{\Omega}, 
\end{align}
for all $t\in\R$. Proceeding similarly as in Proposition \ref{prop4.4}, we obtain 
\begin{align}\label{737}
\limsup\limits_{m\to\infty}\|\u_{m,0}\|_{\H}^2\leq\|\u_0\|_{\H}^2+(M_1^2-\|\S(T^*)\u_T\|_{\H}^2)e^{-\mu\lambda_1(T-T^*)}\leq \|\u_0\|_{\H}^2+M_1^2e^{-\mu\lambda_1(T-T^*)},
\end{align}
for all $T>T^*$. Since $\H_{\Omega}$ is a Hilbert space, passing $T\to\infty$ in \eqref{737} and using \eqref{734}-\eqref{735}, we get \eqref{7.32}, which completes the proof.
\end{proof}
Using Lemmas \ref{lem7.3} and \ref{lem6.5}, we obtain the following main theorem:
\begin{theorem}\label{thm4.1}
	Assume that $\f\in\H$ and $\mu>0$. Let $\mathscr{A}$ and  $\mathscr{A}_m$ be the global attractors corresponding to the problems \eqref{kvf} and \eqref{7.24}, respectively. Then 
	\begin{align}\label{6.3}
	\lim\limits_{m\to+\infty}\mathrm{dist}_{\H_{\Omega}}(\mathscr{A}_m,\mathscr{A})=0, 
	\end{align}
	where $\mathrm{dist}_{\H_{\Omega}}(\mathscr{A}_m,\mathscr{A})=\sup\limits_{\u\in\mathscr{A}_m}\mathrm{dist}_{\H_{\Omega}}(\u,\mathscr{A})$  is the Hausdorff semidistance of the space $\H_{\Omega}$. 
\end{theorem}
    \begin{proof}
    We prove the Theorem by using a contradiction. Let us assume that \eqref{6.3} does not hold true. Then there exists a fixed $\e_0>0$ and a sequence $\u_m\in\mathscr{A}_m$ such that 
    \begin{align}\label{6.25}
    \mathrm{dist}_{\H_{\Omega}}(\u_m,\mathscr{A})\geq\e_0>0, \ m=1,2,\ldots.\end{align} Using Lemma \ref{lem6.5}, we can find a subsequence $\{\u_{m_k}\}$ of $\{\u_m\}$ such that $	\lim\limits_{k\to+\infty}\mathrm{dist}_{\H_{\Omega}}(\u_{m_k},\mathscr{A})=0,$ which is a contradiction to \eqref{6.25} and it completes the proof.  
    \end{proof}

\begin{remark} We can consider the asymptotic behavior of solutions of the  2D CBF equations \eqref{1} in general unbounded domains also. In that case, one has to take the norm defined on $\V$ space as $\|\u\|_{\V}^2:=\|\u\|_{\L^2}^2+\|\nabla\u\|_{\L^2}^2$.  The Darcy coefficient $\alpha>0$ appearing in \eqref{1} helps us to get the existence of global attractors. For instance, from the inequality \eqref{5p3}, one can deduce that 
	\begin{align*}
	\|\u(t)\|_{\H}^2\leq\|\u_0\|_{\H}^2e^{-\alpha t}+\frac{1}{\mu\alpha}\|\f\|_{\V'}^2, 
	\end{align*}
	for all $t\geq 0$. Thus, we obtain 
		\begin{align*}
	\mathcal{B}_1:=\left\{\v\in\H:\|\v\|_{\H}\leq M_1\equiv \sqrt{\frac{2}{\mu\alpha}}\|\f\|_{\V'}\right\},
	\end{align*}
	is a bounded absorbing set  in $\H$ for the semigroup $\S(t)$. For the asymptotic compactness property, we can define  $\langle\!\langle\cdot,\cdot\rangle\!\rangle:\V\times\V\to\R$ by $$\langle\!\langle\u_1,\u_2\rangle\!\rangle=\mu(\nabla\u_1,\nabla\u_2)+\frac{\alpha}{2}(\u_1,\u_2),\ \text{ for all } \ \u_1,\u_2\in\V.$$ Similar changes can be made in Theorems \ref{thm52} and \ref{thm4.1} also. 
	\end{remark}

\section{Quasi-stability and existence of global as well as exponential attractors}\label{sec8}\setcounter{equation}{0}   In this section, we consider the 2D CBF equations \eqref{kvf} in bounded domains.  We show that the semigroup associated with the system \eqref{kvf} is quasi-stable. Moreover, making use of Theorem 3.4.11, \cite{ICh}, we show the existence of global as well as exponential attractors having finite fractal dimension. Quasi-stability and the existence of global as well as exponential attractors for the 2D Navier-Stokes equations and related hydrodynamic type equations have been established in \cite{ICh}. We point out that our model does not fall in the hydrodynamic type equations described in \cite{ICh} due to the presence of the nonlinear damping term. Quasi-stability results for the three dimensional Kelvin-Voigt fluid flow equations with ``fading memory'' is obtained in \cite{MTM4}. As we are using  Theorems 3.2.3 and 3.4.11, \cite{ICh} to obtain our main goal, the results obtained in this section are true for all $r\in[1,\infty)$.

Let $(\X,\S(t))$ be a dynamical system in some Banach space $\X$ and $\B\subset\X$. Assume that there exists (a) compact seminorms (see Definition 3.1.14, \cite{ICh}) $n_1(\cdot)$ and $n_2(\cdot)$ on the space $\X$, and (b) numbers $a_*,t_*>0$ and $0\leq q<1$  such that 
\begin{align}\label{552} 
\|\S(t)y_1-\S(t)y_2\|_{\X}\leq a_*\|y_1-y_2\|_{\X}, \ \text{ for every }\ y_1,y_2\in\B\ \text{ and }\ t\in[0,t_*],
\end{align}
and 
\begin{align}\label{552a}
\|\P_{t_*}y_1-\P_{t_*}y_2\|_{\X}\leq q\|y_1-y_2\|_{\X}+n_1(y_1-y_2)+n_2(\P_{t_*}y_1-\P_{t_*}y_2),
\end{align}
for every $y_1,y_2\in\B$. Then, under the conditions \eqref{552} and \eqref{552a},  using  Proposition 3.4.10, \cite{ICh}, we infer that the system  $(\X,\S(t))$ is quasi-stable on $\B\subset\X$. We follow Proposition 4.5.9, \cite{ICh} to obtain the following  Ladyzhenskaya squeezing property.

\begin{proposition}[Ladyzhenskaya's squeezing property]\label{prop6.12}
	Let $\Q_m=\mathrm{I}-\P_m$ and $\f\in\H$. Then, for every $0<q<1$, $0<T^*\leq T<\infty$, there exists $m_*=m(T^*,T,r,q)$ such that 
	\begin{align}\label{553}
	\|\Q_m[\S(t)\u-\S(t)\u_*]\|_{\H}\leq q\|\u-\u_*\|_{\H}, \ \text{ for all }\ t\in[T^*,T],\ m\geq m_*,
	\end{align}
	for any $\u$ and $\u_*$ from the set $\mathcal{D}$, where $\mathcal{D}=\left\{\u\in\V: \|\S(t)\u\|_{\V}\leq R,\ \text{ for all } \ t\in[0,T]\right\}$. 
\end{proposition}
\begin{proof}
	Let us define $\w=\u-\u_*$, where $\u(t)=\S(t)\u$ and $\u_*(t)=\S(t)\u_*$. Then $\w(\cdot)$ satisfies the following system:
	\begin{equation}\label{652}
	\left\{
	\begin{aligned}
	\frac{\d}{\d t}\w(t)+\mu\A\w(t)+\alpha\w(t)+\B(\w(t),\u(t))+\B(\u_*(t),\w(t))+\mathcal{C}(\u(t))-\mathcal{C}(\u_*(t))&=\mathbf{0}, \\
	\w(0)&=\w_0,
	\end{aligned}
	\right.
	\end{equation}
in $\H$ for all $t\in(0,T)$, where 	$\w_0=\u-\u_*$. 
	Note  that for $\varphi\in\V$, we have  $\mathrm{P}_m\varphi=\sum_{j=1}^{m}(\varphi,e_j)e_j$, $\A^{\frac{1}{2}}\mathrm{P}_m\varphi=\sum_{j=1}^{m}\lambda_j^{1/2}(\varphi,e_j)e_j$, $\mathrm{Q}_m\varphi=\sum_{j=m+1}^{\infty}(\varphi,e_j)e_j$, $\A^{\frac{1}{2}}\mathrm{Q}_m\varphi=\sum_{j=m+1}^{\infty}\lambda_j^{1/2}(\varphi,e_j)e_j$, 
	\begin{align*}
	\|\A^{\frac{1}{2}}\mathrm{Q}_m\varphi\|_{\H}^2=\sum_{j=m+1}^{\infty}\lambda_j|(\varphi,e_j)|^2\geq \lambda_{m+1}\sum_{j=m+1}^{\infty}|(\varphi,e_j)|^2=\lambda_{m+1}\|\Q_m\varphi\|_{\H}^2,
	\end{align*}
	and 
	\begin{align*}
	\|\A^{\frac{1}{2}}\mathrm{P}_m\varphi\|_{\H}^2=\sum_{j=1}^{m}\lambda_j|(\varphi,e_j)|^2\leq \lambda_{m}\sum_{j=1}^{m}|(\varphi,e_j)|^2=\lambda_{m}\|\P_m\varphi\|_{\H}^2.
	\end{align*}
	That is, we get \begin{align}\label{655}\|\Q_m\varphi\|_{\V}\geq \sqrt{\lambda_{m+1}}\|\Q_m\varphi\|_{\H}\ \text{ and }\ \|\P_m\varphi\|_{\V}\leq \sqrt{\lambda_m}\|\P_m\varphi\|_{\H}.\end{align}
	We define 	$\p(t)=\mathrm{P}_m\w(t)$ and 	$\q(t)=\mathrm{Q}_m\w(t)$.	Let us take the inner product with $\q(t)$ to the first equation in \eqref{652} to obtain 
	\begin{align}\label{657}
	&\frac{1}{2}\frac{\d}{\d t}\|\q(t)\|^2_{\H}+\mu \|\q(t)\|_{\V}^2+\alpha\|\q(t)\|^2_{\H}\nonumber\\&=-b(\w(t),\u(t),\q(t))-b(\u_*(t),\w(t),\q(t))-\langle\mathcal{C}(\u(t))-\beta\mathcal{C}(\u_*(t)),\q(t)\rangle,
	\end{align}
for a.e. $t\in[0,T]$.	
	Using H\"older's, Ladyzhenskaya's  and Young's inequalities, and \eqref{655}, we estimate $b(\w,\u,\q)$ as 
\begin{align*}
|b(\w,\u,\q)|&=|b(\w,\q,\u)|\leq\|\q\|_{\V}\|\w\|_{\wi\L^4}\|\u\|_{\wi\L^4}\leq\frac{\mu}{8}\|\q\|_{\V}^2+\frac{4}{\mu}\|\w\|_{\H}\|\w\|_{\V}\|\u\|_{\H}\|\u\|_{\V}\nonumber\\&\leq\frac{\mu}{8}\|\q\|_{\V}^2+\frac{\e}{4}\|\w\|_{\V}^2+\frac{16}{\mu^2\e}\|\u\|_{\H}^2\|\u\|_{\V}^2\|\w\|_{\H}^2. 
\end{align*}	
Similarly, we estimate $b(\u_*,\w,\q)$ as 
\begin{align*}
|b(\u_*,\w,\q)|&\leq \frac{\mu}{8}\|\q\|_{\V}^2+\frac{\e}{4}\|\w\|_{\V}^2+\frac{16}{\mu^2\e}\|\u_*\|_{\H}^2\|\u_*\|_{\V}^2\|\w\|_{\H}^2. 
\end{align*}
For $r=1$, we know that $-\beta\langle\mathcal{C}(\u)-\mathcal{C}(\u_*),\q\rangle=-\beta\|\Q_m\w\|_{\H}^2$. For $r>1$, using Taylor's formula, H\"older's Gagliardo-Nirenberg's and Young's inequalities, we estimate $-\beta\langle\mathcal{C}(\u)-\mathcal{C}(\u_*),\q\rangle$ as 
\begin{align*}
\beta&|\langle\mathcal{C}(\u)-\mathcal{C}(\u_*),\q\rangle|\nonumber\\&=\beta\left|\left<\int_0^1\mathcal{C}'(\theta\u+(1-\theta)\u_*)\w\d\theta,\q\right>\right|\nonumber\\&\leq\beta r\sup_{0\leq\theta\leq 1}\|(\theta\u+(1-\theta)\u_*)\w\|_{\wi\L^{\frac{r+1}{r}}}\|\q\|_{\wi\L^{r+1}}\nonumber\\&\leq\beta Cr\left(\|\u\|_{\wi\L^{r+1}}+\|\u_*\|_{\wi\L^{r+1}}\right)^{r-1}\|\w\|_{\wi\L^{r+1}}\|\q\|_{\V}\nonumber\\&\leq\frac{\mu}{4}\|\q\|_{\V}^2+\frac{C\beta^2}{\mu}\left(\|\u\|_{\wi\L^{r+1}}+\|\u_*\|_{\wi\L^{r+1}}\right)^{2(r-1)}\|\w\|_{\wi\L^{r+1}}^2\nonumber\\&\leq \frac{\mu}{4}\|\q\|_{\V}^2+\frac{C\beta^2}{\mu}\left(\|\u\|_{\H}^{\frac{2}{r+1}}\|\u\|_{\V}^{\frac{r-1}{r+1}}+\|\u_*\|_{\H}^{\frac{2}{r+1}}\|\u_*\|_{\V}^{\frac{r-1}{r+1}}\right)^{2(r-1)}\|\w\|_{\H}^{\frac{4}{r+1}}\|\w\|_{\V}^{\frac{2(r-1)}{r+1}}.
\end{align*}
Combining the above estimates and substituting it in \eqref{657}, we deduce that 
\begin{align*}
&\frac{\d}{\d t}\|\q(t)\|^2_{\H}+\mu \|\q(t)\|_{\V}^2+\alpha\|\q(t)\|^2_{\H}\nonumber\\&\leq\e\|\w(t)\|_{\V}^2+\frac{32}{\mu^2\e}\left(\|\u(t)\|_{\H}^2\|\u(t)\|_{\V}^2+\|\u_*(t)\|_{\H}^2\|\u_*(t)\|_{\V}^2\right)\|\w(t)\|_{\H}^2\nonumber\\&\quad+\frac{C\beta^2}{\mu}\left(\|\u(t)\|_{\H}^{\frac{2}{r+1}}\|\u(t)\|_{\V}^{\frac{r-1}{r+1}}+\|\u_*(t)\|_{\H}^{\frac{2}{r+1}}\|\u_*(t)\|_{\V}^{\frac{r-1}{r+1}}\right)^{2(r-1)}\|\w(t)\|_{\H}^{\frac{4}{r+1}}\|\w(t)\|_{\V}^{\frac{2(r-1)}{r+1}}, 
\end{align*}
for a.e. $t\in[0,T]$. Since $\u,\u_*\in\mathcal{D}$, we find $\|\u(t)\|_{\V},\|\u_*(t)\|_{\V}\leq R$, for all $t\in[0,T]$, so that from the above inequality, we deduce that 
\begin{align}
&\frac{\d}{\d t}\|\q(t)\|^2_{\H}+\mu\lambda_{m+1} \|\q(t)\|_{\H}^2\leq \e\|\w(t)\|_{\V}^2+\frac{CR^4}{\mu^2\e}\|\w(t)\|_{\H}^2+\frac{C\beta^2R^{2(r-1)}}{\mu}\|\w(t)\|_{\H}^{\frac{4}{r+1}}\|\w(t)\|_{\V}^{\frac{2(r-1)}{r+1}},
\end{align}
where we used \eqref{655} also. A variation of constants formula yields 
\begin{align}
\|\Q_m\w(t)\|_{\H}^2&\leq \|\mathrm{Q}_m\w(0)\|_{\H}^2e^{-\mu\lambda_{m+1}t}+\e\int_0^te^{-\mu\lambda_{m+1}(t-s)}\|\w(s)\|_{\V}^2\d s\nonumber\\&\quad+\frac{CR^4}{\mu^2\e}\int_0^te^{-\mu\lambda_{m+1}(t-s)}\|\w(s)\|_{\H}^2\d s\nonumber\\&\quad+\frac{C\beta^2R^{2(r-1)}}{\mu}\int_0^te^{-\mu\lambda_{m+1}(t-s)}\|\w(s)\|_{\H}^{\frac{4}{r+1}}\|\w(s)\|_{\V}^{\frac{2(r-1)}{r+1}}\d s\nonumber\\&\leq \|\w(0)\|_{\H}^2e^{-\mu\lambda_{m+1}t}+\e\int_0^t\|\w(s)\|_{\V}^2\d s+\frac{CR^4}{\mu^3\e\lambda_{m+1}}\sup_{s\in[0,t]}\|\w(s)\|_{\H}^2\nonumber\\&\quad+\frac{C\beta^2R^{2(r-1)}}{\mu^{\frac{r+3}{r+1}}\lambda_{m+1}^{\frac{2}{r+1}}}\sup_{s\in[0,t]}\|\w(s)\|_{\H}^{\frac{4}{r+1}}\left(\int_0^t\|\w(s)\|_{\V}^2\d s\right)^{\frac{r-1}{r+1}},
\end{align}
for all $t\in[0,T]$. Using \eqref{3p12}, for $r>1$, we deduce that 
\begin{align}
\|\Q_m\w(t)\|_{\H}^2&\leq\left\{e^{-\mu\lambda_{m+1}t}+\left(\e+\frac{CR^4}{\mu^3\e\lambda_{m+1}}+\frac{C\beta^2R^{2(r-1)}}{\mu^2\lambda_{m+1}^{\frac{2}{r+1}}}\right)e^{\frac{4R^2t}{\mu}}\right\}\|\w(0)\|_{\H}^2,
\end{align}
for every $\e>0$. The above inequality easily implies \eqref{553}. The case of $r=1$ is easy. 
\end{proof}
\begin{proposition}[Quasi-stability]\label{prop5.13}
	For every $0<q<1$, $0<a\leq b<\infty$, and a forward invariant  set $\mathcal{B}$, which is a bounded set in $\H$,  there exists $m=m(a,b,q,\mathcal{B})$ such that 
	\begin{align}\label{571}
	\|\S(t)\u-\S(t)\u_*\|_{\H}\leq q	\|\S(\theta)\u-\S(\theta)\u_*\|_{\H}+\|\mathrm{P}_m[\S(t)\u-\S(t)\u_*]\|_{\H}, 
	\end{align}
	for every $t\in[a+\theta,b+\theta]$ and for all $\u,\u_*\in\mathcal{B}$ and $\theta\geq 0$. Furthermore, the system $(\S(t),\H)$ is quasi-stable on $\mathcal{B}\subset\H$. 
\end{proposition}
\begin{proof}
	We know that 
	\begin{align*}
	\|\S(t)\u-\S(t)\u_*\|_{\H}\leq\|\Q_m[\S(t)\u-\S(t)\u_*]\|_{\H}+\|\P_m[\S(t)\u-\S(t)\u_*]\|_{\H}
	\end{align*}
	and hence from \eqref{553}, we obtain 
	\begin{align}\label{650}
	\|\S(t)\u-\S(t)\u_*\|_{\H}\leq q\|\u-\u_*\|_{\H}+\|\P_m[\S(t)\u-\S(t)\u_*]\|_{\H}, 
	\end{align} 
	for every $t\in[a,b]$ and for all $\u,\u_*\in\mathcal{B}$. It is clear that $\|\P_m[\cdot]\|_{\H}$ is a seminorm on $\H$. Thus, the condition \eqref{552a}  holds with $\X=\H$, $n_1\equiv 0$ and $n_2=\|\P_m[\cdot]\|_{\H}$. The Lipschitz property \eqref{552}  follows from \eqref{313a}. Therefore, applying  Proposition 3.4.10, \cite{ICh}, we know that the system $(\S(t),\H)$ is quasi-stable on $\mathcal{B}\subset\H$. Using Exercise 4.3.12, \cite{ICh}, we infer that the relation in \eqref{650} can be written in the uniform form \eqref{571}. 
\end{proof}
Then, based on Theorem 3.4.11, \cite{ICh}, we have the following result. 

\begin{theorem}\label{thm614}
Let $\f\in\H$ and $r\in[\infty)$. 	Then the global attractor $\mathcal{A}_{\mathrm{glob}}$ of the  $(\H,\S(t))$ generated by the dynamical system \eqref{kvf} is a bounded set in $\V$ and possesses the following properties: 
\begin{enumerate}
\item [(1)]  $\mathcal{A}_{\mathrm{glob}}$ has a  finite fractal dimension $\mathrm{dim}_{\mathcal{F}}^{\H}(\mathcal{A}_{\mathrm{glob}})$ in $\H$. 
\item [(2)] For any full trajectory $\{\u(t):t\in\mathbb{R}\}$ from the attractor, $\u(t)$ is an absolutely continuous function with values in $\H$  and 
\begin{align}\label{712}
\sup_{t\in\mathbb{R}}\left\{\|\u_t(t)\|_{\H}+\|\u(t)\|_{\V}+\|\B(\u(t))\|_{\H}+\|\mathcal{C}(\u(t))\|_{\H}\right\}\leq C. 
\end{align}
	\end{enumerate} 
Moreover, the system $(\H,\S(t))$ possesses an exponential attractor $\mathcal{A}_{\mathrm{exp}}$, whose fractal
dimension $\mathrm{dim}_{\mathcal{F}}^{\H}(\mathcal{A}_{\mathrm{exp}})$ is finite in the phase space $\H$. 
\end{theorem}
\begin{proof}
	The existence of a global attractor $\mathcal{A}_{\mathrm{glob}}$ and the basic smoothness is proved Theorem \ref{thm3.4}. 	By Proposition \ref{prop5.13}, it is clear that the system satisfies Assumption 3.4.9, \cite{ICh} (see \eqref{552} and \eqref{552a}) on every forward invariant set, which is bounded in $\H$. Thus, we can apply Theorem 3.4.11, \cite{ICh} to conclude that  the global attractor has finite fractal dimension. 
	
	In order to prove the smoothness of the attractor given in \eqref{712}, we note from \eqref{571} that 
	\begin{align}\label{713}
	\|\u(t+\theta)-\u(t)\|_{\H}\leq q\|\u(t+\theta-1)-\u(t-1)\|_{\H}+\|\mathrm{P}_m[\u(t+\theta)-\u(t)]\|_{\H},
	\end{align}
	for every $t\in\mathbb{R}$ and for full trajectory $\{\u(t):t\in\R\}$ from the attractor. It can be easily seen that 
	\begin{align}\label{714}
	\|\mathrm{P}_m[\u(t+\theta)-\u(t)]\|_{\H}&=\left\|\mathrm{P}_m\left[\int_t^{t+\theta}\frac{\d\u}{\d \tau}(\tau)\d\tau\right]\right\|_{\H}\nonumber\\&\leq \int_t^{t+\theta}\|\mathrm{P}_{m}[-\mu\A\u(\tau)+\B(\u(\tau))+\alpha\u(\tau)+\beta\mathcal{C}(\u(\tau))]\|_{\H}\d\tau\nonumber\\&\leq C(m,\mathcal{A}_{\mathrm{glob}})|\theta|. 
	\end{align}
	Substituting \eqref{714} in \eqref{713}, we deduce that 
	\begin{align}\label{715}
	(1-q)\sup_{t\in\R}\|\u(t+\theta)-\u(t)\|_{\H}\leq C(m,\mathcal{A}_{\mathrm{glob}})|\theta|,
	\end{align}
	and on passing $\theta\to 0$, we finally obtain \eqref{712}. 
	
	In order to prove the existence of the fractal exponential attractors $\mathcal{A}_{\mathrm{exp}}$, we use the second part of Theorem 3.4.11, \cite{ICh}. For this, we need to check the H\"older continuity property $${\|\S(t_1)\u-\S(t_2)\u\|_{\H}\leq C_{\mathcal{B},T} |t_1-t_2|^{\gamma}, \ t_1,t_2\in[0,T],\ \u\in\mathcal{B},}$$ for some $0<\gamma\leq 1$, on some forward invariant absorbing set $\mathcal{B}$ for $(\H,\S(t))$. On an absorbing set $\mathcal{B}=\left\{\u\in\V:\|\u\|_{\V}\leq M\right\}$ in $\V$, we have 
	\begin{align*}
	\|\u_t(t)\|_{\V'}&\leq\mu \|\A\u(t)\|_{\V'}+\|\B(\u(t))\|_{\V'}+\alpha\|\u(t)\|_{\V'}+\beta\|\mathcal{C}(\u(t))\|_{\V'}+\|\f\|_{\V'}\nonumber\\&\leq \mu\|\u(t)\|_{\V}+\sqrt{2}\|\u(t)\|_{\H}\|\u(t)\|_{\V}+C\alpha\|\u(t)\|_{\H}+C\beta\|\u(t)\|_{\H}\|\u(t)\|_{\V}^{r-1}+\|\f\|_{\V'}\nonumber\\&\leq C(\mu,\alpha,\beta,\|\f\|_{\V'},M)<\infty.
	\end{align*}
Thus, using interpolation inequality \eqref{inp}, we obtain 
	\begin{align*}
	\|\S(t_1)\u-\S(t_2)\u\|_{\H}&\leq\|\S(t_1)\u-\S(t_2)\u\|_{\V'}^{1/2}\|\S(t_1)\u-\S(t_2)\u\|_{\V}^{1/2}\nonumber\\&\leq \sqrt{2M}\left\|\int_{t_2}^{t_1}\frac{\d\S(\tau)\u}{\d\tau}\d\tau\right\|_{\V'}^{1/2}\leq\sqrt{2M}\left(\int_{t_2}^{t_1}\left\|\frac{\d\S(\tau)\u}{\d\tau}\right\|_{\V'}\d\tau\right)^{1/2}\nonumber\\&\leq C(\mu,\alpha,\beta,\|\f\|_{\V'},M)|t_1-t_2|^{1/2},
	\end{align*}
	for all $\u\in\mathcal{B}$. From the relation \eqref{313a}, we also get 
	\begin{align*}
	\|\S(t)\u-\S(t)\u_*\|_{\H}\leq e^{\frac{M^2T}{\mu}}\|\u-\u_*\|_{\H},
	\end{align*}
for all $\u,\u_*\in\mathcal{B}$ and $t\in[0,T]$.	Thus the existence of the exponential attractors $\mathcal{A}_{\mathrm{exp}}$ follows by using Theorem 3.4.11, \cite{ICh} (see Theorem3.4.7,  \cite{ICh} also). 
	 The bounds for the fractal dimensions of $\mathcal{A}_{\mathrm{glob}}$  and $\mathcal{A}_{\mathrm{exp}}$ can be derived from Theorems 3.4.11 and 3.2.3, \cite{ICh} (see Theorem \ref{thm52} also).
\end{proof}

\medskip\noindent
{\bf Acknowledgments:} M. T. Mohan would  like to thank the Department of Science and Technology (DST), India for Innovation in Science Pursuit for Inspired Research (INSPIRE) Faculty Award (IFA17-MA110).

\end{document}